\documentclass[a4paper,11pt,reqno]{amsart}

\usepackage{graphicx}
\usepackage{a4wide}
\usepackage{eucal}
\usepackage[pdftex,colorlinks]{hyperref}

\numberwithin{equation}{section}

\newcommand{\R}{\mathbb R}
\newcommand{\C}{\mathbb C}
\newcommand{\N}{\mathbb N}
\newcommand{\Z}{\mathbb Z}

\newcommand{\ff}{\frac{1}{3}}

\newcommand{\cN}{\mathcal N}
\newcommand{\cR}{\mathcal R}

\newcommand{\cG}{\mathcal G}
\newcommand{\be}{\begin{equation}}
\newcommand{\ee}{\end{equation}}
\newcommand{\ba}{\begin{eqnarray}}
\newcommand{\ea}{\end{eqnarray}}
\newcommand{\teta}{\tilde \eta}
\newcommand{\tv}{\tilde v}

\newtheorem{theorem}{Theorem}[section]
\newtheorem{proposition}[theorem]{Proposition}
\newtheorem{remark}[theorem]{Remark}
\newtheorem{lemma}[theorem]{Lemma}
\newtheorem{corollary}[theorem]{Corollary}
\newtheorem{definition}[theorem]{Definition}

\begin{document}

\title{Control of a Boussinesq system of KdV-KdV type on a bounded interval}

\noindent

\author[Capistrano--Filho]{Roberto  A. Capistrano--Filho}
\address{Departamento de Matem\'atica, Universidade Federal de
Pernambuco (UFPE), Recife (PE) 50740-545, Brazil}
\email{capistranofilho@dmat.ufpe.br}
\author[Pazoto]{Ademir F. Pazoto}
\address{Instituto de Matem\'{a}tica, Universidade Federal do Rio de Janeiro, C.P.
68530 - Cidade Universit\'{a}ria - Ilha do Fund\~{a}o, 21941-909 Rio de
Janeiro (RJ), Brazil}
\email{ademir@im.ufrj.br}
\author[Rosier]{Lionel Rosier}
\address{Centre Automatique et Syst\`emes (CAS) and Centre de Robotique (CAOR), MINES ParisTech, PSL Research University, 60 Boulevard Saint-Michel, 75272 Paris Cedex 06, France}
\email{Lionel.Rosier@mines-paristech.fr}

\subjclass[2000]{Primary: 35Q53, Secondary: 37K10, 93B05, 93D15}

\keywords{Boussinesq system; KdV-KdV system; exact controllability; stabilization}

\begin{abstract}
We consider a Boussinesq system of KdV-KdV type introduced by J. Bona, M. Chen and J.-C. Saut as a model for the motion of small amplitude long waves on the surface of an ideal fluid. This system of two equations can describe the propagation of waves in both directions, while the single KdV equation is limited to unidirectional waves. We are concerned here with the exact controllability of the 
Boussinesq system by using some boundary controls. 
By reducing the controllability problem to a spectral problem which is solved by using the Paley-Wiener method introduced by the third author for KdV, we determine explicitly all the critical lengths for which the exact controllability fails for the linearized system, and give a complete picture of the controllability results with one or two boundary controls of Dirichlet or Neumann type. The extension of the exact controllability to the full Boussinesq system  is derived in the energy space in the case of a control of Neumann type. It is obtained by incorporating a boundary feedback in the control in order to ensure a global Kato smoothing effect.   
\end{abstract}
\maketitle

\section{Introduction}
J. L. Boussinesq introduced in \cite{boussinesq1,boussinesq2} several simple nonlinear systems of PDEs  (including the Korteweg-de Vries equation) to explain certain physical observations concerning the water waves, e.g. the emergence and stability of solitons. Unfortunately, several systems derived by Boussinesq proved to be ill-posed, so that there was a need to propose other systems similar to Boussinesq's ones but with better mathematical properties.  

The four-parameter family of Boussinesq systems
\begin{equation}
\left\{
\begin{array}
[c]{l}%
\eta_{t}+v_{x}+(   \eta v)  _{x}+av_{xxx}-b\eta_{xxt}=0\text{,}\\
v_{t}+\eta_{x}+vv_{x}+c\eta_{xxx}-dv_{xxt}=0
\end{array}
\right.  \label{int_29e}%
\end{equation}
was introduced by 
J. J. Bona, M. Chen and J.-C. Saut  in \cite{BCS1}  to describe the motion of small amplitude long
waves on the surface of an ideal fluid under the gravity force and in
situations when the motion is sensibly two-dimensional. In (\ref{int_29e}), $\eta$ is the elevation of the fluid surface from the equilibrium position, and
$v=v_{\theta}$ is the horizontal velocity in the flow at height $\theta h$,
where $h$ is the undisturbed depth of the liquid. The parameters $a$, $b$,
$c$, $d$ are required to
fulfill the relations%
\begin{equation}
a+b=\frac{1}{2}(   \theta^{2}-\frac{1}{3})  \text{, \ \ \ }%
c+d=\frac{1}{2}(   1-\theta^{2})  \geq 0,  \footnote{Note that
$a+b+c+d=\frac{1}{3} \cdot$}
 \label{int_30e}%
\end{equation}
where $\theta\in\left[  0,1\right]  $ specifies which horizontal velocity the
variable $v$ represents. 
As it has been proved in \cite{BCS1}, the initial value problem for the {\em linear system} associated with
\eqref{int_29e} is well posed on $\mathbb R$ if and only if the parameters $a,b,c,d$ fall in one of the 
following cases
\begin{eqnarray*}
(\textrm{C}1) && b,d\ge 0,\ a\le 0,\ c\le 0;\\
(\textrm{C}2)&& b,d\ge 0, \ a=c>0. 
\end{eqnarray*}

The wellposedness of the full system  \eqref{int_29e} on the line $(x\in \R$) was investigated in \cite{BCS2}.

Recently,  a rather complete picture of the control properties
of (\ref{int_29e}) on a periodic domain with a locally supported forcing term
was given in \cite{MORZ}. According to the values of the four parameters $a$, $b$, $c$, $d$,
the linearized system may be either controllable in any positive time, or solely in
large time, or may not be controllable at all. These results were also
extended in \cite{MORZ} to the generic nonlinear system (\ref{int_29e}); that is, 
 when all the parameters are different from $0$.

When $b=d=0$ and $(\textrm{C} 2)$ is satisfied, then necessarily $a=c=1/6$. Nevertheless, the scaling
$x\to x/\sqrt{6}$, $t\to t/\sqrt{6}$ gives a system equivalent to \eqref{int_29e} for which $a=c=1$, namely

\begin{equation}
\left\{
\begin{array}
[c]{l}%
\eta_{t}+v_{x}+(   \eta v)  _{x}+v_{xxx}=0\text{,}\\
v_{t}+\eta_{x}+vv_{x}+\eta_{xxx}=0\text{.}%
\end{array}
\right.  \label{new}%
\end{equation}

The above system will be referred to as a {\em Boussinesq system of KdV-KdV type}, or as a {\em KdV-KdV system}. 

The KdV--KdV system is expected to admit global
solutions on $\mathbb{R}$ \cite{BCS2}, and it also possesses good control properties on
the torus \cite{MORZ}.

The boundary stabilization of \eqref{new} on a bounded domain $(0,L)$ was investigated by two of the authors in \cite{PR}.  
They proved that if system \eqref{new} is supplemented with the following
boundary conditions%
\begin{equation}
\left\{
\begin{array}
[c]{lll}%
v(   t,0)  =v_{xx}(   t,0)  =0\text{ } &  & \text{in
}(   0,T)  \text{,}\\
v_{x}(   t,0)  =\alpha_{0}\eta_{x}(   t,0)  &  & \text{in
}(   0,T)  \text{,}\\
v(   t,L)  =\alpha_{2}\eta(   t,L)  &  & \text{in }( 
0,T)  \text{,}\\
v_{x}(   t,L)  =-\alpha_{1}\eta_{x}(   t,L)  &  & \text{in
}(   0,T)  \text{,}\\
v_{xx}(   t,L)  =-\alpha_{2}\eta_{xx}(   t,L)  &  &
\text{in }(   0,T)  \text{,}%
\end{array}
\right.  \label{int_32e}%
\end{equation}
and the initial conditions%
\begin{equation}%
\begin{array}
[c]{lll}%
\eta(   0,x)  =\eta_{0}(   x)  \text{, \ }
v(0,x)  =v_{0}(   x)  &  & \text{in }(   0,L),
\end{array}
\label{int_33e}%
\end{equation}
then the system is locally exponentially stable in the energy space $[L^2(0,L)]^2$ whenever
the constants $\alpha _0$, $\alpha  _1$ and $\alpha _2$ satisfy
\[\alpha_{0}\geq 0,\  \alpha_{1}>0, \textrm{ and } \alpha _2=1.\]

To our best knowledge, the boundary control of the Boussinesq system of KdV-KdV type  on a bounded domain $(0,L)$  is completely open. The aim of this paper 
is to investigate the control properties of the following system%
\begin{equation}
\left\{
\begin{array}
[c]{lll}%
\eta_{t}+v_{x}+ (\eta v)_x + v_{xxx}=0 &  & \text{in }( 
0,T)  \times(   0,L)  \text{,}\\
v_{t}+\eta_{x} + vv_x + \eta_{xxx}=0 &  & \text{in }(   0,T)
\times(   0,L)  \text{,}%
\end{array}
\right.  \label{int_34e}%
\end{equation}
with the boundary conditions%
\begin{equation}
\left\{
\begin{array}{l}
\eta( t,0)  =h_0(t), \ \eta (t,L)=h_1(t), \ \eta _x(t,0)=h_2(t)\   \text { in }(   0,T),\\
v (t,0)  =g_0( t) ,\  v(   t, L) =g_1(t) ,\  v_{x}( t,L)  =g_{2}( t ) \  \text { in }(   0,T)
\end{array}
\right.  \label{int_35e}%
\end{equation}
and the initial conditions%
\begin{equation}%
\begin{array}
[c]{lll}%
\eta(   0,x)  =\eta_{0}(   x)  \text{, \ }
v( 0,x)  =v_{0}(   x)  &  & \text{in }(   0,L)
\text{.}%
\end{array}
\label{int_36e}%
\end{equation}

It is of course desirable to obtain control results (or stabilization results) with a few number of controls inputs. Here, we will provide
a complete picture of the exact controllability of the linearized system with one or two  controls among $h_0,h_1,h_2,g_0,g_1,g_2$.

A similar study was performed for the Korteweg -de Vries (KdV) equation 
\be
\label{AA1}
y_t+y_{xxx}+y_x+yy_x=0,\\ 
\ee
with the boundary conditions
\be
\label{AA2}
y(t,L)=h_0(t),\ y(t,L)=h_1(t), \ y_x(t,L)=h_2(t). 
\ee
More precisely, the exact controllability of \eqref{AA1}-\eqref{AA2} was established in \cite{R1} when $h_0=h_1=0$ ($h_2$ being the only effective control) for a length $L$ which is not critical,  in \cite{cerpa,CeCr1,CoCr} for a length $L$ which is critical, and  in \cite{GG2} when 
$h_0=h_2=0$ for a length $L$ not critical. The null controllability of \eqref{AA1}-\eqref{AA2} was proved in \cite{R2} (see also \cite{GG1}) when 
$h_1=h_2=0$ for any length $L>0$. Note that in that case there is no critical length, and that solely the null controllability holds (say in $L^2(0,L)$), because 
the terminal state $y(\cdot, T)$ is $C^\infty$ smooth for $x>0$.  

A length $L$ is said to be {\em critical} when the linearized equation 
\be
\label{AA3}
y_t+y_{xxx}+y_x=0,\\ 
\ee
fails to be controllable. This phenomenon was first noticed in \cite{R1}. It is due to the influence of the first order derivative $\partial _x$ on the spectrum of the 
operator $\partial _x^3$ with the boundary conditions $y(0)=y(L)=y_x(L)=0$. If the high frequencies are asymptotically preserved,  the low frequencies may be strongly modified, and some of the corresponding eigenfunctions may become uncontrollable for \eqref{AA3} for certain values of $L$ (the critical ones). 

The set of critical lengths for the linear control system 
\ba
y_t+y_{xxx}+y_x=0,&& \label{AA31}\\
y(t,0)=y(t,L)=0,&&y_x(t,L)=h_2(t) \label{AA32}
\ea
was found \cite{R1} to be 
\[
\cN := \{ 2\pi \sqrt{\frac{k^2+kl+l^2}{3}}; \ k,l\in \N ^*\}.  
\]

On the other hand, the set of critical lengths for the linear control system 
\ba
y_t+y_{xxx}+y_x=0,&& \label{AA33}\\
y(t,0)=y_x(t,L)=0,&&y(t,L)=h_1(t) \label{AA34}
\ea
was  proved in \cite{GG2} to be discrete, countable and given by 
\[
\tilde \cN 
: =\{ L>0; \exists a\in\C, b\in \C \textrm{ with } ae^a = be^b = -(a+b)e^{-a-b}, \ L^2=-(a^2+ab+b^2)  \} . 
\]

The determination of the critical lengths for system \eqref{AA31}-\eqref{AA32} in \cite{R1} was based on a series of reductions (as in \cite{BLR}), first to a unique continuation property for the adjoint system, 
next to a spectral problem with an extra condition:
\be
\label{AA100}
-y'''-y'=\lambda y,\  y(0)=y(L)=y'(0)=0 \ \textrm{ and }\  y'(L)=0.  
\ee
This spectral problem was then solved by extending the function $y$ by 0 outside $(0,L)$,  by taking its Fourier transform and 
by using Paley-Wiener theorem. 
The length $L$ is then critical if and only if there exist  some numbers  $p\in \C$ and $(\alpha, \beta )\in \C^2\setminus \{ (0,0)\}$
such that  the function 
\be
\label{AA105}
f(\xi ) = \frac{\alpha -\beta e^{-iL\xi}}{\xi ^3 -\xi +p}
\ee 
is analytic on $\C$; that is, the roots of $\xi ^3 -\xi + p$ are also roots of $\alpha -\beta e^{-iL\xi}$
with at least the same multiplicity. The determination of $\mathcal N$  
follows then with  some algebra.

By contrast, the critical lengths for \eqref{AA33}-\eqref{AA34}, that is the elements of $\tilde {\mathcal N}$, are not explicitly known. 
This is likely due to the lack of symmetries 
in the corresponding spectral problem
\be
\label{AA110}
-y'''-y'=\lambda y,\  y(0)=y(L)=y'(0)=0 \ \textrm{ and }\  y''(L)=0.  
\ee
Note that the boundary conditions in  
\eqref{AA100} are preserved by the transformation $x\to L-x$.  This symmetry yields a more tractable function $f$ in \eqref{AA105}. 

Boussinesq system is more convenient than KdV as a model for the propagation of water waves, for it is adapted to the 
propagation of waves in {\em both}  directions, and it is still valid after bounces of waves at the boundary. 
It is striking that the control theory for Boussinesq system (exposed in this paper) 
 is better understood than for KdV as far as the critical lengths are concerned:  indeed, the critical lengths for Boussinesq system are 
{\em explicitly given} for any set of boundary controls (except in Case 3 where only $g_0$ is used, and in Case 12 where $h_1$ and $g_0$ are used,
see below Table 1), which is not 
the case for KdV with a control as in \eqref{AA34}. We believe that this is due to the numerous symmetries of Boussinesq system: for instance, 
$x=0$ and $x=L$ (resp. $\eta$ and $v$) play a symmetric role for the linearized Boussinesq system.
The price to be paid is the lack of any Kato smoothing effect (the system being conservative), which makes the extension of the control results to the nonlinear Boussinesq system more delicate than for KdV. We refer the reader to \cite{CFPR,cerpa, CeCr1,CeCr2,CCS,coron-book,CoCr,DN,
GG1,GG2,LRZ,LP,pazoto,PMVZ,R1,R2,RZ,RuZ} for the control and stabilization of KdV, and \cite{dSV,DL,dSdSV} for the critical lengths concerning some other dispersive equations.   

To investigate the control properties of the linearized system, we proceed as in \cite{R1}. To prove the observability inequality, we use the reduction to a spectral problem with an extra condition and the Paley-Wiener method. Here, we have to see for which value of $L>0$ two 
functions are entire for a set of parameters.  For instance, in Case 1 where only $g_2$
is used, the two functions read   
\begin{eqnarray*}
f(\xi)    &=& \frac{i}{\xi ^3 -\xi +p}(\alpha +\beta i\xi +\gamma e^{-iL\xi})\\
g( \xi  ) &=& \frac{i}{\xi ^3 -\xi +p} (\alpha ' -\beta i\xi + \gamma' e^{iL\xi}) 
\end{eqnarray*}  
for some $p\in \C$ and $(\alpha,\beta, \gamma,\alpha ', \gamma' )\in \C ^5 \setminus \{ 0 \}$. Using the symmetries of the system, we can prove that the 
set of critical lengths is $\mathcal N$. 

In the reduction to the Unique Continuation Property (UCP), we use in most cases an identity obtained by using Morawetz multiplier, namely  \eqref{A200} (see below). 
However, for Case 3 (see below Table 1), the identity \eqref{A200} proved to be useless. We replace  the observability inequality by a weaker estimate
\be 
\Vert (\eta ^0, v^0)\Vert ^2_{[H^2(0,L)]^2 } \le 
C\left(  \int_0^T |\eta _{xx} (t,L)|^2 dt + \Vert (\eta ^0, v^0)\Vert ^2 _{[L^2(0,L)]^2} \right) , 
\label{TYU}
\ee
which in turn is established by performing a careful investigation of the spectral reduction of the {\em skewadjoint} operator 
$A(\eta, v)=(-v_x -v_{xxx}, -\eta _x  -\eta _{xxx})$ with domain 
\[ 
D(A)= \{ (\eta , v)\in [H^3(0,L)\cap H^1_0(0,L) ] ^2; \ \ \eta _x(0)=v_x(L)=0\}  \subset [L^2(0,L)]^2. 
\]
We show that the space $[L^2(0,L)]^2$ admits an orthonormal basis constituted of eigenfunctions of $A$. To do that, we introduce the 
{\em selfadjoint} operator $(By)(x) = -y_{xxx}(L-x) -y_x(L-x)$ with domain 
\[
D(B) = \{ y\in H^3(0,L) \cap H^1_0 (0,L);\  y_x (L)=0 \} \subset L^2 (0,L). 
\]
Then the spectral reduction of $B$ yields at once those for $A$. 

We give in Table 1 (see below) a complete picture of the exact controllability results for the linearized Boussinesq system.  It is not difficult to extend those exact controllability results to the nonlinear Boussinesq system
in spaces of sufficiently regular functions (e.g. some subspaces of  $[H^2(0,L)]^2$), but here we will not do it for the sake of shortness. 
Rather, we shall explain how to obtain exact controllability/stabilization results in the {\em energy space} $[L^2(0,L)]^2$ 
by  incorporating a feedback law in the control that yields a smoothing effect, as it was done in \cite{LR} for the Benjamin-Ono equation.  We shall limit ourserves to the case of a single Neumann boundary control (Case 1), namely to the system
\be
\label{INTRO1}
\left\{ 
\begin{array}{ll}
\eta _t + v_x + (\eta v)_x + v_{xxx} =0, \quad &  t\in (0,T) , \ x\in (0,L), \\
v_t + \eta _x + vv_x + \eta _{xxx} =0, &  t\in (0,T) , \ x\in (0,L), \\
\eta (t,0)=0, \ \eta (t,L)=0,  \ \eta _x(t,0)=0, &t\in (0,T), \\
v(t,0)=0, \ v(t,L)=0, \ v_x(t,L)=g_2(t), &t\in (0,T) , \\
\eta (0,x)=\eta ^0(x),\ v(0,x)=v^0(x), &x\in (0,L), 
\end{array}\right.
\ee
The first main result in this paper is a local exact controllability result for  \eqref{INTRO1}.
\begin{theorem}
\label{THMINTRO1}
Let $T>0$ and $L\in (0,+\infty )\setminus {\mathcal N}$ where 
$
{\mathcal N} := \{ \frac{2\pi}{\sqrt{3}} \sqrt{k^2+kl+l^2}; \quad k,l\in \N ^*\}.
$
Then there exists some $\delta >0$ such that for all states
$(\eta ^0, v^0), (\eta ^1,v^1)\in [L^2(0,L)]^2$ with
\[
\Vert (\eta ^0,v^0)\Vert _{[L^2(0,L)]^2} \le \delta \quad \textrm{ and }\quad  \Vert (\eta ^1, v^1)\Vert _{[L^2(0,L)]^2} \le \delta ,
\]   
one can find a control $g_2\in L^2(0,T)$ and a solution $(\eta , v)\in C([0,T], [L^2(0,L)]^2)\cap 
L^2(0,T, [H^1(0,L)]^2)$ of \eqref{INTRO1} such that 
\[
\eta (T,x) = \eta ^1(x),\  v(T,x) = v^1(x), \qquad x\in (0,L).
\]
\end{theorem}
The second main result is a local exponential stability result for \eqref{INTRO1}. 
\begin{theorem}
\label{THMINTRO2}
Let $T>0$, $L\in (0,+\infty )\setminus {\mathcal N}$, and  $\alpha >0$. Then there exist some positive numbers $\delta , \mu,C$ such that for all initial data
$(\eta ^0, v^0) \in [L^2(0,L)]^2$ with
$\Vert (\eta ^0,v^0)\Vert _{[L^2(0,L)]^2}  \le \delta $,  
the system \eqref{INTRO1} with the boundary feedback law
$g_2(t) = -\alpha \eta _x(t,L)$  admits for all $T>0$ a unique solution 
\[
(\eta , v)\in C([0,T], [L^2(0,L)]^2) \cap L^2(0,T,[H^1(0,L)]^2), 
\]
and it holds
\begin{eqnarray}
\Vert (\eta , v) (t)\Vert _{[L^2(0,L)]^2} 
&\le& Ce^{-\mu t} \Vert (\eta _0, v_0 )\Vert _{[L^2(0,L)]^2} , \quad \forall t\ge 0, \\
\Vert (\eta , v) (t)\Vert _{[H^1(0,L)]^2} 
&\le& C \frac{e^{-\mu t}}{\sqrt{t}}  \Vert (\eta _0, v_0 )\Vert _{[L^2(0,L)]^2} , \quad \forall t> 0.
\end{eqnarray}
\end{theorem}

The paper is outlined as follows. The wellposedness of the linearized Boussinesq system with the boundary conditions \eqref{int_35e} 
is studied in Section 2. Section 3 is concerned with the exact controllability  of the linearized Boussinesq system with one or two
boundary control inputs. The proof of the main results, namely Theorems  \ref{THMINTRO1} and \ref{THMINTRO2}, is provided
in Section 4. Finally, the proof of the  weak observability estimate  \eqref{TYU} based on some spectral reduction is given in appendix.   

\section {Wellposedness}
\subsection{Wellposedness of the homogeneous problem}
Let $L>0$ be a fixed number. Introduce the spaces 
\ba
X_0&:=& [ L^2(0,L) ]^2 = L^2(0,L)\times L^2(0,L) \label{MMMM1}\\
X_3&:=& \{ (\eta, v)\in [H^3(0,L)\cap H^1_0(0,L)]^2; \ \eta _x(0)=v_x(L)=0 \} , \label{MMMM2}\\
X_{3\theta } &:=& [X_0,X_3]_{ [\theta ]}, \quad \textrm{ for } 0<\theta <1, \label{MMMM3}
\ea
where $[X_0,X_3]_{[\theta ]}$ denotes the Banach space obtained by the  complex  interpolation method 
(see e.g. \cite{BL}). The space  $X_0$ (resp. $X_3$) is endowed with the norm 
\[
\Vert (\eta , v)\Vert _{X_0}  :=\left( \int_0^L [\eta ^2 (x) + v^2 (x) ]dx\right) ^\frac{1}{2}
\]
(resp. with the norm $\Vert (\eta , v) \Vert _{X_3} :=\Vert (\eta , v)\Vert _{X_0} + \Vert (v_x+v_{xxx}, \eta _x + \eta _{xxx})\Vert _{X_0}$). 

It is easily seen that 
\ba
X_1 &=& H^1_0(0,L)\times H^1_0(0,L), \\
X_2 &=&  \{ (\eta, v) \in [ H^2(0,L)\cap H^1_0(0,L) ] ^2; \ \eta _x(0)=v_x(L)=0 \}
\ea
and that in the space $X_1$ (resp. $X_2$), 
the norm $\Vert (\eta, v)\Vert _{X_1}$ (resp. $\Vert (\eta , v)\Vert _{X_2}$)  is equivalent to 
$\displaystyle \left( \int_0^L [\eta _x^2 (x) + v _x^2 (x) ]dx\right) ^\frac{1}{2}$
(resp. $\displaystyle \left( \int_0^L [\eta _{xx}^2 (x) + v_{xx}^2 (x) ]dx\right) ^\frac{1}{2}$). 
We shall use at some place the space 
\begin{eqnarray*}
X_4 &:=& \{ (\eta , v) \in [H^4(0,L)\cap H^1_0(0,L)]^2;\   \eta _x(0)=v_x(L)=\eta _{xxx}(0)=v_{xxx}(L)=0, \\ 
&&\qquad\qquad \qquad \qquad\qquad \qquad \qquad \ \   \eta _{xxx } (L) + \eta _x (L) =v_{xxx} (0)+v_x(0)=0 \} 
\end{eqnarray*}
endowed with its natural norm, 
and for $s\in \{1,2\}$,  the space $X_{-s}=(X_s)'$ which is the dual of $X_s$ with respect to the pivot space $X_0=[L^2(0,L)]^2$. Note that 
\[
X_{-1}=H^{-1}(0,L)\times H^{-1}(0,L). 
\] 
The bracket $\langle \cdot ,\cdot  \rangle _{X_{-s},X_s}$ stands for the duality bracket.  
We first investigate the wellposedness of the initial value problem
\be
\label{A1}
\left\{ 
\begin{array}{ll}
\eta _t + v_x + v_{xxx} =0, \quad &  t\ge 0, \ x\in (0,L), \\
v_t + \eta _x + \eta _{xxx} =0, &  t\ge 0, \ x\in (0,L),\\
\eta (t,0)=\eta (t,L)=\eta _x(t,0)=0, &t\ge 0, \\
v(t,0)= v(t,L)=v_x(t,L)=0, &t\ge 0, \\
\eta (0,x)=\eta ^0(x),\ v(0,x)=v^0(x), &x\in (0,L).  
\end{array}\right.
\ee
We introduce the operator
\[
A(\eta , v) :=(-v_x-v_{xxx}, -\eta _x-\eta _{xxx})
\]
with domain $D(A):=X_3\subset X_0$. 
\begin{proposition}
\label{prop1}
The operator $A$ is skew-adjoint in $X_0$, and thus it generates a group of isometries $(e^{tA})_{t\in \R}$ in $X_0$.  
\end{proposition}
\begin{proof}
We have to prove that $A^*=-A$. It is clear that we have $-A\subset A^*$ (i.e. $(\theta , u)\in D(A^*)$ and  
$A^*(\theta , u) =-A(\theta ,u)$ for all $(\theta ,u)\in D(A)$). Indeed, 
for any $(\eta , v),(\theta , u)\in D(A)$, we have after some integrations by parts
\begin{eqnarray*}
((\theta , u) ,A(\eta ,v))_{X_0} 
&=& -\int_0^L  [\theta (v_x+v_{xxx})+u(\eta _x + \eta _{xxx}) ]dx \\
&=& \int_0^L [v(\theta _x + \theta _{xxx}) +\eta (u_x+u_{xxx}) ]dx \\ 
&=& -(A(\theta, u), (\eta , v))_{X_0}.  
\end{eqnarray*}
Let us prove now that $A^*\subset -A$. Pick any $(\theta , u ) \in D(A^*)$. Then,  we have for some constant 
$C>0$ 
\[
\left\vert ( (\theta , u ) , A(\eta , v))_{X_0}\right\vert \le C \Vert (\eta , v)\Vert _{X_0}\quad \forall (\eta , v) \in D(A), 
\] 
i.e.
\be
\label{A2}
\left\vert \int _0^L  [ \theta (v_x+v_{xxx} ) + u (\eta _x + \eta _{xxx}) ]dx \right\vert 
  \le C \left(  \int_0^L [\eta ^2 + v^2 ]dx\right)^\frac{1}{2},
\qquad \forall (\eta , v)\in D(A). 
\ee
Picking $v=0$ and $\eta \in C_c^\infty (0,L)$, we infer from \eqref{A2} that $u_x+u_{xxx}\in L^2(0,L)$, and hence that 
$u\in H^3(0,L)$. Similarly, we obtain that $\theta \in H^3(0,L)$. 
Integrating by parts in the left hand side  of \eqref{A2}, we obtain that 
\begin{eqnarray*}
&&\left\vert \theta (L) v_{xx}(L) -\theta (0)v_{xx}(0) + \theta _x (0)v_x(0) 
+ u(L) \eta _{xx}(L)  -u(0) \eta _{xx} (0) -u_x(L) \eta _x(L) \right\vert \\
&& \qquad \le C\left( \int_0^L [\eta ^2 + v^2] dx \right) ,\qquad \forall (\eta , v) \in D(A) . 
\end{eqnarray*}
It easily follows that 
\[
\theta (0)=\theta (L)=\theta _x(0)=u(0)=u(L)=u_x(L)=0,
\]
so that $(\theta , u)\in D(A)=D(-A)$. Thus $D(A^*)=D(-A)$ and $A^*=-A$.
\end{proof}

\begin{corollary}
\label{cor1} For any $(\eta ^0,v^0) \in X_0$, system \eqref{A1} admits a unique solution 
$(\eta ,v)\in C(\R , X_0)$, which satisfies $\Vert (\eta (t) , v (t) )\Vert _{X_0} =\Vert (\eta ^0, v^0)\Vert _{X_0}$ for all $t\in \R$. 
If, in addition, $(\eta ^0, v^0)\in X_3$, then  $(\eta, v)\in C(\R , X_3)$ with 
$\Vert (\eta, v)\Vert _{X_3} := \Vert (\eta, v)\Vert _{X_0} + \Vert  A (\eta, v)\Vert _{X_0}$ constant. 
\end{corollary}

Using  Corollary \ref{cor1} combined with some interpolation argument between $X_0$ and $X_3$, 
we infer that for any $s\in (0,3)$, there exists a constant
$C_s>0$ such that for any $(\eta ^0,v^0)\in X_s$, the solution $(\eta , v)$ of \eqref{A1} satisfies 
$(\eta , v) \in C(\R , X_s)$ and 
\be
\label{A5}
\Vert (\eta (t) , v(t)) \Vert _{X_s} \le C_s \Vert (\eta ^0, v^0 ) \Vert _{X_s}, \quad \forall t\in \R.  
\ee 

\subsection{Existence of traces}
For the solutions of system \eqref{A1}, we know that the traces \[ \eta (\cdot ,0), \eta (\cdot , L), \eta _x(\cdot ,0), v(\cdot,0), v(\cdot , L), \textrm{ and } v_x(\cdot ,L)\]
 vanish. 
We have a look at the other traces $\eta _x(\cdot  ,L), v_x(\cdot ,0)$ and prove that they belong to  $L^2_{loc}(\R  _+)$ when 
$ (\eta ^0, v^0)\in X_1$. 
\begin{proposition}
\label{prop2}
Let $(\eta ^0,v^0)\in X_1$ and let $(\eta ,v)$ denote the solution of \eqref{A1}. Pick any $T>0$. Then 
$\eta _x(\cdot ,L), v_x(\cdot ,0) \in L^2(0,T)$ with 
\be
\label{A20}
\int_0^T [ |\eta _x (t,L) |^2 + |v_x(t,0) |^2]dt \le C \Vert (\eta ^0, v^0)\Vert ^2_{X_1}  
\ee 
for some constant $C=C(T)$. 
\end{proposition} 
\begin{proof}
Assume that $(\eta ^0, v^0) \in X_3$, so that $(\eta , v) \in C([0,T], X_3)\cap C^1([0,T], X_0)$. 
We use Morawetz multipliers as in \cite{R1}. We multiply the first (resp. second) equation in \eqref{A1} by 
$xv$ (resp. $x\eta$), integrate by parts over $(0,T)\times (0,L)$, and add the two obtained equations to obtain
\be
\label{A200}
\frac{3}{2} \int_0^T \!\!\! \int_0^L [\eta _x^2 + v_x^2] \, dxdt  
-\frac{1}{2} \int_0^T \!\!\! \int_0^L [\eta ^2 + v^2] \, dxdt + \left[\int_0^L [x\eta v] dx \right] _0^T 
-\frac{L}{2} \int_0^T \eta _x^2(t,L)  \, dt =0. 
\ee
Since $\int_0^T\Vert (\eta  ,v)\Vert ^2_{X_1} dt \le C\Vert (\eta ^0, v^0)\Vert ^2_{X_1}$, this yields 
\[
\int_0^T \eta _x^2 (t,L)\, dt \le C \Vert (\eta ^0, v^0)\Vert ^2 _{X_1}. 
\] 
By symmetry, using now as multipliers $(L-x)v$ and $(L-x)\eta$, we infer that \
\[
\int_0^T v_x^2 (t,0)\, dt \le C \Vert (\eta ^0, v^0)\Vert ^2 _{X_1}. 
\] 
Thus \eqref{A20} is established when $(\eta ^0, v^0)\in X_3$. Since 
$X_1$  is dense in $X_3$, the result holds as well for $(\eta ^0, v^0)\in X_1$.  
\end{proof}

We now turn our attention to the traces of order 2, namely 
$\eta _{xx} ( \cdot ,0)$, $\eta _{xx}(\cdot ,L)$, $v_{xx}(\cdot ,0)$, and $v_{xx}(\cdot ,L)$. 

\begin{proposition}
\label{prop3}
Let $(\eta ^0,v^0)\in X_2$ and let $(\eta ,v)$ denote the solution of \eqref{A1}. Pick any $T>0$. Then 
$\eta _{xx} (\cdot ,0), \eta _{xx}(\cdot ,L), v_{xx}(\cdot ,0), v_{xx}(\cdot ,L)  
 \in L^2(0,T)$ with 
\be
\label{A41}
\int_0^T [ |\eta _{xx} (t,0) |^2 +|\eta _{xx} (t,L)|^2 
+ |v_{xx}(t,0) |^2  + |v_{xx}(t,L) |^2  ]dt \le C \Vert (\eta ^0, v^0)\Vert ^2_{X_2}  
\ee 
for some constant $C=C(T)$. Furthermore, $\eta _{x} (\cdot , L),  v_{x}(\cdot , 0)   \in H^\frac{1}{3} (0,T)$ with  
 \be
 \label{A42}
\Vert \eta _{x} (\cdot , L) \Vert ^2 _{H^\frac{1}{3} (0,T)} + \Vert v_{x}(\cdot ,0) \Vert ^2_{H^\frac{1}{3} (0,T) }
 \le C \Vert (\eta ^0, v^0)\Vert ^2_{X_2}  
\ee 
for some constant $C=C(T)$.
\end{proposition} 
\begin{proof}
Assume that $(\eta ^0, v^0) \in X_3$, so that $(\eta , v) \in C([0,T], X_3)\cap C^1([0,T], X_0)$. 
Pick $\rho \in C^\infty ( [0,L] )$ with $\rho (x)=0$ for $x\le 1/4$ and $\rho (x)=1$ for $x\ge 3/4$, and set 
$\tilde \eta = \rho \eta$, $\tilde v=\rho v$. Then we have
\ba
\teta _t +\tv _x + \tv _{xxx} &=& \rho _x v+ 3\rho _x v_{xx} + 3\rho _{xx} v_x + \rho _{xxx} v =:\tilde f, \label{A31}\\
 \tv _t +\teta _x + \teta _{xxx} &=& \rho _x \eta + 3\rho _x \eta _{xx} + 3\rho _{xx} \eta _x + \rho _{xxx} \eta
 =: \tilde g .\label{A32}
\ea  
Note that $\tilde f, \tilde g\in C([0,T], L^2(0,L) )$. 
Multiplying each term in \eqref{A31} by $\tv _{xx}$ (resp.  in \eqref{A32}  by $\teta _{xx}$) and integrating 
by parts, we arrive to 
\[
\frac{1}{2} \int_0^T \big[\teta _{xx}^2 (t,L)  + \tv _{xx}^2 (t,L) + \teta _{x}^2(t, L) ] \, dt  =
\left[ \int_0^L \teta _{x} \tv _{x} \, dx  \right]_0^T 
+ \int_0^T\!\!\!\int_0^L [ \tilde f \tv_{xx} + \tilde g \teta _{xx} ]\, dxdt. 
\]
It follows that 
\[
\int_0^T \big[\eta _{xx}^2 (t,L)  + v_{xx}^2 (t,L) ] \, dt  \le C \left( \Vert (\eta ^0, v^0)\Vert ^2_{X_2} 
+ \int_0^T \Vert (f , g )\Vert _{X_0}^2 dt \right ) \le  C  \Vert (\eta ^0, v^0)\Vert ^2_{X_2}, 
\]
and we can prove in a similar way that
 \[
\int_0^T \big[ \eta _{xx}^2 (t,0)  + v_{xx}^2 (t,0) ] \, dt 
\le C  \Vert (\eta ^0, v^0)\Vert ^2_{X_2}. 
\]
Thus \eqref{A41} is established when $(\eta ^0, v^0)\in X_3$. Since 
$X_3$  is dense in $X_2$, the result holds as well for $(\eta ^0, v^0)\in X_2$.
Let us proceed with the proof of \eqref{A42}.   
If $(\eta ^0, v^0)\in X_4$, then $(\eta, v)\in C(\R, X_4)$, so that $(\hat \eta , \hat v) :=(\eta _t, v_t)=A(\eta , v)$ is in 
$C(\R , X_1)$ and it solves 
\[
(\hat \eta , \hat v) _t = A(\hat \eta , \hat v),\quad (\hat \eta , \hat v)(0) = A(\eta ^0, v^0).  
\]  
It follows from \eqref{A20} that 
\be
\label{A20a}
\Vert \eta _x (\cdot ,  L)\Vert ^2 _{H^1(0,T)} + \Vert v_x (\cdot , 0)\Vert ^2 _{H^1(0,T)}
\le C \Vert (\eta ^0, v^0)\Vert ^2 _{X_4}.
\ee
Since $X_2=[X_1,X_4]_\frac{1}{3}$, we infer from \eqref{A20} and \eqref{A20a} that 
\be
\label{A20b}
\Vert \eta _x (\cdot ,  L)\Vert ^2 _{H^\frac{1}{3} (0,T)} + \Vert v_x (\cdot , 0)\Vert ^2 _{H^\frac{1}{3} (0,T)}
\le C \Vert (\eta ^0, v^0)\Vert ^2 _{X_2}
\ee
for some constant $C=C(T)$ and all  $(\eta ^0, v^0)\in X_2$.
\end{proof}

\subsection{Wellposedness of the nonhomogeneous problem}
We consider the nonhomogeneous system 
\be
\label{A50}
\left\{ 
\begin{array}{ll}
\eta _t + v_x + v_{xxx} =0, \quad &  t\in (0,T) , \ x\in (0,L), \\
v_t + \eta _x + \eta _{xxx} =0, &  t\in (0,T) , \ x\in (0,L),\\
\eta (t,0)= h_0(t), \ \eta (t,L)= h_1(t) ,  \ \eta _x(t,0)= h_2(t) , &t\in (0,T), \\
v(t,0)= g_0(t), \ v(t,L)= g_1(t), \ v_x(t,L)=g_2(t), &t\in (0,T) , \\
\eta (0,x)=\eta ^0(x),\ v(0,x)=v^0(x), &x\in (0,L), 
\end{array}\right.
\ee
where the initial data $(\eta ^0, v^0)$ and the boundary data $(h_0,h_1,h_2, g_0, g_1, g_2)$ are given in appropriate spaces. 

We also consider the homogenous system 
\be
\label{A51}
\left\{ 
\begin{array}{ll}
\theta _t +u_x + u_{xxx} =0, \quad &  t\in(0,T), \ x\in (0,L), \\
u_t + \theta _x + \theta _{xxx} =0, &  t\in (0,T), \ x\in (0,L),\\
\theta (t,0)=\theta (t,L)=\theta _x(t,0)=0, &t\in (0,T), \\
u(t,0)= u(t,L)=u_x(t,L)=0, &t\in (0,T), \\
\theta (0,x)=\theta ^0(x),\ u(0,x)=u^0(x), &x\in (0,L).  
\end{array}\right.
\ee

If both $(\eta, v)$ and $(\theta , u)$ are in $C([0,T],[H^3(0,L)]^2)\cap C^1([0,T], [L^2(0,L)]^2)$, which is the case when 
$(\eta ^0, v^0),(\theta ^0, u^0) \in X_3$ and $h_i, g_i\in C^2([0,T])$ with $h_i(0)=g_i(0)=0$  for $i=1,2,3$, then we obtain 
after multiplying the two first  equations of \eqref{A50} by $u$ and $\theta$ respectively, and integrating by parts, that for 
all $S\in [0,T]$
\ba
\left[\int_0^L [\eta\theta +vu]dx \right]_0^S 
&=& \int_0^S [ \theta _x (t,L) g_2(t) -\theta _{xx} (t,L) g_1(t) +\theta_{xx} (t,0) g_0(t) \nonumber \\
 &&\qquad    - u_x(t,0)h_2(t)   -u_{xx} (t,L) h_1(t)  +  u_{xx}(t,0) h_0(t)] dt. \label{A60}
\ea    
\begin{definition}
\label{def1}
Let $(\eta ^0, v^0)\in X_{-2}$, $g_0,g_1,h_0,h_1\in L^2(0,T)$ and $g_2,h_2\in H^{-\frac{1}{3}}(0,T)$. We 
say that $(\eta, v)\in C([0,T], X_{-2})$ is a  solution of the nonhomogeneous problem \eqref{A50} if we have that 
\be
\label{A62}
\langle (\eta (S), v(S)), (\theta (S), u(S))\rangle _{X_{-2},X_2} = L_S(\theta ^0, u^0), \qquad \forall (\theta ^0,u^0)\in X_2, \ \forall S\in [0,T],
\ee
where 
\ba
 L_S(\theta ^0, u^0) &:=&
\langle (\eta ^0, v^0), (\theta ^0, u^0)\rangle _{X_{-2},X_2}
+\langle g_2,  {\mathbf 1} _{(0,S)} \theta _x(\cdot  ,L) \rangle _{H^{-\frac{1}{3}} (0,T), H^\frac{1}{3}(0,T)}
\nonumber \\ 
&&+ \int_0^S [ -\theta _{xx} (t,L) g_1(t) +\theta_{xx} (t,0) g_0(t)  -u_{xx} (t,L) h_1(t)  +  u_{xx}(t,0) h_0(t)] dt \nonumber \\
&& - \langle h_2,  {\mathbf 1} _{(0,S)} u_x(\cdot , 0) \rangle _{H^{-\frac{1}{3}} (0,T), H^\frac{1}{3}(0,T)},
\label{A63}
\ea
$(\theta ,u)$ denoting the solution of system \eqref{A51}. 
\end{definition}
Note that $L_S:X_2\to \R$ is well defined for all $S\in [0,T]$, for 
$\theta _{xx} (\cdot , 0)  , \theta_{xx} (\cdot , L) , u_{xx} (\cdot , 0)$,  $u_{xx}(\cdot , L) \in L^2(0,T)$ and 
$\theta _x(\cdot , L), u_x (\cdot , 0)\in H^\frac{1}{3}(0,T)$ by Proposition \ref{prop3}. The fact that 
${\mathbf 1}_{(0,S)} \theta _x(\cdot , L)$, ${\mathbf 1}_{(0,S)}  u_x (\cdot , 0)\in H^\frac{1}{3}(0,T)$ for any $S\in [0,T]$ follows from 
\cite[Th\'eor\`eme 11.4 p. 66]{LM}. 

The existence and uniqueness of a solution of system \eqref{A50} is stated in the following result. 

\begin{proposition}
\label{prop4}
Let $(\eta ^0, v^0)\in X_{-2}$, $g_0,g_1,h_0,h_1\in L^2(0,T)$, and $g_2,h_2\in H^{-\frac{1}{3}}(0,T)$. Then there exists a unique 
solution  $(\eta, v)\in C([0,T], X_{-2})$ of the nonhomogeneous problem \eqref{A50}. Furthermore, we have 
\ba
\Vert (\eta , v)\Vert _{L^\infty (0,T,X_{-2})} &\le& C \bigg(
\Vert (\eta ^0, v^0)\Vert _{X_{-2}} +  \Vert g_0\Vert _{L^2(0,T)} +\Vert h_0\Vert _{L^2(0,T)} \nonumber\\
&& +\Vert g_1\Vert _{L^2(0,T)} +\Vert h_1\Vert _{L^2(0,T)}  
+  \Vert g_2\Vert _{H^{-\frac{1}{3}} (0,T)} +\Vert h_2\Vert _{H^{-\frac{1}{3}} (0,T)} 
\bigg)  \quad \label{A71}
\ea 
for some constant $C=C(T) >0$.
If, in addition, $(\eta ^0, v^0)\in X_{-1}$,  
$g_0=g_1=h_0=h_1=0$,
and $g_2,h_2\in L^2(0,T)$, then 
$(\eta, v)\in C([0,T], X_{-1})$ and we have for some
constant $C'=C'(T)$
\ba
\Vert (\eta , v)\Vert _{L^\infty (0,T,X_{-1} ) } 
&\le& C' \bigg(
\Vert (\eta ^0, v^0)\Vert _{X_{-1}} +  \Vert g_2\Vert _{L^2 (0,T)} +\Vert h_2\Vert _{L^2 (0,T)} 
\bigg) . \label{A72}
\ea 
\end{proposition}
\begin{proof}
Let $(\eta ^0, v^0)\in X_{-2}$, $g_0,g_1,h_0,h_1\in L^2(0,T)$, and $g_2,h_2\in H^{-\frac{1}{3}}(0,T)$. 
Pick any $(\theta ^0, u^0) \in X_2$ and denote by $(\theta , u)$ the solution of the homogeneous system \eqref{A51}. 
Then, by Proposition  \ref{prop3}, $\theta _{xx}(\cdot , 0), \theta _{xx}(\cdot , L), u_{xx}(\cdot , 0), u_{xx}(\cdot , L)\in L^2(0,T)$ and 
$\theta _x(\cdot , L), u_x(\cdot , 0)\in H^\frac{1}{3}(0,T)$. It follows that for any $S\in [0,T]$,  the linear map $L_S:X_{-2}\to \R$ defined in  \eqref{A63} is continuous.  As the map $\Gamma _S: (\theta ^0,u^0)\in X_2\mapsto (\theta (S), u(S))\in X_2$ is invertible, 
$(\eta (S), v(S))$ is defined in $X_{-2}$ in a unique way by \eqref{A62}. Furthermore, as $\Gamma _S$ and
$ \Gamma _S ^{-1} = \Gamma _{-S} $ are uniformly bounded in ${\mathcal L} ( X_2 )$ for $S\in [0,T]$ (see \eqref{A5}), 
we infer that $(\eta ,v)\in L^\infty (0,T, X_{-2})$ and that \eqref{A71} holds. On the other hand, 
if $(\eta ^0, v^0),(\theta ^0, u^0) \in X_3$ and $h_i, g_i\in C^2([0,T])$ with $h_i(0)=g_i(0)=0$  for $i=1,2,3$, 
the respective strong solutions $(\eta, v), (\theta , u) \in C([0,T],[H^3(0,L)]^2)\cap C^1([0,T], [L^2(0,L)]^2)$ of
\eqref{A50} and \eqref{A51} satisfy \eqref{A60}, and thus \eqref{A62}, for $X_3$ is dense in $X_2$. Using \eqref{A71} and a density argument, we infer that  $(\theta ,v)\in C([0,T], X_{-2})$ when we assume merely that $(\eta ^0, v^0)\in X_{-2}$, 
$g_0,g_1,h_0,h_1\in L^2(0,T)$, and $g_2,h_2\in H^{-\frac{1}{3}}(0,T)$.  

Assume finally that  $(\eta ^0, v^0)\in X_{-1}$, 
$g_0=g_1=h_0=h_1=0$ and $g_2,h_2\in L^2(0,T)$.  Pick any $(\theta ^0, u^0)\in X_2$. Then 
\begin{eqnarray}
 L_S(\theta ^0, u^0) &=&
\langle (\eta ^0, v^0), (\theta ^0, u^0)\rangle _{X_{-2},X_2}
+\langle g_2,  {\mathbf 1} _{(0,S)} \theta _x(\cdot , L) \rangle _{H^{-\frac{1}{3}} (0,T), H^\frac{1}{3}(0,T)}\nonumber \\ 
&& - \langle h_2,  {\mathbf 1} _{(0,S)} u_x(\cdot , 0) \rangle _{H^{-\frac{1}{3}} (0,T), H^\frac{1}{3}(0,T)} \nonumber \\
&=&
\langle (\eta ^0, v^0), (\theta ^0, u^0)\rangle _{X_{-1},X_1}
+ \int_0^S [ \theta _{x} (t,L) g_2(t) - u_{x} (t,0) h_2(t) ] dt.  \label{A80}
\end{eqnarray}
Using Proposition \ref{prop2} and \eqref{A80}, we see that $L_S$ can be extended as a continuous linear map from $X_1$ to $\R$. 
It follows from \eqref{A62}, the density of $X_2$ in $X_1$, and the fact that the map $(\theta ^0, u^0) \in X_1\mapsto (\theta (S), u(S))\in X_1$
is invertible with its norm and the norm of its inverse uniformly bounded for $S\in [0,T]$, that $(\eta , v)\in L^\infty (0,T, X_{-1})$ with
\eqref{A72} satisfied. The same argument as above shows that  $(\eta , v)\in C([0,T], X_{-1})$. 
\end{proof}

\section{Controllability of the linearized system}
In this section, we are interested in the exact controllability of the linear system \eqref{A50} using at most two control inputs (the other 
control inputs being replaced by 0). The exact controllability in $X_{-2}$ is defined as follows.
\begin{definition}
We say that the system \eqref{A50} is exactly controllable in $X_{-2}$ (in small time)  if for all $(\eta ^0,v^0),(\eta ^1,v^1)\in X_{-2}$ and all $T>0$, one can
 find $g_0,g_1,h_0,h_1\in L^2(0,T)$ and $g_2,h_2\in H^{-\frac{1}{3}} (0,T)$ such that the solution $(\eta , v)$ of \eqref{A50} satisfies
\be
\label{B1}
(\eta (T,\cdot ),v(T,\cdot )) = (\eta ^1, v^1). 
\ee
Here, we will restrict to situations when at most two control inputs can be prescribed, the other being set to 0. 
Similarly, we say that the system \eqref{A50} with $g_0=g_1=h_0=h_1=0$ is exactly controllable in $X_{-1}$ (in small time)  if for all $(\eta ^0,v^0),(\eta ^1,v^1)\in X_{-1}$ and all $T>0$, one can find $g_2,h_2\in L^2(0,T)$ such that the solution $(\eta , v)$ of \eqref{A50} satisfies \eqref{B1}.
\end{definition} 
As system \eqref{A1} is reversible, the exact controllability of system \eqref{A50} is equivalent to its null controllability (i.e., 
the above property with $(\eta ^0,v^0)$ arbitrary but $(\eta ^1,v^1)=(0,0)$). 

As it is well known, the exact controllability of a control system is equivalent to the observability of its adjoint system
(see e.g. \cite{coron-book,lions1,lions2}). Here, using  \eqref{A62}-\eqref{A63}, we see that the exact controllability of \eqref{A50} in $X_{-2}$ with the six controls
$g_0,g_1,h_0,h_1 \in L^2(0,T)$, $g_2,h_2\in H^{-\frac{1}{3}} (0,T)$
is equivalent to the existence of a constant $C>0$ such that for all $(\theta ^0, u^0)\in X_2$, 
\ba
\Vert (\theta ^0, u^0) \Vert ^2_{X_2} 
&\le& C \left( \int_0^T \big[|\theta _{xx}(t,0)|^2 + |\theta _{xx}(t,L)|^2 + |u_{xx}(t,0)|^2 +  |u_{xx}(t,L)|^2\big]dt \right. \nonumber  \\
&&\left. \qquad +  \Vert \theta _x(\cdot , L)  \Vert ^2 _{H^\frac{1}{3} (0,T)} + \Vert u_x(\cdot , 0) \Vert ^2 _{H^\frac{1}{3} (0,T)}  \right) \label{B2}
\ea
where $(\theta , u)$ denotes the solution of \eqref{A51}. 

The exact controllabity in $X_{-2}$  with less control inputs is equivalent to the observability inequality \eqref{B2} in which we remove the traces 
corresponding to the missing controls.  Similarly, the exact controllability  of \eqref{A50} in $X_{-1}$ with the two controls
$g_2,h_2 \in L^2(0,T)$ is equivalent to the existence of a constant $C>0$ such that for all $(\theta ^0, u^0)\in X_1$, 
\be
\Vert (\theta ^0, u^0) \Vert ^2_{X_1} 
\le C \int_0^T \big[ |\theta _{x}(t,L)|^2 + |u_{x}(t,0)|^2 \big]dt .
\label{B3}
\ee
The aim of this section is to provide a complete picture of the exact controllability results for system \eqref{A50} with one or two control inputs
among $g_0,g_1,g_2,h_0,h_1,h_2$. 

We introduce some subsets of $(0,+\infty )$ that are needed to give a complete picture of the critical lengths 
for the control of the linearized KdV-KdV system:

\begin{eqnarray*}
\cN       &:=& \left\{ \frac{2\pi}{\sqrt{3}} \sqrt {k^2+kl+l^2} ; \quad  k,l\in \N ^* \right\} , \\
\cN _3  &:=&  \left\{ \frac{2\pi}{\sqrt{3}} \sqrt {k^2+kl+l^2} ; \quad  k,l\in \N ^*, \ 3 \vert 2k+l  \right\} \subset \cN , \\
\cR       &:=& \left\{ \pi \sqrt{(\frac{1}{2} + 2k)^2 + (\frac{1}{2} +2l)^2 + (\frac{1}{2}+2k)(\frac{1}{2}+2l) } ; \quad  k,l\in \Z, \ k\ne l \right\} , \\
\cG &:=& \left\{ L\in (0,+\infty);  \ 
\exists a,b\in \C, \ \frac{2a}{ie^a-1}+a=\frac{2b}{ie^b-1} + b = \frac{ 2(-a-b) }{ie^{-a-b} -1}  -a-b \ne 0 , \right. \nonumber \\
&& \left. \qquad L^2=-(a^2+ab+b^2) \right\} , \\
\cG '&:=& \left\{ L\in (0,+\infty);  \ 
\exists a,b\in \C, \ \frac{2a}{-ie^a-1}+a=\frac{2b}{-ie^b-1} + b = \frac{ 2(-a-b) }{-ie^{-a-b} -1}  -a-b \ne 0 , 
\right. \nonumber \\
&& \left. \qquad  L^2= - (a^2+ab+b^2) \right\} . \\
\end{eqnarray*}

Then the control results for the linearized Boussinesq system with one or two boundary controls  of Dirichlet or Neumann type 
are outlined in the following theorem.

\begin{theorem}
\label{thm1} 
Consider system \eqref{A50} with one or two control inputs among $g_0,g_1,g_2,h_0,h_1,h_2$ (the other being set to 0). Then 
the exact controllability of \eqref{A50} holds (in any time $T>0$) if and only if $L$ is not a critical length, a concept depending 
on the set of control inputs that are available. The precise results are reported in Table 1.  
\end{theorem} 

\[%
\begin{array}
[c]{c}%
\begin{tabular}
[c]{|c|c|c|c|c|c|c|c|c|c|}\hline
& \multicolumn{7}{|c|}{Controls} & \multicolumn{2}{|c|}{Properties}\\\hline
Case & $h_0$ & $h_1$ & $h_2$ & $g_0$ & $g_1$ & $g_2$ & Control
Inputs & State Space & Set of Critical\\
&&&&&&&&&  Lengths\\\hline
1 & $0$ & $0$ & $0$ & $0$ & $0$ & $\star$ & $g_2\in L^{2}(0,T)$ & $(\eta
_{0},w_{0})\in X_{-1}$ & $\mathcal{N}$\\\hline
2 & $0$ & $0$ & $0$ & $0$ & $\star$ & $0$ & $g_1\in L^{2}(0,T)$ & $(\eta
_{0},w_{0})\in X_{-2}$ & $\mathcal{N}\cup\mathcal{R}$\\\hline
3 & $0$ & $0$ & $0$ & $\star$ & $0$ & $0$ & $g_0\in L^{2}(0,T)$ & $(\eta
_{0},w_{0})\in X_{-2}$ &  ${\mathcal N}  \cup \cG \cup \cG '$ \\ \hline
4 & $0$ & $0$ & $\star$ & $0$ & $0$ & $\star$ & $h_2,g_2\in L^{2}(0,T)$ & $(\eta
_{0},w_{0})\in X_{-1}$ & $\mathcal{N}$\\\hline
5 & $0$ & $\star$ & $0$ & $0$ & $\star$ & $0$ & $h_1,g_1\in L^{2}(0,T)$ &
$(\eta_{0},w_{0})\in X_{-2}$ & $\emptyset$\\\hline
6 & $0$ & $\star$ & $0$ & $0$ & $0$ & $\star$ & $h_1\in L^{2}(0,T), g_2\in H^{-\frac{1}{3}}(0,T)$ &
$(\eta_{0},w_{0})\in X_{-2}$ & $\mathcal{N}$\\\hline
7 & $0$ & $0$ & $0$ & $0$ & $\star$ & $\star$ & $g_1\in L^{2}(0,T), g_2\in H^{-\frac{1}{3}}(0,T)$ & $(\eta
_{0},w_{0})\in X_{-2}$ & $\cN$ \\\hline
8 & $\star$ & $0$ & $0$ & $0$ & $0$ & $\star$ & $h_0\in L^{2}(0,T),g_2\in H^{-\frac{1}{3}}(0,T)$ &
$(\eta_{0},w_{0})\in X_{-2}$ & $\cN$ \\\hline
9 & $0$ & $0$ & $0$ & $\star$ & $0$ & $\star$ & $g_0\in L^{2}(0,T), g_2\in H^{-\frac{1}{3}}(0,T)$ &
$(\eta_{0},w_{0})\in X_{-2}$ & $\cN$ \\\hline
10 & $\star$ & $0$ & $0$ & $0$ & $\star$ & $0$ & $h_0,g_1\in L^{2}(0,T)$ &
$(\eta_{0},w_{0})\in X_{-2}$ & $\cR $ \\\hline
11 & $0$ & $0$ & $0$ & $\star$ & $\star$ & $0$ & $g_0,g_1\in L^{2}(0,T)$ &
$(\eta_{0},w_{0})\in X_{-2}$ &$\cN _3$ \\\hline
12 & $0$ & $\star$ & $0$ & $\star$ & $0$ & $0$ & $h_1,g_0\in L^{2}(0,T)$ &
$(\eta_{0},w_{0})\in X_{-2}$ &$\cG \cup \cG '$ \\\hline
\end{tabular}
\ \ \smallskip\\
\text{Table 1. Controllability results for the linearized system}%
\end{array}
\]
\begin{remark}
Using the symmetries, we notice that all the possibilities with one or two control inputs  are really considered in Table 1: indeed,
$h_2$ alone is similar to
$g_2$ alone (case 1); $h_0$ alone is similar to $g_1$ alone (case 2); $h_1$ alone is similar to $g_0$ alone (case 3); 
the pair $(h_0,g_0)$ is similar to the pair $(h_1,g_1)$  (case 5); the pair $(h_2,g_1)$ is similar to the pair $(h_0,g_2)$ 
(case 8); the pair $(h_1,h_2)$ is similar to the pair $(g_0,g_2)$ (case 9); the pair $(g_0,h_2)$ is similar to the pair $(h_1,g_2)$
 (case 6); the pair $(h_0,h_2)$ is similar to the pair $(g_1,g_2)$  (case 7); the pair $(h_0,h_1)$ is similar to the pair $(g_0,g_1)$  (case 11). 
\end{remark}
The proof of Theorem \ref{thm1} is detailed in a series of propositions or theorems displayed in several subsections. 
\subsection{Neumann controls}
We consider first two controls (case 4) and next only one control (case 1). 
We shall use repeatedly the following results from \cite{R1}.
\begin{proposition}
\label{propR1} (\cite[Proposition 3.3]{R1})
For all $L\in (0,+\infty )\setminus {\mathcal N}$ and $T>0$, there exists a constant $C=C(L,T)>0$ such that 
\be
\int_0^L |y^0(x)|^2 \, dx \le C \int_0^T |y_x(t,0) |^2 \, dt, \qquad  \forall y^0\in L^2(0,L), 
\ee
where $y=y(t,x)$ denotes the solution to the linearized KdV system
\be
\label{B6}
\left\{ 
\begin{array}{ll}
y_t + y_{x}+y_{xxx}=0, \quad &t\in (0,T), \ x\in (0,L),\\
y(t,0)=y(t,L)=y_x(t,L)=0, \quad &t\in (0,T), \\
y(0,x)=y^0(x), \quad &x\in (0,L). 
\end{array}
\right. 
\ee
\end{proposition}

\begin{lemma}
\label{lemR1} (\cite[Lemma 3.5]{R1}) Let $L>0$. Then $L\in {\mathcal N}$ if and only if  
there exist some numbers $p\in \C$ and $(\alpha, \beta)\in \C ^2\setminus \{ (0,0) \}$ such that the map 
\be
\label{B7}
f(\xi ) =\frac{\alpha -\beta e^{-i L\xi}}{\xi ^3 -\xi +p}
\ee
is an entire function, i.e. a complex analytic function on $\C$. 
\end{lemma}
The controllability result for the case 4 is a consequence of the following observability inequality. 
\begin{proposition}
\label{prop5}
Let $L\not\in {\mathcal N}$ and $T>0$. Then there exists a constant $C=C(L,T)>0$ such that
\be
\label{B11}
\Vert (\theta ^0,u^0)\Vert ^2_{X_1} \le C \int_0^T \big[ |\theta _x(t,L)|^2 + |u_x(t,0 ) |^2 \big] dt, \qquad \forall (\theta ^0,u^0)\in X_1,
\ee 
where $(\theta , u)$ denotes the solution of \eqref{A51}.
\end{proposition}
\begin{proof}
If \eqref{B11} is not true, one can find two sequences $(\theta^0_n)_{n\ge 0}$ and $(u^0_n)_{n\ge 0}$ such that 
\be
\label{B21}
1=\Vert (\theta ^0_n, u^0_n )\Vert ^2 _{X_1} > n \int_0^T  [ | \theta _{n,x} (t,L) |^2 + |u _{n,x} (t,0) |^2]dt.  
\ee
Extracting a subsequence if necessary, one can assume that $(\theta ^0_n,u^0_n)\to (\theta ^0 ,u^0)$ weakly in $X_1$, and
hence strongly in $X_0$. In particular, we have that 
\begin{eqnarray*}
(\theta ^0_n,u^0_n)\to  (\theta ^0, u^0)&& \textrm{ strongly in } X_0, \\
(\theta _n,u_n)\to  (\theta , u) && \textrm{ strongly in } C([0,T],X_0), \\
(\theta  _n,u_n)\to  (\theta , u)&& \textrm{ weakly in } L^2(0,T,X_1), \\
\theta  _{x,n} (\cdot , L)  \to  \theta_x (\cdot , L) && \textrm{ weakly in } L^2(0,T), \\
u_{x,n} (\cdot , 0) \to  u_x(\cdot , 0)&& \textrm{ weakly in } L^2(0,T).
\end{eqnarray*}
Using \eqref{B21}, we infer that $\theta  _{x,n} (\cdot , L)  \to 0$ and $u_{x,n} (\cdot , 0) \to 0$ strongly in $L^2(0,T)$. Thus 
$\theta _x(\cdot , L)=u_x(\cdot , 0)=0$. It follows that $y(t,x):=\theta (t,x) +u(t,x)$ satisfies 
\[
\left\{ 
\begin{array}{ll}
y_t + y_{xxx}+y_x=0, \quad &t\in (0,T), \ x\in (0,L),\\
y(t,0)=y(t,L)=y_x(t,L)=0, \quad &t\in (0,T), \\
y(0,x)=\theta ^0(x) +u^0(x), \quad &x\in (0,L)
\end{array}
\right. 
\]  
together with $y_x(t,0)=0$. We infer from Proposition \ref{propR1} that $\theta ^0+u^0=0$. Similarly, 
considering the function $z(t,x) :=\theta(t,L-x)-u(t,L-x)$,  we can
conclude that $\theta ^0(L-x)-u^0(L-x)=0$. Therefore, $\theta ^0=u^0=0$, and thus $(\theta _n , u_n)\to (0,0)$ strongly in $C([0,T], X_0)$.
But it follows from the above convergences and 
from \eqref{A200}  that $(\theta _n,u_n) \to (0,0)$ strongly in $L^2(0,T, X_1)$, and with  \eqref{A5}, that $( \theta _n ^0, u_n^0) \to (0,0)$ strongly in $X_1$. 
This contradicts \eqref{B21}. 
\end{proof}
Let us investigate the exact controllability of \eqref{A50} in case 4. If $L\not\in {\mathcal N}$, then by Proposition \ref{prop5}, 
system \eqref{A50} is exactly controllable in $X_{-1}$. If $L\in {\mathcal N}$, then by  \cite[Remark 3.6 (i)]{R1} there  exists a nontrivial
solution $y$ of \eqref{B6} such that $y_x(\cdot , 0)=0$. Then $(\theta (t,x),u(t,x)):=(y(t,x)+y(t,L-x), y(t,x)-y(t,L-x))$ is a nontrivial solution of 
\eqref{A51} such that $\theta _x(\cdot , L)=u_x(\cdot , 0)=0$. It follows that \eqref{A50} is not exactly controllable in $X_{-1}$ if $L\in \cN$.

We now turn into the problem of the controllability of \eqref{A50} with only one Neumann control (case 1).  
The following observability inequality improves Proposition \ref{prop5}, for only {\em one} boundary measurement is needed 
to estimate the initial data. 
\begin{theorem}
\label{thm2}
Let $L\not\in {\mathcal N}$ and $T>0$. Then there exists a constant $C=C(L,T)>0$ such that
\be
\label{B31}
\Vert (\theta ^0,u^0)\Vert ^2_{X_1} \le C \int_0^T  |\theta _x(t,L ) |^2  dt, \qquad \forall (\theta ^0,u^0)\in X_1,
\ee 
where $(\theta, u)$ denotes the solution of \eqref{A51}. 
\end{theorem}

\begin{proof}
We follow closely \cite{R1}. In the first step, we transform the problem into a spectral problem. In the second step, we solve the spectral problem by using Paley-Wiener theorem combined with complex analysis. Introduce some notations.  For any $T>0$, let 
\begin{eqnarray*}
N_T &:=&\{ (\theta ^0, u^0)\in X_1;\  \textrm{ the solution } (\theta ,u) \textrm{ of \eqref{A51}  satisfies } \theta _x(t,L)=0 \textrm{ for a.e. } t\in (0,T) \} , \\
M_T&:=&\{ (\theta, u)=e^{tA}(\theta ^0, u^0); \ (\theta ^0, u^0)\in N_T, \ t\in [0,T]\} \subset C([0,T],X_1).  
\end{eqnarray*} 
\noindent
{\sc Step 1. Reduction to a spectral problem.}\\

Pick any $L\in (0,+\infty ) \setminus \cN$ and any $T>0$. If \eqref{B31} is not true, one can find a sequence $(\theta^0_n,u^0_n)_{n\ge 0}$ in $X_1$ such that 
\be
\label{B32}
1=\Vert (\theta ^0_n, u^0_n )\Vert ^2 _{X_1} > n \int_0^T   |\theta _{n,x} (t,L)|^2 dt, \qquad \forall n\ge 0.  
\ee
Extracting a subsequence if necessary, one can assume that $(\theta ^0_n,u^0_n)\to (\theta ^0 ,u^0)$ weakly in $X_1$, and
hence strongly in $X_0$. Let $(\theta _n,u_n)$ and $(\theta ,u)$ denote the solutions of \eqref{A51} associated with $(\theta ^0_n,u^0_n)$ and
$(\theta ^0,u^0)$, respectively.  
Using \eqref{B32}, we infer that $\theta _{n,x} (\cdot , L) \to 0$ strongly in $L^2(0,T)$, and hence
$\theta _x(\cdot , L)=0$.
Using  \eqref{A200}, \eqref{B32} and the fact that $(\theta _n,u_n)\to (\theta , u)$ in $C([0,T], X_0)$,  we see that $(\theta _n,u_n)\to (\theta  ,u)$ strongly in $L^2(0,T,X_1)$, so that 
$(\theta ^0_n , u^0_n)\to (\theta ^0,u^0)$ strongly in $X_1$, and hence $\Vert (\theta ^0, u^0)\Vert _{X_1}=1$. 
Thus $(\theta ^0, u^0)\in N_T \setminus \{ (0,0)\} $. To obtain a contradiction, it is sufficient to prove the following\\[3mm] 
{\sc Claim 1.} $N_T=\{ (0,0)\} $ for all $T>0$.\\[3mm]
It is clear that $T<T'\Rightarrow N_{T'}\subset N_T$. On the other hand, 
the vector space $N_T$ has a finite dimension, for its unit ball is compact.  Indeed, the same argument as above shows that any 
sequence $(\theta ^0_n,u^0_n)_{n\in \N}$ in the unit ball of $N_T$ has a convergent subsequence for the norm $\Vert \cdot\Vert _{X_1}$.  
If Claim 1 is not true, one may find $T'>0$ such that $\textrm{dim } N_{T'}>0$. Since the map $T\in (0,+\infty ) \mapsto \textrm{dim } N_T \in \N$ is
nonincreasing, one may pick $T\in (0,T')$ and $\varepsilon >0$ such that $T+\varepsilon <T'$ and
$\textrm{dim  } N_T=\textrm{dim }N_{T+\varepsilon} \ge \textrm{dim  } N_{T'} >0$.   It follows that $N_t=N_T$ for $T\le t\le T+\varepsilon$. 
Let $(\theta ^0, u^0)\in N_T$, and let $(\theta, u)$ denote the solution of \eqref{A51}.  By  the semigroup property, we have that 
\[ 
\frac{1}{t}\left(  (\theta (t, \cdot ),u(t,\cdot ))-(\theta ^0, u^0)\right)  \in N_T, \quad \forall t\in (0,\varepsilon ), 
\]
and hence 
\[ 
\frac{1}{t}\left(  (\theta (t+\cdot, \cdot ),u(t+\cdot, \cdot ))-(\theta , u)  \right) \in M_T, \quad \forall t\in (0,\varepsilon ). 
\]
Since $(\theta , u)\in H^1(0,T+\varepsilon,[H^{-2}(0,L)]^2)$, we have that 
\[
\lim_{t \to 0^+} \frac{1}{t}\left(  (\theta (t+\cdot, \cdot ),u(t+\cdot, \cdot ))-(\theta , u)  \right) =(\theta _t, u_t)\quad
\textrm{ in } L^2(0,T, [H^{-2}(0,L)]^2). 
\] 
But the vector space $M_T$ is closed in $L^2(0,T, [H^{-2}(0,L)]^2)$, for it is finite-dimensional. It follows that 
$(\theta _t , u_t)\in M_T\subset C([0,T], X_1)$, so that $(\theta ,u)\in C^1 ([0,T], X_1)$. This implies that 
\[
(\theta ^0, u^0)\in D(A)=X_3, \quad A(\theta ^0, u^0) =(\theta _t(0,.), u_t(0,.))\in N_T\quad \textrm{ and } \theta _x(\cdot , L)\in C([0,T]).
\]

In particular 
\[
\frac{d\theta ^0}{dx}(L)=\theta _x(0,L)=0. 
\]

Let $N_T^\C$ denote the complexification of $N_T$, and let $'=d/dx$. Since it was assumed that $\textrm{ dim } N_T\ne 0$, we infer that the 
linear map   $(\theta ^0,u^0)\in N_T^\C \mapsto A(\theta ^0,u^0)\in N_T^\C$ has at least one eigenvalue; that is, there exist 
$\lambda \in \C$  and $(\theta ,u)\in X_3\setminus \{ (0,0) \}$ \footnote{We removed the superscript $0$ to simplify the notations. We still used the notation $X_3$ to 
denote the space of {\em complex-valued functions} in \eqref{MMMM2}.} solution of the spectral problem
\ba
-u'-u'''=\lambda \theta,&&  x\in (0,L), \label{B41}\\
-\theta ' -\theta '''=\lambda u,&& x\in (0,L), \label{B42}\\
\theta (0)=\theta (L)=\theta '(0)=\theta '(L)=0,\label{B43}   \\
u(0)=u(L)=u'(L)=0.&& \label{B44}
\ea
We show in the next step that \eqref{B41}-\eqref{B44} has no nontrivial solution when $L\not\in \cN$. \\

\noindent
{\sc Step 2. Study of the spectral problem.}\\

We shall prove the following\\[3mm]
{\sc Claim 2.} For $L\not\in \cN$, if $\lambda\in \C$ and $(\theta, u)\in X_3$ satisfy \eqref{B41}-\eqref{B44}, then $\theta =u=0$. \\[3mm]
Let $(\theta ,u)$ be as in Claim 2, and extend $\theta$ and $u$ to $\R$ by setting  $\theta(x)=u(x)=0$ for $x\not\in [0,L]$.  Then we have in ${\mathcal S }'(\R )$ 
\begin{eqnarray*}
\lambda \theta +u'+u'''  &=& u'(0)\delta '_0 + u''(0) \delta _0-u''(L)\delta _L, \\
\lambda u + \theta '+ \theta '''  &=& \theta ''(0) \delta _0 -\theta ''(L) \delta _L,
\end{eqnarray*}
where $\delta_\zeta$ denotes the Dirac measure at $x=\zeta$ and the derivatives $u'(0),u''(0),u''(L),\theta ''(0),\theta ''(L)$ are those of the 
functions $u$ and $\theta$ when restricted to $[0,L]$.  Taking the Fourier transform of each term in the above system, we obtain
\begin{eqnarray*}
\lambda \hat \theta (\xi )  + i\xi \hat u (\xi )  + (i\xi )^3 \hat u (\xi)  &=& u'(0) i\xi + u''(0)-u''(L) e^{-iL\xi}, \\
\lambda \hat u (\xi ) + i\xi \hat \theta  (\xi)  + (i\xi )^3 \hat \theta (\xi )  &=& \theta ''(0) -\theta ''(L) e^{-iL\xi }.
\end{eqnarray*}   
Setting 
$\lambda = -ip$, $f(\xi) :=\hat \theta (\xi) +\hat u(\xi)$,
and $g(\xi ) := \hat \theta (\xi) -\hat u(\xi)$, we arrive to  
\begin{eqnarray*}
f(\xi ) &=& \frac{i}{\xi ^3-\xi +p} ( \alpha +\beta i\xi +\gamma  e^{-iL\xi} ) ,\\
g(-\xi ) &=& \frac{i}{\xi ^3 -\xi +p} (\alpha ' -\beta i\xi + \gamma ' e^{iL\xi} ),
\end{eqnarray*}
where 
$\alpha := u''(0)+\theta ''(0)$,
$\alpha ' := u''(0)-\theta ''(0)$,  
$\beta := u'(0)$, 
$\gamma  := -u''(L) -\theta ''(L)$, and
$\gamma ' :=  -u''(L )  + \theta ''(L)$.   
Since both $\theta$ and $u$ have a compact support, it follows from the Paley-Wiener theorem that the functions $f$ and $g$ have  to be 
entire  (i.e. holomorphic in the whole plane $\C$). The same is true for the function
\[
h(\xi):=f(\xi ) + g(-\xi) = \frac{ie^{-iL\xi}}{\xi ^3-\xi +p} (\gamma ' e^{2iL\xi} +(\alpha + \alpha ') e^{i L \xi} + \gamma). 
\]
Introduce the polynomial functions $P(z) :=\gamma ' z^2 + (\alpha +\alpha ') z + \gamma$ and $Q(\xi) := \xi ^3-\xi +p$, and let 
$\mu _k$, $k=0,1,2$, denote the three roots of $Q$. Then $P(e^{iL\mu _k})=0$ for $k=0,1,2$, since $h$ is entire. 
Let 
\[
z_\pm = \frac{ -(\alpha + \alpha ') \pm \sqrt{ (\alpha + \alpha ')^2 -4\gamma \gamma ' }} {2\gamma ' }
\] 
denote the roots of the polynomial function $P$. \\
(i) If $\beta=0$, then applying Lemma \ref{lemR1} to $f(\xi )$ and $g(-\xi)$, we infer that $\alpha=\gamma=\alpha '=\gamma '=0$, so that 
\[
u''(0)=u''(L)=u'(0)=\theta ''(0)=\theta ''(L)=0. 
\]
It follows then that $\theta=u=0$ in $[0,L]$, as desired.\\
(ii) If $\beta \ne 0$, then since $g(-\xi ) $ is entire,  each root $\mu _k$ of $Q$ has to satisfy
$\alpha ' -\beta i \mu _k+\gamma ' e^{ i L\mu _k}=0$, i.e. 
$\mu _k =(\alpha ' +\gamma ' e^{iL\mu _k} ) / ( i\beta)$. Since $e^{iL\mu _k}$ is a root of $P$,  
we arrive to the conclusion that 
\[
\mu _k \in \left\{ \frac{\alpha ' + \gamma ' z_+}{i\beta } ,   \frac{\alpha ' + \gamma ' z_-}{i\beta }  \right\} \cdot 
\]
We infer that $Q$ cannot have three distinct roots.\\
Assume that $Q$ has a root of order 3, i.e. $\mu _0=\mu_1=\mu _2$. Then we have that $Q(\mu _0)=Q'(\mu _0)=Q''(\mu _0)=0$. 
Since $Q''(\xi )=6\xi$, we conclude that $\mu _0=0$, hence $p=0$, so that $Q(\xi ) =\xi (\xi -1)(\xi +1)$ has three distinct roots, a contradiction.   

We arrive to the conclusion that $Q$ has a double root and a simple root. We can assume that $\mu _0=\mu _1\ne \mu _2$. Then, we infer from 
$Q(\mu _0)=Q'(\mu _0)=0$ and the fact that $Q' (\xi)= 3\xi ^2- 1$ that $\mu _0=\pm \frac{1}{\sqrt{3}}$ and $p=\pm \frac{2}{3\sqrt{3}}$. \\
$\star$ Assume that $\mu _0=\frac{1}{\sqrt{3}}$. Then $Q(\xi ) = \xi ^3-\xi  +\frac{2}{3\sqrt{3}}  = (\xi - \frac{1}{\sqrt{3}} )^2 (\xi + \frac{2}{\sqrt{3}})$. 
On the other hand, since $f$ is entire,  $\mu _2=-\frac{2}{\sqrt{3}}$ (resp. $\mu _0$)  has to be a root (resp. a root of order at least 2)  of the numerator of $f$, i.e. 
\begin{eqnarray}
\alpha +\frac{i}{ \sqrt{3} }   \beta  + \gamma  e^{ -i \frac{L}{\sqrt{3}}} &=&0, \label{B61}\\
i\beta -  iL \gamma  e^{-i\frac{L}{\sqrt{3}}} &=& 0,\label{B62} \\
\alpha -\frac{2i}{\sqrt{3}}\beta + \gamma e^{2i\frac{L}{\sqrt{3}}}&=&0.  \label{B63}
\end{eqnarray}
From \eqref{B62}, we infer that $\beta = L\gamma e^{-i L/\sqrt{3}}$. Substituting this value of $\beta$ in \eqref{B61} and \eqref{B63} and taking
the difference of the two obtained equations, we obtain that 
\be
\label{B64}
(i\sqrt{3} L\gamma  + \gamma) e^{-i\frac{L}{\sqrt{3}}} = \gamma e^{2i\frac{L}{\sqrt{3}}}. 
\ee  
- If $\gamma =0$, we obtain $\beta =0$ which contradicts the assumption in (ii).  \\
- If $\gamma \ne 0$, then taking the module of both terms in \eqref{B64} yields $L=0$, which is impossible.\\
$\star$ Assume now that $\mu _0=-\frac{1}{\sqrt{3}}$. Then, proceeding as above, we obtain that $\beta = L\gamma e^{i L/\sqrt{3}}$ and that 
\[
(-i\sqrt{3} L\gamma  + \gamma) e^{i\frac{L}{\sqrt{3}}} = \gamma e^{-2i\frac{L}{\sqrt{3}}}. 
\]
Again, we see that  we obtain a contradiction for $\gamma =0$ or for $\gamma \ne 0$.
The proof of Theorem \ref{thm2} is complete.   
\end{proof}
With Theorem \ref{thm2} at hand, we deduce that \eqref{A50} is exactly controllable in $X_{-1}$ with $g_2$ as unique control (case 1) when 
$L\not\in\cN $. On the other hand, the exact controllability fails when $L\in \cN$ according to case 4.  

Let us turn our attention to the cases 6, 7, 8 and 9 for which we have added one Dirichlet control to the Neumann control $g_2$.
First, we notice that in case 1, we have for $L\not\in \cN$ the exact controllability in $X_{-2}$ by picking the controls $g_2$ 
in $H^{-\frac{1}{3}}(0,T)$.  Indeed, from \eqref{B31}, we have for $(\theta ^0,u^0)\in X_4$
\[
\Vert A(\theta ^0, u^0)\Vert _{X_1} ^2 \le C  \int_0^T|\theta _{xt} (t,L) | ^2 dt
\]
 and hence 
 \be
 \label{B31a}
 \Vert (\theta ^0, u^0)\Vert _{X_4} ^2 \le C  \Vert \theta _{x} (\cdot , L) \Vert _{H^1(0,T )} ^2.
 \ee
By interpolation between \eqref{B31} and \eqref{B31a}, we infer that for all $(\theta ^0,u^0)\in X_2$,  
  \be
 \label{B31b}
 \Vert (\theta ^0, u^0)\Vert _{X_2} ^2 \le C  \Vert \theta _{x} (\cdot , L) \Vert _{H^\frac{1}{3}(0,T )} ^2.
 \ee
The observability inequality \eqref{B31b} implies the exact controllability in $X_{-2}$ in case 1 (with $g_2\in H^{-\frac{1}{3}}(0,T)$)
 as well as in cases  6 to  9 for $L\not\in\cN$. 
It remains to show that any $L\in \cN$ is still a critical length in cases 6 to 9. We claim that in each of those cases, the 
corresponding spectral problem has for $L\in \cN$ a nontrivial solution $(\theta ,u)$, so that the  corresponding observability inequality 
fails for the exponential solution $\textrm{Re} [ e^{\lambda t}(\theta ,u)]$.    

\begin{proposition}
\label{prop6}
Let $L\in \cN$. Then\\
1. (case 6) there exist 
$\lambda \in \C$  and $(\theta ,u)\in X_3\setminus \{ (0,0) \}$ \ satisfying \eqref{B41}-\eqref{B44} and $u''(L)=0$. \\
2. (case 7) there exist 
$\lambda \in \C$  and $(\theta ,u)\in X_3\setminus \{ (0,0) \}$ \ satisfying \eqref{B41}-\eqref{B44} and $\theta  ''(L)=0$. \\
3. (case 8) there exist 
$\lambda \in \C$  and $(\theta ,u)\in X_3\setminus \{ (0,0) \}$ \ satisfying \eqref{B41}-\eqref{B44} and $u ''(0)=0$. \\
4. (case 9) there exist 
$\lambda \in \C$  and $(\theta ,u)\in X_3\setminus \{ (0,0) \}$ \ satisfying \eqref{B41}-\eqref{B44} and $\theta '' (0)=0$.
\end{proposition}

\begin{proof}
Let $L\in {\mathcal N}$. Then we can write $L=2\pi \sqrt{\frac{k^2+kl+l^2}{3}}$ with $k,l\in \N^*$. Following \cite{R1}, we introduce the numbers
\[
\mu _0= - \frac{1}{3} (2k+l)\frac{2\pi}{L} , \quad
\mu _1 := \mu _0+k\frac{2\pi}{L}, \quad
\mu _2=\mu _1 + l\frac{2\pi}{L}, \quad
p:=i\lambda = -\mu _0\mu _1\mu _2. 
\]
We use again the notations in the proof of Theorem \ref{thm2} (Claim 2). Let
\begin{eqnarray*}
f(\xi ) &=& \frac{i}{\xi ^3-\xi +p} ( \alpha +\beta i\xi +\gamma  e^{-iL\xi} ),\\
g(-\xi ) &=& \frac{i}{\xi ^3 -\xi +p} (\alpha ' -\beta i\xi + \gamma ' e^{iL\xi} ),
\end{eqnarray*}
where
$\alpha := u''(0)+\theta ''(0)$,
$\alpha ' := u''(0)-\theta ''(0)$,  
$\beta := u'(0)$, 
$\gamma  := -u''(L) -\theta ''(L)$, and
$\gamma ' :=  -u''(L )  + \theta ''(L)$.  
We know that we can find coefficients $\alpha, \beta, \gamma, \alpha ', \gamma '$ not all zero so that the two functions $f(\xi ) $ and $g (-\xi ) $ 
are entire, the roots of $Q(\xi) =\xi ^3-\xi +p$ being $\mu_0,\mu_1,\mu _2$. Furthermore, it follows from Paley-Wiener theorem 
(see \cite{R1}) that 
the spectral problem \eqref{B41}-\eqref{B44} has indeed a nontrivial solution.  Our concern is to prove that we can as well 
impose the addition condition in each case.

From our choice of the $\mu _k$'s, we have that the quantity $e^{-i L \mu _k}$ is independent of $k$. Set 
\[
C:=e^{-iL\mu_0} =  e^{-iL\mu _1} = e^{-iL\mu _2}.  
\]
We shall pick $\beta =0$ in all the cases. 
 \\
1.  The additional condition $u''(L)=0$ is equivalent to $\gamma ' = -\gamma$. We can pick $\gamma =1$, $\gamma' =-1$, 
$\alpha =-C$, and $\alpha ' =\overline{C}$. \\
2.  The additional condition $\theta ''(L) =0 $ is equivalent to $\gamma '=\gamma $. We can pick $\gamma =\gamma' =1$, 
$\alpha =-C$, and $\alpha ' =-\overline{C}$. \\
3.  The additional condition $u''(0)=0$ is equivalent to $\alpha ' = -\alpha $. We can pick
$\alpha =1$ and $\alpha ' =-1$, $\gamma =-\overline{C}$, and $\gamma' =C$.  \\
4. The additional condition $\theta ''(0)=0$ is equivalent to $\alpha ' = \alpha $. We can pick
$\alpha =\alpha ' =1$, $\gamma =-\overline{C}$, and $\gamma' =-C$.  
\end{proof}

\subsection{Dirichlet controls} We consider the cases in which only Dirichlet controls are involved. We start with a 
preparatory result which is an observability inequality with the measurement of three traces. 

\begin{proposition}
\label{prop7}
For all $L>0$ and $T>0$, there is a constant $C=C(L,T)>0$ such that 
\be
\label{C1} 
\Vert (\theta ^0,u^0)\Vert ^2 _{X_2} \le C \int _0^T \big[   | \theta _{xx} (t,L)|^2 
+ | u_{xx}  ( t,L ) |^2  + | \theta _x (t,L)|^2 \big] \, dt, \qquad \forall (\theta ^0, u^0)\in X_2. 
\ee
\end{proposition} 
\begin{proof}
First, we introduce the space 
\[
Y:=L^2(0,T,  [ H^\frac{7}{4}(0,L)]^2)\cap C([0,T], [H^1(0,L)]^2)
\]
which is a Banach space when endowed with the norm
\[
\Vert (\theta , u)\Vert ^2_Y = \int_0^T  [\Vert \theta (t,\cdot ) \Vert ^2_{H^\frac{7}{4}} + \Vert u (t,\cdot ) \Vert ^2_{H^\frac{7}{4}  } ]dt 
+\sup _{0\le t \le T} [\Vert \theta (t,\cdot )  \Vert ^2_{H^1} + \Vert u(t,\cdot )\Vert ^2 _{H^1}].    
\]
Pick any  $(\theta ^0,u^0)\in X_4$ and let $(\theta, u)$ denote the solution of \eqref{A51}. Taking the derivative w.r.t. $x$ of each term in the 
two first equations of \eqref{A51} results in 
\ba
\theta _{tx} + u _{xx} + u _{xxxx}=0, && t\in (0,T),\  x\in (0,L), \label{C11}\\
u _{tx} + \theta _{xx} + \theta _{xxxx}=0,&& t\in (0,T), \  x\in (0,L). \label{C12}
\ea
Multiplying \eqref{C11} by $xu_x$, \eqref{C12} by $x\theta _x$, integrating by parts over $(0,T)\times (0,L)$ and adding the two equations, 
we arrive to 
\ba
&&\left[ \int_0^L x \theta _x u_x\, dx \right]_0^T -\frac{1}{2}\int_0^T\!\! \!\int_0^L [\theta _x^2+u_x^2]dxdt 
+ \frac{3}{2} \int_0^T\!\!\!\int_0^L [\theta _{xx} ^2 +u_{xx}^2] \, dxdt \nonumber\\
&&\qquad -\int_0^T \left[ \frac{x}{2}(\theta _x^2+u_x^2)  +\theta _x\theta _{xx} +
 u_xu_{xx}+ \frac{x}{2} (\theta _{xx}^2 + u_{xx}^2) \right] _0^Ldt=0. \label{C14}
\ea
Combining \eqref{C14}  with the forward/backward wellposedness in $X_2$ and Sobolev embedding, this yields
\begin{eqnarray}
\Vert (\theta ^0, u^0 )\Vert ^2 _{X_2} &\le & C\int_0^T \Vert (\theta , u)\Vert ^2_{X_2} dt \nonumber  \\
&\le & C \int_0^T \!\!\! \int_0^L[ \theta _{xx}^2 +  u_{xx}^2  ]\, dxdt \nonumber \\
&\le & C \left( \Vert (\theta , u)\Vert ^2_Y + \int _0^T [\theta ^2_{xx}(t,L)+ u ^2_{xx}(t,L)+ u ^2_{xx}(t,0)]dt \right) . \label{C15}
\end{eqnarray} 

Multiplying the two first equations in \eqref{A51} by $u_{xx}$ and $\theta _{xx}$ respectively, integrating by parts 
over $(0,T)\times (0,L)$ and adding the two equations, we obtain 
\be
\label{C16}
\frac{1}{2} \int_0^T [\theta _{xx}^2 + u_{xx}^2 ]_0^L dt 
-\left[ \int_0^L  \theta _x u_x\, dx \right]_0^T +\frac{1}{2} \int_0^T \theta _x^2(t,L)dt -\frac{1}{2}\int_0^T u_x^2(t,0)dt=0,  
\ee
and hence 
\be\label{C17}
\int_0^T [\theta _{xx}^2(t,0)+  u_{xx}^2(t,0) ]dt \le C \left( \Vert (\theta , u)\Vert ^2_Y + \int_0^T [ \theta _{xx}^2 (t,L) + u_{xx}^2  (t,L)]dt \right) .
\ee
Combined with \eqref{C15}, this gives 
\be
\Vert (\theta ^0, u^0 )\Vert ^2 _{X_2} 
\le C \left( \Vert (\theta , u)\Vert ^2_Y + \int _0^T [\theta ^2_{xx}(t,L)+ u ^2_{xx}(t,L)]dt \right) . \label{C18}
\ee
The result is also true for $(\theta ^0, u^0)\in X_2$, by density of $X_4$ in $X_2$. It remains to ``remove'' the term  $ \Vert (\theta , u)\Vert ^2_Y$ 
in \eqref{C18}. Assuming that \eqref{C1} is not true, we can find a sequence $(\theta ^0_n,u^0_n)_{n\in \N}$ such that 
\be
\label{C19} 
1= \Vert (\theta ^0_n,u^0_n)\Vert ^2 _{X_2} > n \int _0^T \big[   | \theta ^n_{xx} (t,L)|^2 
+ | u ^n_{xx}  ( t,L ) |^2  + | \theta ^n_x (t,L)|^2 \big] \, dt. 
\ee
Extracting a subsequence if needed, we can assume that $(\theta ^0_n, u^0_n)\to (\theta ^0,u^0)$ weakly in $X_2$ (hence strongly
in $X_1$). As the corresponding solution $(\theta _n,u_n)$ of \eqref{A51} is bounded in $C([0,T], X_2)$ and in $H^1(0,T, X_{-1} )$, and 
since the first embedding in $X_2\subset [H^\frac{7}{4}(0,L) \cap H^1_0(0,L) ]^2\subset X_1$ is compact, we infer
from Aubin-Lions' lemma (see e.g. \cite[Corollary 4]{simon}) that the sequence  $(\theta _n,u_n)_{n\ge 0}$ admits a subsequence, still denoted
$(\theta _n,u_n)_{n\ge 0}$, which is strongly convergent in $ C( [0,T], [H^\frac{7}{4} (0,L)\cap H^1_0(0,L) ]^2)$, 
hence in $Y$. It follows 
then from \eqref{C18} that the sequence $(\theta ^0_n,u^0_n)_{n\ge 0} $ is strongly convergent in $X_2$. Thus its limit $(\theta ^0,u^0)$ is such that $\Vert (\theta ^0, u^0)\Vert _{X_2}=1$ and the corresponding 
solution of \eqref{A51} satisfies 
\[
\theta _{xx}(t,L)=u_{xx}(t,L)=\theta _x(t,L)=0. 
\] 
We infer that the function $y(t,x):=\theta(t,x)+u(t,x)$ (resp. $y(t,x):=\theta (t,L-x)-u(t,L-x)$) solves the linearized KdV equation
$y_t+y_x+y_{xxx}=0$ and satisfies the boundary conditions $y(t,L)=y_x(t,L)=y_{xx}(t,L)=0$ (resp. $y(t,0)=y_x(t,0)=y_{xx}(t,0)=0$), and hence
it vanishes in $(0,T)\times (0,L)$ by Holmgren's theorem. We conclude that $\theta=u=0$ and this contradicts the fact that $\Vert (\theta ^0, u^0)\Vert _{X_2}=1$. 
\end{proof}

We are in a position to investigate the case 5. The following observability inequality improves those in Proposition \ref{prop7}.  
\begin{corollary}
\label{cor2}
For all $L>0$ and $T>0$, there is a constant $C=C(L,T)>0$ such that
\be
\label{C51} 
\Vert (\theta ^0,u^0)\Vert ^2 _{X_2} \le C \int _0^T \big[   | \theta _{xx} (t,L)|^2 
+ | u_{xx}  ( t,L ) |^2  \big] \, dt, \qquad \forall (\theta ^0, u^0)\in X_2. 
\ee
\end{corollary}
\begin{proof}

We proceed  in two steps as in the proof of Theorem \ref{thm2}.  \\

\noindent
{\sc Step 1. Reduction to a spectral problem.}\\
If \eqref{C51} is false, then one can pick a sequence $(\theta ^0_n, u^0_n)_{n\in \N}$ such that
\be
\label{C52} 
1= \Vert (\theta ^0_n,u^0_n)\Vert ^2 _{X_2} > n  \int _0^T \big[   | \theta _{n,xx} (t,L)|^2 
+ | u_{n,xx}  ( t,L ) |^2  \big] \, dt. 
\ee
Extracting a subsequence if needed, we can assume that $(\theta ^0_n, u^0_n)\to (\theta ^0, u^0)$ weakly in $X_2$. 
Denote by $(\theta _n, u_n)$ (resp. $(\theta , u)$) the solution of \eqref{A51} with initial data $(\theta ^0_n , u^0_n)$ (resp. 
$(\theta ^0, u^0)$).  
Extracting again a subsequence, we can assume that $(\theta _n, u_n)\to (\theta, u)$ in $Y$. By \eqref{C18} and \eqref{C52}, we
have that $(\theta ^0_n, u^0_n)\to (\theta ^0, u^0)$ strongly in $X_2$. Thus $\Vert (\theta ^0, u^0)\Vert_{X_2}=1$ and 
$\theta _{xx}(\cdot , L)=u_{xx}(\cdot , L)=0$ a.e. in $(0,L)$. The same argument as the one in Step 1 of the proof of Theorem \ref{thm2} shows
that there exist
 $\lambda \in \C$  and $(\theta ,u)\in X_3\setminus \{ (0,0) \}$ (superscript $0$ dropped for simplicity) 
 which solve  the spectral problem:
\ba
-u'-u'''=\lambda \theta,&&  x\in (0,L), \label{C61}\\
-\theta ' -\theta '''=\lambda u,&& x\in (0,L), \label{C62}\\
\theta (0)=\theta (L)=\theta '(0)=\theta ''(L)=0,\label{C63}   \\
u(0)=u(L)=u'(L)=u''(L)=0.&& \label{C64}
\ea
We show in the next step that system \eqref{C61}-\eqref{C64} has no nontrivial solution for any $L>0$. \\

\noindent
{\sc Step 2. Study of the spectral problem.} \\

\noindent
{\sc Claim 3. } Let $L>0$. If $\lambda\in \C$ and $(\theta, u)\in X_3$ satisfy \eqref{C61}-\eqref{C64}, then $\theta =u=0$. \\[3mm]
Let us prove Claim 3.
Pick $(\lambda , \theta ,u)$ as in Claim 3, and extend $\theta$ and $u$ to $\R$ by setting  $\theta(x)=u(x)=0$ for $x\not\in [0,L]$.  Then we have 
\begin{eqnarray*}
\lambda \theta +u'+u'''  &=& u'(0)\delta '_0 + u''(0) \delta _0, \\
\lambda u + \theta '+ \theta '''  &=&-\theta '(L) \delta _L'  + \theta ''(0) \delta _0, 
\end{eqnarray*}
where $\delta_\zeta$ denotes the Dirac measure at $x=\zeta$, and the derivatives $u'(0),u''(0),\theta '(L),\theta ''(0)$ are those of the 
functions $u$ and $\theta$ when restricted to $[0,L]$.  Taking the Fourier transform of each term in the above system, we obtain
\begin{eqnarray*}
\lambda \hat \theta (\xi )  + i\xi \hat u (\xi )  + (i\xi )^3 \hat u (\xi)  &=& u'(0) i\xi + u''(0), \\
\lambda \hat u (\xi ) + i\xi \hat \theta  (\xi)  + (i\xi )^3 \hat \theta (\xi )  &=& -  \theta '(L) (i\xi)  e^{-iL\xi } +  \theta ''(0).
\end{eqnarray*}   
Setting 
$\lambda = -ip$, $f(\xi) :=\hat \theta (\xi) +\hat u(\xi)$,
and $g(\xi ) := \hat \theta (\xi) -\hat u(\xi)$, we arrive to  
\begin{eqnarray*}
f(\xi ) &=& \frac
{i}{\xi ^3-\xi +p} ( \alpha +\beta i\xi +\gamma ' (i\xi) e^{-iL\xi} )\\
g(-\xi ) &=& \frac{i}{\xi ^3 -\xi +p} (\alpha ' -\beta i\xi + \gamma ' (i\xi)  e^{iL\xi} )
\end{eqnarray*}
where 
$\alpha := u''(0)+\theta ''(0)$,
$\alpha ' := u''(0)-\theta ''(0)$,  
$\beta := u'(0)$, and
$\gamma ' :=  -\theta '(L )$.   It follows that 
\begin{eqnarray}
f(\xi )+ g( - \xi ) &=& \frac{i}{\xi ^3-\xi +p} ( \alpha + \alpha '  +2 \gamma ' (i\xi) \cos(L\xi ) ) \label{ABC1}\\
f(\xi ) - g( - \xi )  &=& \frac{i}{\xi ^3-\xi +p} ( \alpha - \alpha '  +  2\beta i\xi  -2i \gamma ' (i\xi) \sin (L\xi ) ). \label{ABC2}
\end{eqnarray}
Let $Q(\xi )=\xi ^3 -\xi +p$ and let $\mu_0,\mu_1,\mu_2$ be the roots of $Q$. 

\noindent
1. Assume that the three roots $\mu_0,\mu_1,\mu_1$ are simple. If $\xi \in \{ \mu _0, \mu _1, \mu _2\} $, then
$\xi$ must be a root of the numerators of $f(\xi )+g(-\xi )$ and $f(\xi ) -g(-\xi)$, so that 
\begin{eqnarray}
2\gamma ' i \xi \cos (L\xi) &=& -\alpha -\alpha ', \label{C71}\\
2\gamma ' i  \xi \sin (L\xi) &=& \frac{\alpha - \alpha '}{i} + 2\beta \xi. \label{C72} 
\end{eqnarray}
Taking the square in both equations and summing, we obtain 
\[
-4{\gamma ' }^2 \xi ^2= (\alpha +\alpha ')^2 - (\alpha -\alpha ')^2 + 4\beta ^2 \xi ^2 +\frac{4\beta\xi}{i}(\alpha -\alpha ')
\]
and hence
\be
\label{C80}
(\beta ^2 +\gamma '^2) \xi ^2 +\frac{\beta\xi}{i}(\alpha -\alpha ') + \alpha \alpha '=0. 
\ee
As the three numbers $\mu_0,\mu_1,\mu_2$ satisfy \eqref{C80},
we infer that all the coefficients in \eqref{C80} are null; that is
\[
\beta^2+{\gamma '}^2 = \beta (\alpha -\alpha ' )=\alpha \alpha ' =0. 
\]

If $\gamma '=0$, then the solution $(\theta ,u)$ of the system  \eqref{C61}-\eqref{C62} of order 3 satisfies 
\eqref{C63}-\eqref{C64} together with $\theta ' (L)=0$ and is therefore null, by Cauchy-Lipschitz theorem (Cauchy data taken at $x=L$). 

If $\gamma '\ne 0$, then $\beta =\pm i\gamma '\ne 0$ and thus $\alpha =\alpha '=0$. Assume that $\beta =i\gamma '$. 
From \eqref{C71}-\eqref{C72}, we infer that 
\[\xi\cos(L\xi)=0, \quad  \xi (\sin (L\xi) -1)  =0 \quad \textrm{ for }  \xi \in \{\mu_0, \mu_1,\mu_2\} .\]

If $0\in \{ \mu _0, \mu _1, \mu_2 \}$, say $\mu  _0=0$, then $\sin (L\mu_j)=1$ for $j=1,2$, and hence 
\[
L \mu _j=\frac{\pi}{2} + 2k_j\pi, \quad k_j\in \Z .
\]
Therefore $ L ( \mu_0+\mu_1+\mu_2 ) = \pi(1+ 2(k_1+k_2))$, but this contradicts the fact that the sum of the roots of $Q$ is $0$.  

If $0\not\in  \{ \mu _0, \mu _1, \mu_2 \}$, then for $j=1,2,3$ we have $\sin (L\mu _j)=1$ and hence
\[
L \mu _j=\frac{\pi}{2} + 2k_j\pi, \quad k_j\in \Z .
\]
Then $ L ( \mu_0+\mu_1+\mu_2) = \pi(\frac{3}{2}+ 2(k_1+k_2+ k_3))\ne 0$, a contradiction. Similarly, we obtain also a contradiction
when $\beta = -i\gamma '$.  

\noindent
2. Assume now that $Q$ has a double root $\mu_0=\mu_1$ and a simple root $\mu_2\ne \mu_0$.  Then 
$(\mu_0,\mu_2)=\pm (\frac{1}{\sqrt{3}}, -\frac{2}{\sqrt{3}})$. We shall consider the case  
$(\mu_0,\mu_2)= (\frac{1}{\sqrt{3}}, -\frac{2}{\sqrt{3}})$, the case $(\mu_0,\mu_2)= (-\frac{1}{\sqrt{3}}, \frac{2}{\sqrt{3}})$
being similar. As $\mu _0$ is a double root of $Q$, it should be a root of the numerators of the functions $f(\xi ) +g(- \xi)$ and 
$f(\xi ) -g( -\xi )$ (see  
\eqref{ABC1} and \eqref{ABC2}) and of their first derivatives, so that 
\ba
2\gamma'  \frac{i}{\sqrt{3}} \cos (\frac{L}{\sqrt{3}} ) &=& -(\alpha + \alpha '), \label{F1}\\
2\gamma ' \frac{i}{\sqrt{3}} \sin (\frac{L}{\sqrt{3}}) &=& \frac{\alpha -\alpha '}{i} 
+2\frac{\beta}{\sqrt{3}} ,\label{F2}\\
2\gamma 'i \cos(\frac{L}{\sqrt{3}}) +2\gamma ' \frac{i}{\sqrt{3}} (-L\sin \frac{L}{\sqrt{3}}) &=& 0, \label{F3}\\ 
2\beta i + 2\gamma' \sin (\frac{L}{\sqrt{3}}) + 2\gamma ' \frac{L}{\sqrt{3}} \cos (\frac{L}{\sqrt{3}}) &=&0 .\label{F4}
\ea 
On the other hand, the number $\mu _2$ is a simple root of $Q$, and hence it should be a root of the numerators of the functions $f(\xi ) +g(- \xi)$ and  $f(\xi ) -g( -\xi )$, so that 
\ba
-4\gamma ' \frac{i}{\sqrt{3}} \cos (\frac{2L}{\sqrt{3}}) &=& -(\alpha + \alpha '), \label{F5}\\
-4\gamma ' \frac{i}{\sqrt{3}} \sin (-2 \frac{L}{\sqrt{3}}) &=& \frac{\alpha -\alpha '}{i} 
-4\frac{\beta}{\sqrt{3}} .\label{F6}
\ea
Note that \eqref{F3}-\eqref{F4} can be rewritten as 
\[
\left( 
\begin{array}{cc}
\cos (\frac{L}{\sqrt{3}})  & - \sin  (\frac{L}{\sqrt{3}})  \\[3mm]
\sin (\frac{L}{\sqrt{3}}) & \cos (\frac{L}{\sqrt{3}})  
\end{array}
\right)  \ 
\left(
\begin{array}{c}
\gamma' \\
\gamma' \frac{L}{ \sqrt{3} } 
\end{array}
\right) 
= \left(
\begin{array}{c}0 \\
 -\beta i 
 \end{array}  
 \right) \cdot
\]
This yields
\ba
\gamma' &=& - i\beta \sin( \frac{L}{\sqrt{3} }  ), \label{F11}\\
\gamma' \frac{L}{\sqrt{3}} &=& -i \beta  \cos (\frac{L}{\sqrt{3} } )  \cdot \label{F12}
\ea
If $\gamma '=0$, then we infer as above that 
$(\theta ^0, u^0)=(0,0)$. \\
Assume that $\gamma '\ne 0$, so that $\beta \ne 0$. Without loss of generality, we can assume 
by linearity that $\beta =1$. 
It follows from \eqref{F11} and \eqref{F12} that 
\be
\label{F13}
\gamma ' = \pm \frac{i}{\sqrt{1+\frac{L^2}{3}}} \cdot
\ee
Let $X:=\cos (L/\sqrt{3})$. Then  we infer from \eqref{F12} and \eqref{F13} that $L^2/(3X^2)=1+(L^2/3)$, so that 
\[
X=\cos(\frac{L}{\sqrt{3}}) = \cos (\frac{1}{\sqrt{X^{-2}-1}} )\cdot
\]
On the other hand, from \eqref{F1} and \eqref{F5}, we obtain that $\cos(L/\sqrt{3})=-2\cos(2L/\sqrt{3})$, and hence $X$ satisfies
$4X^2+X-2=0$.  We denote the two roots of this equation by
\[
X_+:=\frac{-1+\sqrt{33}}{8}\approx 0.5931, \qquad X_-:= \frac{-1-\sqrt{33}}{8}\approx -0.8431.
\]  
Then we obtain by using a numerical computation that 
\[
\cos (\frac{1}{\sqrt{X_+^{-2}-1}} )\approx 0.7408\ne 0.5931, \qquad \cos (\frac{1}{\sqrt{X_-^{-2}-1}} )\approx 0.0032\ne -0.8431.
\] 
This proves that the assumption $\gamma '\ne 0$ leads to a contradiction. 

\noindent
3. The polynomial function $Q$ cannot have a triple root $\mu_0=\mu_1=\mu_2$, otherwise $0=Q''(\mu_0)=6\mu _0$ yields
$\mu _0=0$, $p=0$, and $Q(\xi ) =\xi (\xi -1)(\xi +1)$, a contradiction. 
\end{proof}

We now turn our attention to cases 2 and 3. We need the following estimate, whose (long) proof is postponed in an appendix. 

\begin{theorem}
\label{thm100}
For all $L\in (0,+\infty) \setminus 2\pi {\mathbb Z} $ and all $T>0$, there is a constant $C=C(L,T)>0$ such that for all $(\theta ^0, u^0)\in X_2$, we have
\ba
\label{GG1} 
\Vert (\theta ^0,u^0)\Vert ^2 _{X_2} \le C \left( \int _0^T  | \theta _{xx} (t,L)|^2  \, dt + \Vert (\theta ^0, u^0)\Vert _{X_0}^2\right), \\
\label{GG2}
\Vert (\theta ^0,u^0)\Vert ^2 _{X_2} \le C \left( \int _0^T  | \theta _{xx} (t,0)|^2  \, dt + \Vert (\theta ^0, u^0)\Vert _{X_0}^2\right). 
\ea
\end{theorem}

We first consider case 2. 
\begin{corollary}
\label{cor10}
For all $L\in (0, + \infty ) \setminus  (\cN\cup \cR )$ and all $T>0$, there exists a constant $C=C(L,T)>0$ such that 
\be
\label{GG3} 
\Vert (\theta ^0,u^0)\Vert ^2 _{X_2} \le C \int _0^T  | \theta _{xx} (t,L)|^2  \, dt  .
\ee 
\end{corollary}
\noindent
{\em Proof of Corollary  \ref{cor10}:} 
We proceed  in two steps as in the proof of Theorem \ref{thm2}.  \\

\noindent
{\sc Step 1. Reduction to a spectral problem.}\\[3mm]
If \eqref{GG3} is false, then one can pick a sequence $(\theta ^0_n, u^0_n)_{n\in \N}$ such that
\be
\label{GG4} 
1= \Vert (\theta ^0_n,u^0_n)\Vert ^2 _{X_2} > n  \int _0^T   | \theta _{n,xx} (t,L)|^2  \, dt. 
\ee
Extracting a subsequence if needed, we can assume that $(\theta ^0_n, u^0_n)\to (\theta ^0, u^0)$ weakly in $X_2$. 
Denote by $(\theta _n, u_n)$ (resp. $(\theta , u)$) the solution of \eqref{A51} with initial data $(\theta ^0_n , u^0_n)$ (resp. 
$(\theta ^0, u^0)$).  
Extracting again a subsequence, we can assume that $(\theta _n, u_n)\to (\theta, u)$ in $X_1$. By \eqref{GG1} and \eqref{GG4}, we
have that $(\theta ^0_n, u^0_n)\to (\theta ^0, u^0)$ strongly in $X_2$. Thus $\Vert (\theta ^0, u^0)\Vert_{X_2}=1$ and 
$\theta _{xx}(\cdot , L)=0$ a.e. in $(0,T)$. The same argument as the one in Step 1 of the proof of Theorem \ref{thm2} shows
that there exists 
 $\lambda \in \C$  and $(\theta ,u)\in X_3\setminus \{ (0,0) \}$ (superscript $0$ dropped for simplicity) solution of the spectral problem
\ba
-u'-u'''=\lambda \theta,&&  x\in (0,L), \label{C101}\\
-\theta ' -\theta '''=\lambda u,&& x\in (0,L), \label{C102}\\
\theta (0)=\theta (L)=\theta '(0)=\theta ''(L)=0,\label{C103}   \\
u(0)=u(L)=u'(L)=0.&& \label{C104}
\ea
We show in the next step that \eqref{C101}-\eqref{C104} has no nontrivial solution for any 
$L\in (0, + \infty ) \setminus  (\cN\cup \cR )$.\\

\noindent
{\sc Step 2. Study of the spectral problem.} 
\begin{proposition}
\label{prop88}
Let $L>0$. Then there exist $\lambda \in \C$  and $(\theta ,u)\in X_3\setminus \{ (0,0) \}$
solution of \eqref{C101}-\eqref{C104} if and only if $L\in \cN\cup \cR$. 
\end{proposition}
\begin{proof}
For $\theta$ and $u\in H^3(0,L)$, we still denote by $\theta$ and $u$ their extension by 0 on $\R$. 
Then
\[
u'''={\bf 1}_{(0,L)}u''' +  u''(0)\delta _0 - u''(L)\delta _L + u'(0)\delta _0' - u'(L) \delta _L'  +u(0)\delta _0''-u(L)\delta _L''\quad \textrm{ in } 
{\mathcal S} ' (\R ),
\]
where the derivatives of $u$ at $0$ or $L$ stand for the derivatives of $u_{\vert (0,L)}$ viewed as traces and ${\bf 1}_{(0,L)}$ denotes the characteristic function 
of the interval $(0,L)$. 
Assume that
\be
\label{C200}
\exists \lambda\in \C, \ \exists (\theta , u)\in X_3\setminus \{ (0,0)\}  \quad \textrm{ such that \eqref{C101}-\eqref{C104} hold}. 
\ee
Note that \eqref{C101}-\eqref{C102} yields
\begin{eqnarray}
\lambda \theta + u'+u''' &=&u''(0)\delta _0 - u''(L)\delta _L + u'(0)\delta _0' \quad
\textrm{ in } {\mathcal S} ' (\R ),    \label{C201} \\
\lambda u + \theta '+\theta ''' &=&\theta ''(0)\delta _0  - \theta '(L) \delta _L  '    \quad \textrm{ in } {\mathcal S} ' (\R ). \label{C202}
\end{eqnarray}

Conversely, if a pair $(\theta, u)$ in $L^2(\R ) ^2$ with $\textrm{Supp } \theta\cup \textrm{Supp }u\subset [0,L]$
satisfies \eqref{C201}-\eqref{C202} for some $\lambda\in\C$ and some complex coefficients
(instead of $u''(0), u''(L),$ etc.) in the r.h.s. of \eqref{C201}-\eqref{C202}, then $\theta , u\in H^3(0,L)$  and
 \eqref{C101}-\eqref{C104} hold.  

Let $\hat\theta (\xi )=\int_\R \theta (x) e^{-i\xi x} dx$  denote the Fourier transform of $\theta$. 
Introduce $ f (\xi) :=\hat \theta (\xi ) + \hat u(\xi )$ and $g(\xi ) := \hat \theta (\xi ) -\hat u (\xi )$. Then \eqref{C201}-\eqref{C202}
give
\ba
\lambda \hat \theta (\xi ) + (i\xi + (i\xi )^3) \hat u (\xi ) &=& u'(0) i\xi + u''(0) -u''(L)e^{-i L\xi }, \label{C221}\\
\lambda \hat u (\xi ) + (i\xi + (i\xi )^3 ) \hat \theta (\xi ) &=&  -i\xi \theta '(L) e^{-i L \xi } + \theta ''(0), \label{C222}  
\ea
and hence 
\ba
f(\xi ) &=& \frac{1}{-i\xi ^3 +i\xi +\lambda} \left(u''(0) +\theta ''(0) +u'(0) i\xi -u''(L) e^{-i L \xi}  -i\xi \theta ' (L) e^{-iL\xi}   \right),\\
g(\xi ) &=&   \frac{1}{i\xi ^3 - i\xi +\lambda} \left(u''(0) -\theta ''(0) +u'(0) i\xi -u''(L) e^{-i L \xi}  + i\xi \theta ' (L) e^{-iL\xi}   \right) .
\ea
Set $\lambda =-ip$. 
By Paley-Wiener theorem, we conclude that \eqref{C200} is equivalent to the existence of numbers
$p\in \C$ and 
$(\alpha ,\alpha ', \beta, \gamma , \gamma ' )\in \C ^5 \setminus \{ (0,0,0,0,0) \} $ such that the two functions defined for
$\xi\in \R$ by
\ba
f(\xi ) &=& \frac{i}{\xi ^3 -\xi +p} \left( \alpha  + \beta  i\xi +\gamma  e^{-i L \xi}  +\gamma '(i\xi)  e^{-iL\xi}   \right), \label{C241} \\
g(-\xi ) &=&   \frac{i}{\xi ^3 - \xi +p} \left( \alpha '  - \beta i\xi  +\gamma  e^{i L \xi}  + \gamma ' (i\xi)  e^{iL\xi}   \right) \label{C242}  
\ea
fulfill the conditions
\ba
&&f\textrm{ and } g \textrm{ are entire}; \label{C301}\\
&& (f,g)\in L^2(\R )^2 ; \label{C302} \\
&&\exists (C,N)\in (0,+\infty )\times \N , \ \forall \xi \in \C \quad | f(\xi ) | + |g(\xi ) |  \le C (1+ |\xi |)^N e^{L | \textrm{Im } \xi | } .\label{C303}
\ea
(We have set $\alpha : = u''(0) +\theta ''(0)$, $\alpha ' := u''(0) -\theta ''(0)$, $\beta := u'(0) $, $\gamma := -u''(L)$ and
$\gamma ' := -\theta '(L)$.) It is clear that for $f$ and $g$ given by \eqref{C241}-\eqref{C242}, \eqref{C301} implies both \eqref{C302} and \eqref{C303} (with $N=1$). 
Clearly, \eqref{C301} holds if and only if the two following functions 
\ba
f(\xi ) +  g ( -\xi  )&=& \frac{i}{\xi ^3 -\xi +p} \left( \alpha  + \alpha ' + 2 (\gamma  +\gamma 'i\xi) \cos (L\xi )    \right),\label{C401}\\
f(\xi ) -  g ( -\xi ) &=& \frac{i}{\xi ^3 -\xi +p} \left( \alpha  -\alpha ' + 2 \beta  i\xi   - 2i (\gamma + \gamma ' i\xi )  \sin (L\xi )    \right)  \label{C402}
\ea
are entire. 
Let $Q(\xi )=\xi ^3 -\xi +p$ and let $\mu _0,\mu _1, \mu _2$ denote its roots. The polynomial function $Q$ cannot have a triple root (see the proof of Corollary \ref{cor2}), so either 
the roots $\mu_k$, $k=0,1,2$, are simple, or there  are a double root $\mu_0=\mu_1$ and a simple root $\mu_2\ne \mu_0$.  \\
1. Assume that the roots $\mu_0, \mu _1,\mu _2$ of $Q$ are simple. If $\xi\in \{ \mu _0, \mu _1, \mu_2\}$, then from 
the fact that the functions in \eqref{C401}-\eqref{C402} are entire, we infer that   
\ba
2(\gamma + \gamma 'i\xi) \cos (L\xi ) &=& -\alpha -\alpha ', \label{C403}\\
2(\gamma + \gamma ' i\xi ) \sin (L\xi ) &=& \frac{\alpha -\alpha '}{i} +2\beta \xi , \label{C404}
\ea
and hence 
\[
4(\gamma + \gamma 'i\xi )^2 = (\alpha +\alpha ')^2 -(\alpha -\alpha ')^2 + 4\beta ^2 \xi ^2 + 4\frac{\beta \xi}{i}(\alpha -\alpha '). 
\]
The polynomial functions in the two sides of the above equation are of degree two and they take the same values on the numbers $\mu_k$, $k=0,1,2$. Therefore, they must have the same coefficients; that is, 
\ba
\beta ^2 &=& -\gamma '^2, \label{D1}\\
\alpha \alpha ' &=& \gamma ^2, \label{D2}\\
\beta (\alpha -\alpha ') &=&  -2\gamma \gamma' . \label{D3} 
\ea
(a) Assume that $\beta =i\gamma '$, so that \eqref{D3} becomes
\be
\label{D4}
i\gamma ' (\alpha -\alpha ') =-2\gamma \gamma '. 
\ee
(i) If $\gamma '=0$, then $\beta=0$  and \eqref{C241}-\eqref{C242} read 
\ba
f(\xi ) &=& \frac{i}{\xi ^3 -\xi +p} \left( \alpha  +\gamma  e^{-i L \xi}   \right), \label{D11} \\
g(-\xi ) &=&   \frac{i}{\xi ^3 - \xi +p} \left( \alpha '   +\gamma  e^{i L \xi}    \right) .\label{D12}  
\ea
It follows from Lemma \ref{lemR1} that there exist
$p\in \C$ and  $(\alpha,  \gamma)\in \C ^2\setminus \{ (0,0)\}$ such that the function
 $f$ defined in  \eqref{D11} is entire if and only if $L\in\cN$. If it is the case,  then the roots $\mu _0, \mu _1, \mu _2$ of $Q$ are
 simple and real, and $g(-\xi)$ defined in \eqref{D12} is also an entire function if we pick $\alpha' =
 \overline{\alpha} \gamma  / \overline{\gamma}$.  \\
 (ii) Assume now that $\gamma '\ne 0$. From \eqref{D2} and \eqref{D4}, we infer that 
 \[
 \alpha '=- i\gamma , \quad \alpha =i \gamma . 
 \] 
Then \eqref{C401}-\eqref{C402} become
\ba
f(\xi ) +  g ( -\xi  )&=& \frac{2i}{\xi ^3 -\xi +p}  (\gamma  +\gamma 'i\xi) \cos (L\xi )   , \label{D61bis}\\
f(\xi ) -  g ( -\xi ) &=& -\frac{2}{\xi ^3 -\xi +p} (\gamma + \gamma ' i\xi ) (1- \sin (L\xi ) ).\label{D62bis}   
\ea
$\bullet$ Assume that $-\frac{\gamma}{i\gamma ' }\not\in \{ \mu _0, \mu _1, \mu _2\}$, then each root $\mu _j$ of $Q$ should also be a root of 
$1-\sin (L\xi )$ by \eqref{D62bis}, and hence it could be written as
\[
L\mu _j =  \frac{\pi}{2} + 2k_j\pi , \quad k_j\in \Z .
\]
We arrive to the conclusion that 
\[
0=L(\mu _0+\mu_1 +\mu _2) = \pi (\frac{3}{2} + 2(k_0+k_1+k_2)),
\]
which is impossible, for $2(k_0+k_1+k_2)\in 2\Z$. \\
$\bullet$ Assume that $-\frac{\gamma}{i\gamma ' }\in \{ \mu _0, \mu _1, \mu _2\}$, say $\mu_0=-\gamma /(i\gamma ')$. Then 
both $f(\xi ) +  g ( -\xi  )$ and $f(\xi ) -  g ( -\xi  )$ are entire if and only if the other roots $\mu _1$ and  $\mu _2$ of $Q$ satisfy 
both equations 
\[
\cos(L\xi)=0 \textrm{ and } 1-\sin(L\xi )=0. 
\]
We arrive to the system
\ba
\mu_0 &=& -\frac{\gamma }{i\gamma '} , \label{E1}\\
L\mu _1 &=& \frac{\pi}{2} + 2k_1\pi, \quad k_1\in \Z ,\label{E2} \\
L\mu _2 &=& \frac{\pi}{2} + 2k_2\pi, \quad k_2\in \Z . \label{E3}
\ea

With those expressions, $\mu_0,\mu _1, \mu_2$ are roots of $Q$ if and only if
\[
\xi ^3 - (\mu _0+\mu _1+\mu _2)\xi ^2 + (\mu _0\mu_1  + \mu _0\mu _2 + \mu _1\mu _2) \xi  - \mu _0\mu _1\mu _2 =\xi ^3-\xi+p,\quad 
\forall \xi \in \C ,   
\] 
or
\ba
\mu _0+ \mu _1 +\mu _2 &=& 0, \label{E4}\\
\mu _0\mu _1 + \mu _0\mu_2 + \mu _1\mu _2 &=& -1 , \label{E5} \\
-\mu_0\mu_1\mu_2 &=& p. \label{E6} 
\ea
 
 We note that $\mu_0$ (and hence $\gamma$, if we pick $\gamma '=1$) is defined in terms of $k_1$ and $k_2$ by \eqref{E2}-\eqref{E4}, while
 $p$ is defined in terms of $k_1$ and $k_2$ by \eqref{E2}-\eqref{E4} and \eqref{E6}.  Replacing $\mu_0$ by $-(\mu _1+\mu _2)$ in 
 \eqref{E5} results in 
 \be
 \label{E7}
 \mu ^2_1+ \mu_1\mu_2 + \mu_2^2 =1. 
 \ee
Substituting the expressions of $\mu_1$ and $\mu _2$ given in \eqref{E2}-\eqref{E3} in \eqref{E7}, we obtain the critical length
\be
\label{E8}
L=\pi \left( (\frac{1}{2} + 2k_1)^2 + (\frac{1}{2} + 2k_2)^2 + (\frac{1}{2} + 2k_1)(\frac{1}{2} + 2k_2) \right) ^\frac{1}{2}.
\ee
Finally, the roots $\mu_0,\mu_1,\mu_2$ are pairwise distinct if and only if $k_1\ne k_2$. Indeed, if $k_1\ne k_2$, it is clear that $\mu _1\ne\mu_2$, while 
\[
\mu_0 = -\frac{\pi}{L} (1+2(k_1+k_2)) \ne \mu_j=\frac{\pi}{L} (\frac{1}{2} + 2k_j), \quad j=1,2. 
\] 
(b) Assume now that $\beta =-i\gamma '$. The analysis is completely similar to the one in case (a), so that we only sketch the main facts. 
Here $i\gamma '(\alpha-\alpha ')=2\gamma \gamma '$.\\
 If $\gamma '=0$, we obtain again that
the existence of a nontrivial solution $p\in\C$, $(\alpha, \alpha ', \gamma )\in \C ^3 \setminus \{ (0,0,0) \}$  of the spectral 
problem is equivalent to $L\in \cN$. \\
If $\gamma ' \ne 0$, then we obtain that $\alpha '=i\gamma$, $\alpha =-i\gamma$, and
\begin{eqnarray*}
f(\xi ) +  g ( -\xi  )&=& \frac{2i}{\xi ^3 -\xi +p}  (\gamma  +\gamma 'i\xi) \cos (L\xi )   , \label{D61}\\
f(\xi ) -  g ( -\xi ) &=& \frac{2}{\xi ^3 -\xi +p} (\gamma + \gamma ' i\xi ) (1 + \sin (L\xi ) ).\label{D62}   
\end{eqnarray*}
As in case (a), we can see that the assumption  $-\frac{\gamma}{i\gamma ' }\not\in \{ \mu _0, \mu _1, \mu _2\}$ 
leads to a contradiction. 
Assume thus that   $-\frac{\gamma}{i\gamma ' }\in \{ \mu _0, \mu _1, \mu _2\}$, say $\mu_0=-\gamma /(i\gamma ')$. 
The other roots $\mu_1$ and $\mu_2$ should solve the equations
\[
\cos (L\xi )=0, \quad \sin (L\xi) =-1. 
\]
Writing 
\[
L\mu_1= -\frac{\pi}{2} + 2k_1\pi,\ L\mu _2 = -\frac{\pi}{2} + 2k_2\pi,\qquad k_1\ne k_2\in \Z, 
\]
we obtain as critical length
\[
L=\pi \left( ( 2k_1 -\frac{1}{2} )^2 + (2k_2 -\frac{1}{2} )^2 + (2k_1 -\frac{1}{2} )(2k_2 -\frac{1}{2} ) \right) ^\frac{1}{2},
\]
and we notice that $L\in \cR$. \\
\noindent
2. Assume that $Q$ has a double root $\mu_0=\mu_1$ and a simple root $\mu_2\ne \mu_0$.  Then 
$(\mu_0,\mu_2)=\pm (\frac{1}{\sqrt{3}}, -\frac{2}{\sqrt{3}})$. We shall consider the case  
$(\mu_0,\mu_2)= (\frac{1}{\sqrt{3}}, -\frac{2}{\sqrt{3}})$, the case $(\mu_0,\mu_2)= (-\frac{1}{\sqrt{3}}, \frac{2}{\sqrt{3}})$
being similar. As $\mu _0$ is a double root of $Q$, it should be root of the numerators of the functions $f( \xi ) + g (\xi )$ 
and $f(\xi ) -g (-\xi )$ (see
\eqref{C401} and \eqref{C402}) and of their first derivatives, so that 
\ba
2(\gamma + \gamma'  \frac{i}{\sqrt{3}}) \cos (\frac{L}{\sqrt{3}} ) &=& -(\alpha + \alpha '), \label{F21}\\
2i(\gamma + \gamma ' \frac{i}{\sqrt{3}} ) \sin (\frac{L}{\sqrt{3}}) &=& \alpha -\alpha ' 
+2\frac{\beta i }{\sqrt{3}} ,\label{F22}\\
2\gamma 'i \cos(\frac{L}{\sqrt{3}}) +2(\gamma + i \frac{\gamma '}{\sqrt{3}} ) (-L\sin \frac{L}{\sqrt{3}}) &=& 0, \label{F23}\\ 
2\beta i + 2  \gamma'  \sin (\frac{L}{\sqrt{3}}) - 2 i (\gamma  + i \frac{\gamma '}{\sqrt{3}} )  L \cos (\frac{L}{\sqrt{3}}) &=&0 .\label{F24}
\ea 
As $\mu _2$ is a simple root of $Q$, it has to be a root of the numerators of the functions $f(\xi ) +g(- \xi)$ and  $f(\xi ) -g( -\xi )$, so that 
\ba
2(\gamma - \gamma '  \frac{2i}{\sqrt{3}} )  \cos (\frac{2L}{\sqrt{3}}) &=& -(\alpha + \alpha '), \label{F25}\\
2i (\gamma - \gamma ' \frac{2i}{\sqrt{3}}) \sin (-2 \frac{L}{\sqrt{3}}) &=& \alpha -\alpha ' 
-4\frac{\beta i}{\sqrt{3}} .\label{F26}
\ea
Note that \eqref{F23}-\eqref{F24} can be rewritten as 
\[
\left( 
\begin{array}{cc}
\cos (\frac{L}{\sqrt{3}})  & - \sin  (\frac{L}{\sqrt{3}})  \\[3mm]
\sin (\frac{L}{\sqrt{3}}) & \cos (\frac{L}{\sqrt{3}})  
\end{array}
\right)  \ 
\left(
\begin{array}{c}
i \gamma' \\
L (\gamma + i \frac{\gamma '}{ \sqrt{3} } ) 
\end{array}
\right) 
= \left(
\begin{array}{c}0 \\
 \beta  
 \end{array}  
 \right)  \cdot
\]
Picking $\beta =1$, this yields
\ba
i\gamma' &=&  \sin( \frac{L}{\sqrt{3} }  ), \label{F31}\\
L( \gamma +  i\frac{\gamma '}{\sqrt{3} } ) &=& \cos (\frac{L}{\sqrt{3} } )  \cdot \label{F32}
\ea
In particular, $(i\gamma',\gamma )\in\R ^2$. From \eqref{F21} and \eqref{F25}, we infer that 
\[
(\gamma + \gamma'  \frac{i}{\sqrt{3}}) \cos (\frac{L}{\sqrt{3}} )
= (\gamma - \gamma '  \frac{2i}{\sqrt{3}} )  \cos (\frac{2L}{\sqrt{3}}).
\]
Combined with \eqref{F31} and \eqref{F32}, this yields
\be
\label{F33}
\cos ^2 ( \frac{L}{\sqrt{3}} )=\left( \cos (\frac{L}{\sqrt{3}} ) 
 - \sqrt{3}L \sin ( \frac{L}{\sqrt{3}} )  \right) \cos ( \frac{2L}{\sqrt{3}} ).  
\ee
As the set of solutions of \eqref{F33} cannot have a limit point, we conclude that on any segment $[0,R]$ there are at most finitely many $L$ satisfying \eqref{F33}. 

 From \eqref{F22} and \eqref{F26}, we infer that 
\[
2i (\gamma + \gamma ' \frac{i}{\sqrt{3}} ) \sin (\frac{L}{\sqrt{3}}) 
- 2\frac{\beta i}{\sqrt{3}}  =
2i (\gamma - \gamma ' \frac{2i}{\sqrt{3}}) \sin (-2 \frac{L}{\sqrt{3}}) 
+ 4\frac{\beta i}{\sqrt{3}} .
\]
Combined with \eqref{F31} and \eqref{F32}, this yields (with $\beta =1$)
\be
\label{F34}
L^{-1} \cos  ( \frac{L}{\sqrt{3}} ) \sin  ( \frac{L}{\sqrt{3}} )
=\left( L^{-1}  \cos (\frac{L}{\sqrt{3}}  )
 - \sqrt{3} \sin ( \frac{L}{\sqrt{3}} )  \right) \sin( \frac{-2L}{\sqrt{3}} ) +\sqrt{3}.  
\ee
We claim that  
\[
\cos (\frac{L}{\sqrt{3}} ) \sin ( \frac{L}{\sqrt{3}} ) \ne 0. 
\]
Indeed, otherwise we would have either $\cos ( \frac{L}{\sqrt{3}} )=0$ and then $\sin (\frac{L}{\sqrt{3}} ) =\pm 1$, so that $L=0$  by \eqref{F33},
which is impossible,  or  $\sin ( \frac{L}{\sqrt{3}} ) =0$ which is impossible by \eqref{F34}.  

We infer from \eqref{F33} and \eqref{F34} that 
\be
\label{F35}
\cos  ( \frac{L}{\sqrt{3}} ) \sin  ( \frac{L}{\sqrt{3}} )
= \frac{\cos ^2 (\frac{L}{\sqrt {3}} )  }{ \cos  (\frac{2L}{\sqrt {3}} )}
\sin( \frac{-2L}{\sqrt{3}} ) +\sqrt{3} L.  
\ee
On the other hand, it follows from \eqref{F33} that 
\be
\label{F36}
\sqrt{3} L = \frac{1}{\sin (\frac{L}{\sqrt {3}} )} \left( \cos   (\frac{L}{\sqrt {3}} ) - \frac{\cos ^2 (\frac{L}{\sqrt {3}} )}{\cos  (\frac{2L}{\sqrt {3}} )} \right)  . 
\ee
Replacing the last term in \eqref{F35} by its expression in \eqref{F36}, multiplying the equation by 
$\cos ( \frac{ 2L}{\sqrt {3}} )    \sin  (\frac{L}{\sqrt {3}} )$ and simplifying by $\cos ( \frac{ L}{\sqrt {3}} )$, we arrive to 
\[
\sin ^2  ( \frac{ L}{\sqrt {3}} ) \cos  ( \frac{ 2L}{\sqrt {3}} ) =-2 \cos ^2  ( \frac{ L}{\sqrt {3}} ) \sin ^2  ( \frac{ L}{\sqrt {3}} ) 
+ \cos  ( \frac{ 2L}{\sqrt {3}} )  - \cos  ( \frac{ L}{\sqrt {3}} ) ,  
\]  
or 
\[
\sin ^2  ( \frac{ L}{\sqrt {3}} ) \big( \cos  ( \frac{ 2L}{\sqrt {3}} ) + 2 \cos ^2  ( \frac{ L}{\sqrt {3}} ) \big)  
=-2  \sin  ( \frac{ 3L}{2\sqrt {3}} )   \sin  ( \frac{ L}{2\sqrt {3}} ) \cdot
\]  
Letting  $x: =\frac{L}{2\sqrt{3}}$, we obtain that 
\[ \sin ^2 (2x) (4\cos ^2 (2x) -1) = -2\sin (3x)\sin x = -2(3\sin x - 4\sin ^3 x) \sin x.  \]
Simplifying in the above equation by $\sin ^2 x$ (which is different from 0 since $\sin (L/\sqrt {3})$ is), we arrive to
\[
4 (1-\sin ^2 x) (3-16 \sin ^2 x + 16 \sin ^4 x) = -2(3-4\sin ^2 x). 
\]
Setting $y:= \sin ^2 x$, we obtain the polynomial equation 
\[
32 y^3 -64 y^2 +42 y-9=0
\]
which has $y=3/4$ as a double root and $y=1/2$ as a single root.  
But $y=1/2$ would give $\cos (L/\sqrt{3})=\cos (2x) = 2\cos  ^2 x -1=0$, which is impossible, and $y=3/4$ 
would give $\cos (L/\sqrt{3})=1-2y=-1/2$, $\cos (2L/\sqrt{3})=-1/2$, $\sin (L/\sqrt{3})=\pm \sqrt{3}/2$. Plugging those values in 
\eqref{F33} would give 
\[
\frac{1}{4} = \left( -\frac{1}{2} -\sqrt{3}L \cdot (\pm \frac{\sqrt{3}}{2}) \right) \cdot (- \frac{1}{2})
\] 
and hence $L=0$, which is impossible. 
The proof of Proposition \ref{prop88} and of Corollary \ref{cor10} is complete. 
\end{proof}

Next, we investigate the spectral problems associated with cases 10 and 11. 
\begin{corollary}
\label{cor3}
Let $L>0$. Then there exist $\lambda \in \C$  and $(\theta ,u)\in X_3\setminus \{ (0,0) \}$
solution of \eqref{C101}-\eqref{C104}  together with   $u''(0)=0$ if and only if $L\in \cR$. 
\end{corollary}

\begin{proof}
We use the same notations as in the proof of Proposition \ref{prop88}. The additional condition $u''(0)=0$ is equivalent to 
the condition $\alpha' =-\alpha$. When $\gamma '\ne 0$, we obtained that $(\alpha, \alpha ') =(i\gamma, -i\gamma)$ for
$\beta =i\gamma'$ (case (a)), and that $(\alpha, \alpha ') =(-i\gamma, i\gamma)$ for
$\beta = - i\gamma'$ (case (b)). Thus the condition $\alpha' =-\alpha$ is automatically satisfied for the critical lengths $L\in \cR$. 

Consider now the case $\gamma '=0$ with $\beta =\pm i\gamma ' =0$ corresponding to the critical length $L\in \cN$.  
From \cite[Lemma 3.5]{R1}, we know that the three roots $\mu _0$, $\mu _1$, $\mu _2$ of $Q$ can be written as
\be
\label{G1}
\mu _0=-\frac{1}{3} (2k+l)\frac{2\pi}{L},\quad  \mu_1= \mu _0+ k\frac{2\pi}{L}, \quad  \mu _2= \mu_1 + l\frac{2\pi}{L}, 
\ee
for some $k,l\in \N^*$. From \eqref{C401}, we know that the $\mu _j$ have to be zeros of the function $\gamma \cos (L\xi)$. 
If $\gamma=0$, then we infer from \eqref{C241}-\eqref{C242} that $0=\alpha =\alpha '$, and hence $(\theta ^0,u^0)=(0,0)$.
If $\gamma \ne 0$, from $\cos(L\mu_0)=0$, we infer the existence of some $p\in\Z$ such that 
$L\mu _0=\pi (\frac{1}{2}+p)$, so that 
$\frac{1}{2} + p = -2(2k+l)/3$. This gives $3/2=-3p-2(2k+l)$ where $k,l,p\in\Z$, which is impossible. 
\end{proof}

\begin{corollary}
\label{cor4}
Let $L>0$. Then there exist $\lambda \in \C$  and $(\theta ,u)\in X_3\setminus \{ (0,0) \}$
solution of \eqref{C101}-\eqref{C104}  together with   $\theta ''(0)=0$ if and only if $L\in \cN _3$. 
\end{corollary}

\begin{proof}
The additional condition $\theta ''(0)=0$ is equivalent to  $\alpha '=\alpha$. \\
If $\gamma ' \ne 0$, then we have that 
$\alpha '=- \alpha =\pm i \gamma$, so that $\alpha=\alpha '=\gamma =0$.  
If $\beta = i\gamma '$ and $0=-\frac{\gamma}{i\gamma '} \in \{ \mu _0, \mu  _1, \mu _2\}$, say $\mu_0=0$, then by \eqref{E1}-\eqref{E3} 
\[
0 = - L \mu _0 = L( \mu _1 + \mu _2 ) =  \pi (1 + 2(k_1+k_2))
\] 
with $k_1,k_2\in \Z$, which is impossible. If $\beta =  - i\gamma '$ and $0=-\frac{\gamma}{i\gamma '}  = \mu _0$, we arrive to 
\[
0 = - L \mu _0 = L( \mu _1 + \mu _2 ) =  \pi ( - 1 + 2(k_1+k_2))
\]
with  $k_1,k_2\in \Z$, which again is impossible.\\
Thus no $L\in \cR$ is a critical length for \eqref{C101}-\eqref{C104} 
with the additional condition $\theta ''(0)=0$.\\ 
If $\gamma '=0$, then $\beta =0$ by \eqref{D1}, and \eqref{D11}-\eqref{D12} hold with $\alpha '=\alpha$. For each root 
$\xi \in \{ \mu _0, \mu _1, \mu _2 \}$ of $Q$, we have that 
\[ 
\alpha + \gamma e^{-iL \xi} = 0 = \alpha + \gamma e^{iL\xi}. 
\]
This is equivalent to \eqref{G1} together with the the condition $2L\mu_j\in 2\pi \Z$ for $j=0,1,2$. 
With \eqref{G1}, the latter condition is equivalent to $2L \mu _0\in 2\pi \Z$, i.e. there exists $q\in\Z$ such that 
$ (-(2k+l)2\pi /(3L))2L =2\pi q$, or $3q=-2(2k+l)$. This is possible if and only if $3\vert 2k+l$, i.e. $L\in \cN _3$. 
\end{proof}

We now turn our attention to case 3. 
\begin{theorem}
\label{thm11}
For all $L\in (0, + \infty ) \setminus  (\cN\cup \cG \cup \cG ' )$ and all $T>0$, there exists a constant $C=C(L,T)>0$ such that 
\be
\label{GG10} 
\Vert (\theta ^0,u^0)\Vert ^2 _{X_2} \le C \int _0^T  | \theta _{xx} (t,0)|^2  \, dt  .
\ee 
\end{theorem}
\noindent
{\em Proof of Theorem  \ref{thm11}:} 
We proceed  in two steps as in the proof of Theorem \ref{thm2}.  \\

\noindent
{\sc Step 1. Reduction to a spectral problem.}\\[3mm]
If \eqref{GG10} is false, then one can pick a sequence $(\theta ^0_n, u^0_n)_{n\in \N}$ such that
\be
\label{GG11} 
1= \Vert (\theta ^0_n,u^0_n)\Vert ^2 _{X_2} > n  \int _0^T   | \theta _{n,xx} (t,0)|^2  \, dt. 
\ee
Extracting a subsequence if needed, we can assume that $(\theta ^0_n, u^0_n)\to (\theta ^0, u^0)$ weakly in $X_2$. 
Denote by $(\theta _n, u_n)$ (resp. $(\theta , u)$) the solution of \eqref{A51} with initial data $(\theta ^0_n , u^0_n)$ (resp. 
$(\theta ^0, u^0)$).  
Extracting again a subsequence, we can assume that $(\theta _n^0, u_n^0)\to (\theta ^0, u^0)$ in $X_1$. By \eqref{GG2} and \eqref{GG11}, we
have that $(\theta ^0_n, u^0_n)\to (\theta ^0, u^0)$ strongly in $X_2$. Thus $\Vert (\theta ^0, u^0)\Vert_{X_2}=1$ and 
$\theta _{xx}(\cdot , 0)=0$ a.e. in $(0,T)$. The same argument as the one in Step 1 of the proof of Theorem \ref{thm2} shows
that there exist
 $\lambda \in \C$  and $(\theta ,u)\in X_3\setminus \{ (0,0) \}$ (superscript $0$ dropped for simplicity) which are solution of the spectral problem:
\ba
-u'-u'''=\lambda \theta,&&  x\in (0,L), \label{W101}\\
-\theta ' -\theta '''=\lambda u,&& x\in (0,L), \label{W102}\\
\theta (0)=\theta (L)=\theta '(0)=\theta ''(0)=0,\label{W103}   \\
u(0)=u(L)=u'(L)=0.&& \label{W104}
\ea

\begin{proposition}
\label{prop8}
Let $L>0$. Then there exist $\lambda \in \C$  and $(\theta ,u)\in X_3\setminus \{ (0,0) \}$
solution of \eqref{W101}-\eqref{W104} if and only if $L\in \cN\cup \cG \cup \cG '$. 
\end{proposition}

\begin{proof}
We use the same notations as those in  the proof of Proposition \ref{prop88}.  
Consider the property:
\be
\label{W200}
\exists \lambda\in \C, \ \exists (\theta , u)\in X_3\setminus \{ (0,0)\}, \quad \textrm{ \eqref{W101}-\eqref{W104} hold}. 
\ee
Note that \eqref{W101}-\eqref{W102} yields
\begin{eqnarray}
\lambda \theta + u'+u''' &=&u''(0)\delta _0 - u''(L)\delta _L + u'(0)\delta _0' \quad
\textrm{ in } {\mathcal S} ' (\R ),    \label{W201} \\
\lambda u + \theta '+\theta ''' &=& - \theta ''(L)\delta _L  - \theta '(L) \delta _L  '    \quad \textrm{ in } {\mathcal S} ' (\R ). \label{W202}
\end{eqnarray}

Conversely, if a pair $(\theta, u)$ in $L^2(\R ) ^2$ with $\textrm{Supp } \theta\cup \textrm{Supp }u\subset [0,L]$  satisfies \eqref{W201}-\eqref{W202} for some $\lambda\in\C$ and some complex coefficients
(instead of $u''(0), u''(L),$ etc.) in the r.h.s. of \eqref{W201}-\eqref{W202}, then $\theta , u\in H^3(0,L)$ and 
\eqref{W101}-\eqref{W104} hold.  

Let $ f (\xi) :=\hat \theta (\xi ) + \hat u(\xi )$ and $g(\xi ) := \hat \theta (\xi ) -\hat u (\xi )$. Then \eqref{W201}-\eqref{W202}
yield
\ba
\lambda \hat \theta (\xi ) + (i\xi + (i\xi )^3) \hat u (\xi ) &=&  u''(0) -u''(L)e^{-i L\xi }  + u'(0) i\xi , \label{W221}\\
\lambda \hat u (\xi ) + (i\xi + (i\xi )^3 ) \hat \theta (\xi ) &=&  - \theta '' (L) e^{-iL\xi }  -i\xi \theta '(L) e^{-i L \xi } , \label{W222}  
\ea
and hence 
\ba
f(\xi ) &=& \frac{1}{-i\xi ^3 +i\xi +\lambda} \left(
u''(0)  - (u''(L) + \theta ''(L) ) e^{-iL\xi}   + u'(0) i\xi   -i\xi \theta ' (L) e^{-iL\xi}   \right),\\
g(\xi ) &=&   \frac{1}{i\xi ^3 - i \xi +\lambda} 
\left( u''(0)  +(-u''(L) +\theta ''( L) ) e^{-iL\xi }  + u'(0)i\xi  +  i\xi \theta ' (L) e^{-iL\xi}   \right) .
\ea
Set $\lambda =-ip$. 
By Paley-Wiener theorem, we conclude that \eqref{W200} is equivalent to the existence of numbers
$p\in \C$ and 
$(\alpha , \beta, \gamma _1, \gamma _2 ,  \gamma ' )\in \C ^5 \setminus \{ (0,0,0,0,0) \} $ such that the two functions defined for
$\xi\in \R$ by
\ba
f(\xi ) &=& \frac{i}{\xi ^3 -\xi +p} \left( \alpha  + \beta  i\xi +( \gamma _1 + \gamma _2 )  e^{-i L \xi}  +\gamma '(i\xi)  e^{-iL\xi}   \right), \label{W241} \\
g(-\xi ) &=&   \frac{i}{\xi ^3 - \xi +p} \left( \alpha   - \beta i\xi  + ( \gamma _1 - \gamma _2)  e^{i L \xi}  + \gamma ' (i\xi)  e^{iL\xi}   \right) \label{W242}  
\ea
fulfill the conditions
\ba
&&f\textrm{ and } g \textrm{ are entire}; \label{W301}\\
&& (f,g)\in L^2(\R )^2 ; \label{W302} \\
&&\exists (C,N)\in (0,+\infty )\times \N ,  \ \forall \xi \in \C \quad | f(\xi ) | + |g(\xi ) |  \le C (1+ |\xi |)^N e^{L | \textrm{Im } \xi | } .\label{W303}
\ea
(We have set $\alpha : = u''(0)$, $\beta := u'(0) $, $\gamma _1 := - u''(L)$, $\gamma _2 := -\theta '' ( L)$ and
$\gamma ' := -\theta '(L)$.) It is clear that for $f$ and $g$ as in \eqref{W241}-\eqref{W242}, the condition \eqref{W301} implies both \eqref{W302} and \eqref{W303} (with $N=1$). 
Clearly, \eqref{W301} holds if and only if the two following functions 
\ba
f(\xi ) +  g ( -\xi  )&=& \frac{i}{\xi ^3 -\xi +p} \left( 2 \alpha  - 2 \gamma _2 i \sin (L\xi ) + 2 (\gamma _1 +\gamma ' i \xi) \cos (L\xi )    \right),\label{W401}\\
f(\xi ) -  g ( -\xi ) &=& \frac{i}{\xi ^3 -\xi +p} \left( 2 \beta  i\xi  +2\gamma _2 \cos (L\xi )  - 2i (\gamma _1 + \gamma ' i\xi )  \sin (L\xi )    \right)  \label{W402}
\ea
are entire. 
Let $Q(\xi )=\xi ^3 -\xi +p$ and let $\mu _0,\mu _1, \mu _2$ denote the roots of $Q$. The polynomial function $Q$ cannot have a triple root, so either 
the roots $\mu_k$, $k=0,1,2$, are simple, or there  are a double root $\mu_0=\mu_1$ and a simple root $\mu_2\ne \mu_0$.  \\

\noindent
1. Assume that $Q$ has three simple roots. Let $\xi$ be a root of $Q$. Setting 
\[
a:=\gamma _1 + \gamma ' i\xi, \quad b := i \gamma _2, 
\]
we infer from the fact that the numerators of the  r.h.s. of (\ref{W401}) and (\ref{W402}) vanish at $\xi$ that 
\[
\left( \begin{array}{cc} a& -b \\ b& a \end{array} \right) \, 
\left( \begin{array}{c} \cos L\xi \\ \sin L\xi \end{array} \right)
=
\left(
\begin{array}{c} 
-\alpha \\
\beta \xi
\end{array} 
 \right)  \cdot 
\] 
This yields 
\[
\cos (L\xi) = \frac{-\alpha a + b\beta \xi}{a^2+b^2}, \quad \sin (L\xi ) = \frac{a\beta \xi + \alpha b}{a^2+b^2}\cdot 
\]
Expanding in the identity $\cos ^2(L\xi ) + \sin ^2 (L\xi )=1$ gives $a^2+b^2=\alpha ^2 + \beta ^2 \xi ^2$, i.e. 
\[
-(\gamma ' )^2 \xi ^2 + 2\gamma  _1 \gamma ' i \xi + \gamma _1^2 -\gamma  _2^2 = \alpha ^2 + \beta ^2 \xi ^2. 
\]
Since $Q$ has three different roots, we arrive at 
\begin{eqnarray}
-(\gamma ')^2 &=& \beta ^2 , \\
\gamma _1 \gamma '&=& 0, \\
\gamma _1^2 - \gamma _2 ^2 &=& \alpha ^2. 
\end{eqnarray}
a. If $\gamma '=0$, then $\beta =0$, and the existence of $p\in \C$ and of  $(\alpha, \gamma _1, \gamma _2)\in 
\C ^3\setminus \{ (0,0,0) \}$  such that both $f(\xi)$ and $g(-\xi)$ are analytic in $\C$ is equivalent to the fact that
$L\in {\mathcal N}$, by Lemma \ref{lemR1}.  \\
b. If $\gamma ' \ne 0$, then $\gamma _1=0$, $\beta =\pm i\gamma ' \ne 0$ and $\alpha =\pm i\gamma _2$. \\[3mm]
{\sc Case 1. } $\beta = i\gamma '$ and $\alpha = i\gamma _2$. \\
Then 
\[
f(\xi ) = \frac{i}{\xi ^3 -\xi +p} (i+e^{-iL\xi})(\gamma _2 + \gamma ' (i\xi) ), \quad 
g(-\xi ) = \frac{i}{\xi ^3 - \xi +p}  (i-e^{iL\xi})   (\gamma _2 -\gamma ' (i\xi)).   
\] 
If, each root $\mu _j$ ($0\le j\le 2$) of $Q$ is such that  $\gamma _2 + i\mu _j \gamma '\ne 0$, then 
\[
e^{-iL\mu _j} = -i= e ^{-i\frac{\pi}{2}},
\]
and hence $L\mu _j= \frac{\pi}{2} + 2k_j\pi$, for some $k_j\in \Z$. This contradicts the fact that  $\mu _0+\mu _1+\mu _2=0$.
Therefore, there exists some root $\mu_j$ of $Q$, say $\mu _0$, such that $\gamma _2 + i\mu _0\gamma '=0$.   In 
 the same way, we can prove that some root $\mu_k$ of $Q$ should satisfy $\gamma _2 -i\mu _k \gamma '=0$. 
 If $\mu _k=\mu _0$, then $\gamma _2=\mu _0=0$, and hence $p=0$ and $\{ \mu _1, \mu_2 \} =\{ \pm 1 \}$. 
 Then we should have $i+e^{- i L\mu _1}= i+e^{ -  i L \mu _2}=0$, i.e $e^{i L}=e^{-  i L}= -i$, and this is impossible. 
 If $\mu _k\ne \mu_0$, say $\mu_k =\mu _1$, then  $\mu _1 = \gamma _2 /(i\gamma ' ) = -\mu _0$, and thus 
 $\mu_2 =-(\mu _0+ \mu _1)=0$. But then $i+ e^{-iL\mu_2}=i\ne 0$, and $\mu _2$ is a pole of $f$, which yields a contradiction. \\
 
\noindent 
{\sc Case 2. } $\beta = -i\gamma '$ and $\alpha = -i\gamma _2$. \\
It is similar to Case 1 and is impossible. \\

\noindent
{\sc Case 3. } $\beta = i\gamma '$ and $\alpha = -i\gamma _2$. \\
Then 
\ba
f(\xi ) &=& \frac{i}{\xi ^3 -\xi +p} \left( -i\gamma _2  +  i\gamma ' (i\xi ) +  \gamma _2   e^{-i L \xi}  +\gamma '(i\xi)  e^{-iL\xi}   \right), \label{W501} \\
g(-\xi ) &=&   \frac{i}{\xi ^3 - \xi +p} \left( -i\gamma _2    - i\gamma ' (  i\xi )   - \gamma _2  e^{i L \xi}  + \gamma ' (i\xi)  e^{iL\xi}   \right)  \cdot \label{W502}  
\ea

\noindent
If $\gamma _2=0$, then 
\[
f (\xi ) + g(-\xi ) = \frac{-2 \gamma ' } {\xi ^3 -\xi + p} \xi \cos (L\xi ), \quad 
f(\xi ) -g(-\xi ) = -\frac{2\gamma '  i \xi }{\xi ^3 -\xi +p }(1-\sin  (L\xi )). 
\]
If $p\ne 0$, then $0\not\in \{\mu _0, \mu _1,\mu _2\}$ and $\cos (L\mu _j)=0$ yields $L\mu _j =(\pi /2)+ 2 k_j\pi$ 
with $k_j\in \Z$ for $k=0,1,2$, contradicting  the fact that $\mu_0 + \mu _1 + \mu_2=0$. \\
If $p=0$, then  we have (after relabeling) $\mu _0=0$, $\mu _1=1$ and $\mu _2=-1$. Then 
both $\mu_1$ and $\mu _2=-\mu _1$ should solve $1-\sin (L\xi)=0$, which is impossible. We conclude that 
$\gamma _2 \ne 0$ when $\beta = i\gamma ' \ne 0$ and $\alpha = -i\gamma _2$. 

Dividing in  (\ref{W501})-(\ref{W502}) by $\gamma'$, 
we can assume that $\gamma '=1$. Each root $\xi$ of $Q$ should satisfy
\ba
\gamma _2 (-i + e^{-iL \xi } ) + i\xi (i+ e^{-iL\xi }) &=&0, \label{WW1}\\
\gamma _2 (-i-e^{iL\xi}) +  i \xi (-i+ e^{iL\xi})  &=& 0. \label{WW2}
\ea
Then system (\ref{WW1})-(\ref{WW2}) is equivalent to
\ba
e^{-iL\xi } &=& \frac{i\gamma _2 +\xi}{\gamma _2 + i\xi},  \label{WW3}\\
e^{iL\xi}   &=& \frac{i\gamma _2 - \xi}{-\gamma _2 + i\xi} \cdot \label{WW4}
\ea
(Note that both $\gamma _2+i\xi $ and $-\gamma _2 +i\xi $ are different from $0$, otherwise we would infer from 
\eqref{WW1}-\eqref{WW2}
 that $\gamma _2=0$.) We notice that (\ref{WW3}) and (\ref{WW4}) are equivalent, so that we can focus on (\ref{WW4}). Using basic algebra, we have that 
\begin{eqnarray*}
e^{iL\xi} =\frac{i\gamma _2-\xi}{-\gamma _2 +i\xi} 
&\iff& ie^{iL\xi } = 1+\frac{2i\xi}{\gamma _2 -i\xi} \\
&\iff & \frac{2i\xi}{ie^{iL\xi} -1} +i\xi =\gamma _2 =Const. 
\end{eqnarray*}
Thus, the existence of $p\in \C$, $\gamma _2\in \C ^*$ and $\mu_0,\mu_1, \mu_2$ such that (\ref{E4})-(\ref{E6}) hold together with 
\[
\frac{2i\mu _0}{ie^{iL\mu _0} -1} + i\mu _0 = 
\frac{2i\mu _1}{ie^{iL\mu _1} -1} + i\mu _1 =
\frac{2i\mu _2}{ie^{iL\mu _2} -1} + i\mu _2 
=\gamma _2
\] 
is reduced, by letting $a:=iL\mu _0$, $b:=iL\mu _1$, to the existence of $a,b\in \C$ such that
\begin{eqnarray*}
\frac{2a}{ie^a-1} + a = \frac{2b}{ie^b -1} + b = \frac{ 2(-a-b)}{ie ^{-a-b} -1} -a-b\ne 0, \\
L^2= -(a^2+ab+b^2).
\end{eqnarray*}
Thus, this is possible if and only if $L\in \cG$.  \\

\noindent
{\sc Case 4. } $\beta = -i\gamma '$ and $\alpha = i\gamma _2$. \\
It is similar to Case 3. We find, instead of (\ref{WW3})-(\ref{WW4}), the system
\begin{eqnarray*}
-e^{-iL\xi } &=& \frac{i\gamma _2 +\xi}{\gamma _2 + i\xi},  \\
-e^{iL\xi}   &=& \frac{i\gamma _2 - \xi}{-\gamma _2 + i\xi}
\end{eqnarray*}
which can be reduced to its second equation equivalent to
\[
\frac{2i\xi}{-ie^{iL\xi} -1} +i\xi =\gamma _2 =Const. 
\]
Introducing again $a:=iL\mu _0$, $b=iL\mu _1$, we see that the existence of $p\in \C$, $\gamma _2\in \C ^*$ 
and $\mu_0,\mu_1, \mu_2$ (the roots of $Q$) holds if and only if $L\in \cG '$. \\

\noindent
2. Assume that $Q$ has a double root $\mu_0=\mu_1$ and a simple root $\mu_2\ne \mu_0$.  Then 
$(\mu_0,\mu_2)=\pm (\frac{1}{\sqrt{3}}, -\frac{2}{\sqrt{3}})$. We shall consider the case  
$(\mu_0,\mu_2)= (\frac{1}{\sqrt{3}}, -\frac{2}{\sqrt{3}})$, the case $(\mu_0,\mu_2)= (-\frac{1}{\sqrt{3}}, \frac{2}{\sqrt{3}})$
being similar. As $\mu _0$ is a double root of $Q$, it should be a root of the numerators of the functions $f( \xi ) + g  ( - \xi )$ 
and $f(\xi ) -g (-\xi )$ (see
\eqref{W401} and \eqref{W402}) and of their first derivatives, so that 
\ba
-\gamma _2 i\sin( \frac{L}{\sqrt{3}}) +
 (\gamma _1 + \gamma'  \frac{i}{\sqrt{3}}) \cos (\frac{L}{\sqrt{3}} ) &=& - \alpha , \label{WW21}\\
\gamma _2 \cos (\frac{L}{\sqrt{3}}) 
-i (\gamma _1 + \gamma ' \frac{i}{\sqrt{3}} ) \sin (\frac{L}{\sqrt{3}}) &=& -\beta \frac{i}{\sqrt{3}}  ,\label{WW22}\\
(-\gamma _2 iL +\gamma 'i ) \cos(\frac{L}{\sqrt{3}})  - L(\gamma _1 +  \frac{ i\gamma '}{\sqrt{3}} ) \sin (\frac{L}{\sqrt{3}}) &=& 0, \label{WW23}\\ 
 (\gamma '  -   \gamma _2 L )  \sin (\frac{L}{\sqrt{3}}) 
 - i (\gamma _1 + \frac{i\gamma '}{\sqrt{3}} )  L \cos (\frac{L}{\sqrt{3}}) &=& -\beta i .\label{WW24}
\ea 
As $\mu _2$ is a simple root of $Q$, it has to be a root of the numerators of the functions $f(\xi ) +g(- \xi)$ and  $f(\xi ) -g( -\xi )$, so that 
\ba
-\gamma _2 i \sin (-\frac{2L}{\sqrt{3}})  + (\gamma _1- \gamma '  \frac{2i}{\sqrt{3}} )  \cos (\frac{2L}{\sqrt{3}}) &=& -\alpha , \label{WW25}\\
 \gamma _2 \cos (\frac{2L}{\sqrt{3}} ) 
-i (\gamma _1 -  \gamma ' \frac{2i}{\sqrt{3} } ) \sin (-2 \frac{L}{\sqrt{3}}) &=&  
2\frac{\beta i}{\sqrt{3}} .\label{WW26}
\ea
Note that \eqref{WW23}-\eqref{WW24} can be rewritten as 
\[
\left( 
\begin{array}{cc}
\cos (\frac{L}{\sqrt{3}})  & - \sin  (\frac{L}{\sqrt{3}})  \\[3mm]
\sin (\frac{L}{\sqrt{3}}) & \cos (\frac{L}{\sqrt{3}})  
\end{array}
\right)  \ 
\left(
\begin{array}{c}
-\gamma _2 iL +  i  \gamma'   \\
L (\gamma _1 +   i \frac{ \gamma '}{ \sqrt{3} } ) 
\end{array}
\right) 
= \left(
\begin{array}{c}0 \\ 
\beta  
 \end{array}  
 \right)  \cdot
\]
If $\beta=0$, then we obtain $-\gamma _2 iL+ i\gamma '=\gamma _1 + i\frac{\gamma '}{\sqrt{3}}=0$, and 
\eqref{WW22} gives $\gamma _2 \cos(L/\sqrt{3})=0$. 
Then either $\cos (L/\sqrt{3} )\ne 0$ and $\gamma _2=0$, or $\cos ( L/ \sqrt{3} ) =0$ and  \eqref{WW26} 
gives again $\gamma _2=0$.  
Thus $0=\gamma _2=\gamma'=\gamma _1$, and \eqref{WW21} yields $\alpha=0$, so that 
$(\theta , u)=(0,0)$, which is impossible. 

Assume now that $\beta\ne 0$. Picking $\beta =1$, we obtain 
\ba
-\gamma _2 iL + i\gamma' &=&  \sin( \frac{L}{\sqrt{3} }  ), \label{WW31}\\
L( \gamma  _ 1 +  i\frac{\gamma '}{\sqrt{3} } ) &=& \cos (\frac{L}{\sqrt{3} } )  \cdot \label{WW32}
\ea
We can compute all the coefficients in \eqref{WW21}-\eqref{WW26} in terms of $L$. Indeed, we infer from 
\eqref{WW22} and \eqref{WW32} that 
$\cos (L/ \sqrt{3})\ne 0$ and 
\[
\gamma _2 =\frac{i}{L} \sin (\frac{L}{\sqrt{3}}) -\frac{i}{\sqrt{3} \cos (\frac{L}{\sqrt{3} } )} \cdot
\]
Combined with \eqref{WW31}, this yields 
\[
\gamma ' = L\gamma _2 -i\sin ( \frac{L}{ \sqrt{3}} ) = -\frac{iL}{\sqrt{3} \cos (\frac{L}{\sqrt{3}})} \cdot
\]
It follows that 
\[
\gamma _1-\gamma ' \frac{2i}{\sqrt{3}} = (\gamma _1 + \frac{i\gamma '}{\sqrt{3}})-\sqrt{3} i \gamma ' 
= \frac{1}{L} \cos ( \frac{L}{ \sqrt{3}} ) - \frac{ L }{ \cos ( \frac{L}{\sqrt{3} } ) }\cdot
\]
From \eqref{WW21} and \eqref{WW25}, we infer that 
\be
\label{WW100}
-\gamma _2 i\sin( \frac{L}{\sqrt{3}}) +
 (\gamma _1 + \gamma'  \frac{i}{\sqrt{3}}) \cos (\frac{L}{\sqrt{3}} ) 
 =\gamma _2 i \sin (\frac{2L}{\sqrt{3}})  + (\gamma _1- \gamma '  \frac{2i}{\sqrt{3}} )  \cos (\frac{2L}{\sqrt{3}}) .
\ee
Substituting the values of $\gamma_2$, $\gamma _1 + \gamma'  \frac{i}{\sqrt{3}} $ and 
$\gamma _1 - \gamma'  \frac{2i}{\sqrt{3}}$ in \eqref{WW100}, we obtain the equation 
\begin{multline*}
i\left( \frac{i}{L} \sin (\frac{L}{\sqrt{3}})  - \frac{i}{\sqrt{3} \cos (\frac{L}{\sqrt{3}})}  \right)\cdot 
\left( \sin (\frac{L}{\sqrt{3}} ) + \sin (\frac{2L}{\sqrt{3}}) \right) + \left( \frac{1}{L} \cos(\frac{L}{\sqrt{3}}) 
-\frac{L}{\cos (\frac{L}{\sqrt{3}})} \right) \cos (\frac{2L}{\sqrt{3}}) \\
-\frac{1}{L} \cos ^2 (\frac{L}{\sqrt{3}})=0.   
\end{multline*}
Simplifying the above equation, we arrive to 
\[
\underbrace{\frac{1}{L} \cos (\frac{L}{\sqrt{3}}) \left( \cos (\frac{L}{\sqrt{3}}) (1-4\sin ^2(\frac{L}{\sqrt{3}})) -1\right) 
-L\cos(\frac{2L}{\sqrt{3}}) + \frac{1}{\sqrt{3}} \left( \sin (\frac{L}{\sqrt{3}}) + \sin (\frac{2L}{\sqrt{3}}) \right)}_{h(L)} =0. 
\]

From \eqref{WW22} and \eqref{WW26}, we infer that 
\be
\label{WW101}
-2\gamma _2 \cos (\frac{L}{\sqrt{3}}) +
2i (\gamma _1 + \gamma'  \frac{i}{\sqrt{3}}) \sin (\frac{L}{\sqrt{3}} ) 
 =\gamma _2 \cos (\frac{2L}{\sqrt{3}})  + i(\gamma _1- \gamma '  \frac{2i}{\sqrt{3}} )  \sin (\frac{2L}{\sqrt{3}}) .
\ee
Substituting the values of $\gamma_2$, $\gamma _1 + \gamma'  \frac{i}{\sqrt{3}} $ and 
$\gamma _1 - \gamma'  \frac{2i}{\sqrt{3}}$ in \eqref{WW101}, we obtain the equation 
\begin{multline*}
\left( \frac{i}{L} \sin (\frac{L}{\sqrt{3}})  - \frac{i}{\sqrt{3} \cos (\frac{L}{\sqrt{3}})}  \right)\cdot 
\left( \cos (\frac{2L}{\sqrt{3}} ) + 2\cos (\frac{L}{\sqrt{3}}) \right) + i  \left( \frac{1}{L} \cos(\frac{L}{\sqrt{3}}) 
-\frac{L}{\cos (\frac{L}{\sqrt{3}})} \right) \sin (\frac{2L}{\sqrt{3}}) \\
-\frac{2i}{L} \cos (\frac{L}{\sqrt{3}}) \sin (\frac{L}{\sqrt{3}}) =0.   
\end{multline*}
After some simplifications, we arrive to 
\[
\underbrace{(\frac{1}{2L} +L)  \sin (\frac{2L}{\sqrt{3}}) + 
\frac{1}{\sqrt{3}} \left( \cos (\frac{2L}{\sqrt{3}}) + 2 \cos (\frac{L}{\sqrt{3}}) \right)}_{k(L)} =0. 
\]
Let $\zeta (L):=h(L)^2+k(L)^2$. Then $\zeta (L)=L^2+O(L)$ as $L\to +\infty$, 
so that the equation $\zeta (L)=0$ has no root $L\gg 1$. 
Using a numerical computation,  we see that $\inf_{L>0} \zeta(L)\approx 5.3333>0$. Thus
the system 
\[
\left\{ 
\begin{array}{l}
h(L)=0, \\
k(L)=0
\end{array}
\right.
\] 
has no solution $L>0$.  Therefore, there is no critical length in the case of a double
root.  The proof of Proposition \ref{prop8} and of Theorem \ref{thm11} is complete. 
\end{proof}

Next, we investigate the spectral problem associated with case 12.
\begin{corollary}
Let $L>0$. Then there exist $\lambda \in \C$  and $(\theta ,u)\in X_3\setminus \{ (0,0) \}$
solution of \eqref{W101}-\eqref{W104} together with $u''(L)=0$  if and only if $L\in  \cG \cup \cG '$. 
\end{corollary}
\begin{proof}  We use the same notations as in the proof of Proposition \ref{prop8}. The extra condition 
$u''(L)=0$ means that $\gamma _1=0$. As for Proposition \ref{prop8}, the case when $Q$ has multiple roots is impossible.
We therefore assume that $Q$ has three different roots.\\
a. If $\gamma '=0$, then $\beta =0$ and for each root $\xi$ of $Q$, we have that 
\[\alpha +\gamma _2 e^{-iL\xi}=0= \alpha -\gamma _2 e^{iL\xi} .\]   
We assume  $\gamma _2\ne 0$ (otherwise all the coefficients are zero), and we arrive to $e^{2iL\xi}=-1$ and $2L\xi=\pi +2\pi k$, $k\in \Z$. Since the sum of the three
roots of $Q$ is $0$, we arrive to a contradiction. \\
b. If $\gamma '\ne 0$, then we have (still) $\gamma _1=0$, $\beta =\pm i\gamma '\ne 0$ and $\alpha =\pm i\gamma _2$, a case
already considered in the proof of Proposition \ref{prop8}. The cases 1 and 2 have to be excluded, while the cases 3 and 4 are valid
and they lead to $L\in \cG \cup \cG '$. 
\end{proof}

\section{Proof of Theorems \ref{THMINTRO1} and \ref{THMINTRO2}.}
We shall first investigate the wellposedness of  system \eqref{INTRO1} 
when picking $g_2(t)=-\alpha \eta _x (t,L)$ as control input, with $\alpha >0$. We first notice that a {\em global Kato smoothing 
effect holds}, thanks to which we can derive the wellposedness of the system in the energy space
$X_0=[L^2(0,L)]^2$.  

Let us first pay our attention to the linearized system
\be
\label{M1}
\left\{ 
\begin{array}{ll}
\eta _t + v_x + v_{xxx} =0, \quad &  t\in (0,T) , \ x\in (0,L), \\
v_t + \eta _x + \eta _{xxx} =0, &  t\in (0,T) , \ x\in (0,L), \\
\eta (t,0)=0, \ \eta (t,L)=0,  \ \eta _x(t,0)=0, &t\in (0,T), \\
v(t,0)=0, \ v(t,L)=0, \ v_x(t,L) = - \alpha \eta _x(t,L), &t\in (0,T) , \\
\eta (0,x)=\eta ^0(x),\ v(0,x)=v^0(x), &x\in (0,L). 
\end{array}\right.
\ee

We introduce the operator
\[
\tilde A (\eta , v) :=(-v_x-v_{xxx}, -\eta _x-\eta _{xxx})
\]
with domain 
\[
D(\tilde A):= 
\{ (\eta, v)\in H^3(0,L)^2; \  \eta (0)=\eta (L) = \eta _x(0)=v(0)=v(L)=0, \ 
v_x(L)=-\alpha \eta _x(L) \} \subset X_0.
\] 
\begin{proposition}
\label{prop1bis}
Assume that $\alpha >0$. Then the operator $\tilde A$ generates a semigroup of contractions $(e^{t\tilde A})_{t\ge 0}$  in $X_0$.  Furthermore, 
for any $T>0$ and any $(\eta ^0,v^0)\in X_0$,  the solution $(\eta , v):=e^{t\tilde A} (\eta ^0, v^0)$ 
of \eqref{M1} belongs to  $L^2(0,T, [H^1(0,L)]^2)$ and  we have
\be
\label{M2}
\int_0^T\!\!\! \int_0^L (\eta _x^2 + v_x^2 ) \, dx dt \le 
\frac{2}{3} \left( L + \frac{T}{2} + \frac{L(\alpha ^2 +1)}{4\alpha} \right) \Vert (\eta ^0, v^0)\Vert ^2_{X_0}.   
\ee
\end{proposition}
\begin{proof}
It is clear that $D(\tilde A)$ is dense in $X_0$ and that $\tilde A$ is closed. 
Let $(\eta , v) \in D(\tilde A)$. Then we readily obtain that 
\[
(\tilde A (\eta , v), (\eta , v) )_{X_0} = -\alpha \eta _x (L)^2\le 0,
\]
so that  $\tilde A$ is a dissipative operator. 
Introduce the operator $B(\theta , u) := (u_x+u_{xxx}, \theta _x + \theta _{xxx})$ with domain 
\[
D(B):= 
\{ (\theta , u)\in H^3(0,L)^2; \  \theta (0)=\theta (L) = \theta _x(0)=u(0)=u(L)=0, \ 
u_x(L)=\alpha \theta _x(L) \} \subset X_0. \]

We easily obtain that $(B(\theta  , u), ( \theta , u))_{X_0} =-\alpha \theta _x (L)^2\le 0$, so that $B$ is also 
a dissipative operator. 
We claim that $B=\tilde A^*$. First, for all $( \eta , v ) \in D(\tilde A)$ and all $(\theta , u) \in D(B)$, we have by integration by parts that 
\[
(\tilde A (\eta , v), (\theta , u))_{X_0} = ((\eta , v), B(\theta , u))_{X_0}, 
\]
which shows that $B\subset \tilde A ^*$.

Let us now check that $\tilde A^*\subset B$. Pick any $(\theta , u ) \in D(\tilde A^*)$. Then,  we 
have for some constant $C>0$ that
\[
\left\vert ( (\theta , u ) ,\tilde  A(\eta , v))_{X_0}
\right\vert \le C \Vert (\eta , v) \Vert _{X_0}
\quad \forall (\eta , v) \in D(\tilde A), 
\] 
i.e.
\be
\label{M2bis}
\left\vert \int _0^L  [ \theta (v_x+v_{xxx} ) + u (\eta _x + \eta _{xxx}) ]dx \right\vert 
  \le C \left(  \int_0^L [\eta ^2 + v^2 ]dx\right)^\frac{1}{2}
\quad \forall (\eta , v)\in D(\tilde A). 
\ee
Picking $v=0$ and $\eta \in C_c^\infty (0,L)$, we infer from \eqref{A2} that $u_x+u_{xxx}\in L^2(0,L)$, and hence that 
$u\in H^3(0,L)$. Similarly, we obtain that $\theta \in H^3(0,L)$. 
Integrating by parts in the left hand side  of \eqref{M2bis}, we obtain that 
\begin{eqnarray*}
&&\left\vert \theta (L) v_{xx}(L) -\theta (0)v_{xx}(0)  + \alpha \theta _x (L) \eta _x(L) +  \theta _x (0)v_x(0) 
+ u(L) \eta _{xx}(L)  -u(0) \eta _{xx} (0) -u_x(L) \eta _x(L) \right\vert \\
&& \qquad \le C\left( \int_0^L [\eta ^2 + v^2] dx \right)\qquad \forall (\eta , v) \in D(\tilde A) . 
\end{eqnarray*}
It follows that 
\[
\theta (0)=\theta (L)=\theta _x(0)=u(0)=u(L)=0,\quad u_x(L)=\alpha \theta _x(L),
\]
so that $(\theta , u)\in D(B)$ and  $\tilde A^*=B$. We conclude that the operator $A$ is $m-$dissipative in $X_0$ and that it generates a semigroup of contractions in $X_0$. 

Pick any $(\eta ^0, v^0 ) \in D(\tilde A )$ and let $(\eta ,v) (t) :=e^{t\tilde A}(\eta ^0, v^0)$. Scaling the two first equations   in \eqref{M1} by $\eta$ and $v$ respectively, and summing the two obtained equations we arrive to the energy identity
\be
\label{M5}
\left[ \frac{1}{2} \int_0^L [\eta ^2+v^2]dx \right] _0^T + \alpha \int_0^T \eta _x (t,L) ^2dt =0.  
\ee   
Scaling the two first equations in \eqref{M1} by $x v$ and $x \eta$, respectively, and summing the obtained equations, we arrive to 
\[
\left[ \int_0^L (x\eta v)dx \right]_0^T -\frac{1}{2}\int_0^T\!\!\!\int_0^L (\eta ^2 + v^2)dxdt 
+\frac{3}{2} \int_0^T\!\!\! \int_0^L (\eta _x^2 + v_x^2) dxdt -\frac{L}{2} (\alpha ^2 + 1)  \int_0^T \eta _x (t,L)^2 dt=0. 
\]
Combined with \eqref{M5}, it yields
\begin{eqnarray*}
\frac{3}{2} \int_0^T\!\!\! \int_0^L (\eta _x^2 + v_x^2 ) dx dt 
&=& -\left[ \int _0^L (x\eta v) dx\right]_0^T + \frac{1}{2} \int_0^T\!\!\!\int_0^L (\eta ^2+v^2) dxdt 
+\frac{L}{2} (\alpha ^2 +1) \int_0^T \eta _x (t,L)^2 dt \\
&\le& L\left(  \Vert \eta (T)\Vert _{L^2(0,L)} \Vert v(T) \Vert _{L^2(0,L)} + 
\Vert \eta ^0\Vert _{L^2(0,L)} \Vert v^0\Vert _{L^2(0,L)}  \right) \\
&&\quad +\left( \frac{T}{2} + \frac{L}{4\alpha } (\alpha ^2 +1)\right) 
(\Vert \eta ^0\Vert ^2_{L^2(0,L)} +  \Vert v^0\Vert ^2 _{L^2(0,L)}) \\
&\le & \left( L + \frac{T}{2} + \frac{L}{4\alpha } (\alpha ^2 +1) \right) \Vert (\eta ^0, v^0)\Vert ^2_{X_0},
\end{eqnarray*}
as desired. 
\end{proof}
We search a solution of 
\be
\label{M10}
\left\{ 
\begin{array}{ll}
\eta _t + v_x +(\eta v)_x +  v_{xxx} =0, \quad &  t\in (0,\infty)  , \ x\in (0,L), \\
v_t + \eta _x + vv_x + \eta _{xxx} =0, &  t\in (0,\infty ) , \ x\in (0,L), \\
\eta (t,0)=0, \ \eta (t,L)=0,  \ \eta _x(t,0)=0, &t\in (0,\infty ), \\
v(t,0)=0, \ v(t,L)=0, \ v_x(t,L) = - \alpha \eta _x(t,L), &t\in (0,\infty) , \\
\eta (0,x)=\eta ^0(x),\ v(0,x)=v^0(x), &x\in (0,L), 
\end{array}\right.
\ee
as a  solution of the integral equation
\be
\label{M11}
(\eta , v) (t) =e^{t \tilde A } (\eta ^0, v^0)  - \int_0^t e^{ (t-\tau ) \tilde A } ((\eta v) _x (\tau ), (vv_x)(\tau ))d\tau, \quad t\ge 0.    
\ee
Introducing the map $\Gamma $ defined by
\[ 
\Gamma (\eta ,v ) (t) := e^{t \tilde A } (\eta ^0, v^0)  - \int_0^t e^{ (t-\tau ) \tilde A } ((\eta v) _x (\tau ), (vv_x)(\tau ))d\tau ,
\]
we will show that for any (sufficiently small) time $T>0$, the map $\Gamma$ has a unique fixed-point in some closed   
ball $B(0,R)$ in the space 
\[
E_T:= C([0,T], X_0) \cap L^2(0,T, [H^1(0,L)]^2)    
\]
endowed with the norm
\[
\Vert (\eta ,v ) \Vert _{E_T} := \sup _{t\in [0,T]} \Vert (\eta (t), v(t)) \Vert _{X_0} 
+ \left( \int_0^T \Vert (\eta _x(t), v _x(t) ) \Vert ^2 _{X_0}  \right) ^\frac{1}{2}. 
\]
Such a fixed-point yields a local solution of \eqref{M10}. 
\begin{proposition}
For any $(\eta ^0, v^0)\in X_0$, there are some positive numbers $T$ and $R$  such that the  system \eqref{M10} has a unique (integral) solution 
$(\eta , v) \in B(0,R) \subset E_T$. 
\end{proposition}
\begin{proof}
The proof is very similar to those of the wellposedness of the KdV equation with the boundary conditions 
$u(t,0)=u(t,L)=u_x(t,L)=0$ (see \cite{PMVZ,R1}). First, we can prove as in \cite{R1} that for some constant 
$C_1>0$, we have for all $T\in (0,1]$ and all $(f,g)\in L^1(0,T, X_0)$
\be
\Vert \int _0^t e^{ (t-\tau) \tilde A}  ( f(\tau ), g(\tau )) d\tau \Vert _{E_T} \le C_1 \Vert (f,g)\Vert _{L^1(0,T,X_0)}. 
\ee 
Following \cite{PMVZ}, one can find a constant $C(L)>0$ such that for all $T\in (0,1]$ and all 
$(\eta , v)\in E_T$, it holds
\ba
\int_0^T \Vert (\eta v)_x \Vert _{X_0} dt 
&\le& CT^\frac{1}{4}
\left( 
\Vert \eta _x\Vert _{L^2 (0,T,L^2(0,L) ) }  \Vert v \Vert ^\frac{1}{2}_{L^\infty (0,T,L^2(0,L))}
\Vert v_x \Vert ^\frac{1}{2}_{L^2 (0,T,L^2(0,L))}  \right.  \nonumber \\
&&    
\left. \qquad   +  \Vert v_x\Vert _{L^2 (0,T,L^2(0,L) ) }  \Vert \eta \Vert ^\frac{1}{2}_{L^\infty (0,T,L^2(0,L))}
\Vert \eta _x \Vert ^\frac{1}{2}_{L^2 (0,T,L^2(0,L))}
\right).
\ea
It follows that there are some positive constants $C_2(L),C_3(L)$ such that for all $T>0$ and all 
$(\eta ^1,v^1) , (\eta ^2,v^2) \in E_T$ we have 
\ba
\Vert \Gamma (\eta ^1, v^1) \Vert _{E_T} &\le& C_2\Vert (\eta  ^0,v ^0 )\Vert _{X_0} + C_3 T^\frac{1}{4} \Vert (\eta ^1, v^1)\Vert ^2_{E_T}, \\
\Vert \Gamma (\eta ^1, v^1) - \Gamma (\eta ^2, v^2)  \Vert _{E_T} 
&\le& C_3  T^\frac{1}{4} (\Vert (\eta ^1, v^1) \Vert _{E_T} + \Vert (\eta ^2, v^2) \Vert _{E_T}) 
 \Vert (\eta ^1-\eta ^2, v^1- v^2) \Vert _{E_T}. \qquad
\ea
Picking $R = 2 C_2 \Vert (\eta ^0, v^0)\Vert _{X_0} $ and $T>0$ such that $2C_3T^\frac{1}{4}R=1/2$, we 
see that the map $\Gamma$ is a contraction in the closed ball $B(0,R)$ of $E_T$, and hence that it has a unique
fixed point by the contraction mapping theorem. 
\end{proof}
Let us now proceed to the proof of Theorems \ref{THMINTRO1} and \ref{THMINTRO2}. Assume that $L\in (0,+\infty ) \setminus {\mathcal N}$. 
First, we show that the semigroup $(e^{ t \tilde A} )_{t\ge 0}$ is exponentially stable in $X_0$; that is, for some 
constants $C,\mu >0$ it holds
\be
\label{M51}
\Vert e^{t\tilde A} (\eta ^0, v^0) \Vert _{X_0} \le Ce^{-\mu t} \Vert (\eta ^0, v^0 ) \Vert _{X_0}\qquad \forall t\in \R_+. 
\ee
It is actually sufficient to prove that for some $T>0$ and some $C=C(T)>0$,  
\be
\label{M52}
\Vert  (\eta ^0, v^0) \Vert ^2_{X_0} \le C \int_0^T \eta _x (t,L)^2dt. 
\ee
Indeed, combining \eqref{M52} with  \eqref{M5}, we obtain $\Vert (\eta (T), v(T))\Vert ^2_{X_0} \le (1-2\alpha C^{-1})
\Vert (\eta ^0, v^0)\Vert ^2 _{X_0}$ which yields \eqref{M51} by the semigroup property. 

To prove \eqref{M52}, we proceed by contradiction. If \eqref{M52} is not true, one can find a sequence 
$(\eta ^{0,n}, v^{0,n})_{n\ge 0}$ in $X_0$ such that, denoting $(\eta ^n, v^n) (t) :=e^{t\tilde A} (\eta ^{0,n}, v^{0,n})$, we have 
\be
\label{M60}
1=\Vert (\eta ^{0,n}, v^{0,n})\Vert ^2_{X_0} > n \int_0^T \eta_x ^n (t,L)^2 dt. 
\ee
Scaling the two first equations  in \eqref{M1} by $ (T-t) \eta$ and $ (T-t) v$ respectively, and summing the two obtained equations we obtain 
\be
\label{M61}
\frac{T}{2} \Vert (\eta (0), v(0)) \Vert ^2_{X_0}  =\frac{1}{2}  \int_0^T\!\!\!\int_0^L (\eta ^2 +v^2)dxdt  + \alpha \int_0^T (T-t) \eta_x(t,L)^2dt. 
\ee
Since by  \eqref{M1}  and \eqref{M2} the sequence $(\eta ^n, v^n)$ is bounded in 
$L^2(0,T, [H^1_0(0,L)]^2)\cap$$ H^1(0,T, $ $[H^{-2} (0,L)]^2)$, we infer from Aubin-Lions' lemma
that a subsequence $(\eta ^{n_k}, v^{n_k})$ is convergent in $L^2(0,T, [L^2(0,L)]^2)$. Combined to 
\eqref{M60}-\eqref{M61}, this yields that $(\eta^{0,n_k}, v^{0,n_k})\to (\eta^0,v^0)$ strongly in $X_0$ for some 
$(\eta ^0,v^0)\in X_0$ satisfying  $\Vert (\eta ^0, v^0)\Vert _{X_0} =1$ and $\eta_x(\cdot ,  L)\equiv 0$. Taking into account
\eqref{M1}, we infer that $v_x(\cdot ,  L)\equiv 0$ as well.  But since $L\not\in \mathcal N$, this is impossible by Theorem
\ref{thm2}.

Thus \eqref{M52} is established. Proceeding in much the same way as for \cite[Theorem 1.1]{PR}, we can derive Theorem 
\ref{THMINTRO2}. (The proof is omitted for the sake of shortness.)

Let us now proceed to the proof of Theorem \ref{THMINTRO1}. We first notice that the linear system  
\be
\label{M70}
\left\{ 
\begin{array}{ll}
\eta _t + v_x +  v_{xxx} =0, \quad &  t\in (0,T)  , \ x\in (0,L), \\
v_t + \eta _x + \eta _{xxx} =0, &  t\in (0,T) , \ x\in (0,L), \\
\eta (t,0)=0, \ \eta (t,L)=0,  \ \eta _x(t,0)=0, &t\in (0,T), \\
v(t,0)=0, \ v(t,L)=0, \ v_x(t,L) = - \alpha \eta _x(t,L) +h(t), &t\in (0,T) , \\
\eta (0,x)=\eta ^0(x),\ v(0,x)=v^0(x), &x\in (0,L), 
\end{array}\right.
\ee
is well posed for $(\eta ^0, v^0)\in X_0$ and $h\in L^2(0,T)$. Clearly, the wellposedness 
can be derived for $h\in C^2 ([0,T])$ by performing the change of unknowns $\tilde \eta (t,x) :=\eta (t,x), \ \tilde v (t,L) := v(t,L) 
+ h(t) g(x)$, where the function $g\in C^\infty ([0,L])$ is such that $g(0)=g(L)=0$ and $g'(L)=-1$.  To extend the result from 
$C^2([0,T]) $ to $L^2(0,T)$, we need to derive some {\em a priori} estimates. 
Scaling in the first (resp. second) equation of \eqref{M70} by $\eta$ (resp. $v$), we obtain after some integrations by parts
\[
\left[ \frac{1}{2}  \int_0^L [ \eta ^2 + v^2]dx \right]_0^T + \alpha \int_0^T \eta _x (t,L)^2dt - \int_0^T \eta _x(t,L) h(t)dt =0.   
\] 
This yields for all $T\ge 0$
\be
\label{PPPP1}
\Vert (\eta (T, .), v(T,.)) \Vert _{X_0}^2 + \alpha \int_0^T \eta _x (t,L)^2 dt  \le  \Vert (\eta ^0, v^0)\Vert ^2_{X_0} +\frac{1}{\alpha} \int_0^T h(t)^2dt, 
\ee
and thus $(\eta, v)\in C([0,T], X_0)$ if $(\eta ^0, v^0)\in X_0$ and $h\in L^2(0,T)$. 
Scaling now in the first (resp. second) equation of \eqref{M70} by $xv$ (resp. $x\eta$) yields for some constant $C=C(L,T) >0$
\[
\int_0^T\!\!\!\int_0^L (\eta _x^2 + v_x^2) dxdt \le C \left( \Vert (\eta ^0, v^0)\Vert ^2_{X_0} + \int_0^T h(t)^2dt\right),
\]
so that $(\eta, v)\in E_T$ if  $(\eta ^0, v^0)\in X_0$ and $h\in L^2(0,T)$. 

The same computations as in the proof of Theorem \ref{THMINTRO2} show that the operator 
\[ \hat A (\eta , v) =(-v_x-v_{xxx}, -\eta _x -\eta _{xxx}) \]
with domain
\[
D(\hat A) =\{ (\eta, v)\in H^3(0,L)^2; \ \eta (0) = \eta (L)  = \eta _x(L) =v(0)=v(L) =0,\ \ v_x(0)=\alpha \eta _x(0)\} \subset X_0 
\]
generates a semigroup of contractions in $X_0$. Next, performing the change of variables $t\to T-t$ and $x\to L-x$, 
we infer that for any $(\theta ^1, u^1)\in X_0$, the backward system 
\be
\label{M100}
\left\{ 
\begin{array}{ll}
\theta _t + u_x +  u_{xxx} =0, \quad &  t\in (0,T)  , \ x\in (0,L), \\
u_t + \theta _x + \theta _{xxx} =0, &  t\in (0,T) , \ x\in (0,L), \\
\theta (t,0)=0, \ \theta (t,L)=0,  \ \theta _x(t,0)=0, &t\in (0,T), \\
u(t,0)=0, \ u(t,L)=0, \ u_x(t,L) = \alpha \theta _x(t,L), &t\in (0,T), \\
\theta (T,x)=\theta ^1(x),\ u(T,x)=u^1(x), &x\in (0,L), 
\end{array}\right.
\ee
has a unique solution $(\theta ,u)\in C([0,T], X_0)$, which belongs to $C([0,T],[H^3(0,L)]^2)\cap C^1([0,T], X_0)$ if 
$(\theta ^1(L-\cdot), u^1(L-\cdot))\in D(\hat A)$. 
 
We now prove the exact controllability in $X_0$ of the linear system  \eqref{M70}. Scaling in the first (resp. second) equation of \eqref{M70} by $\theta $ (resp. $u$), we obtain after some integrations by parts
\[
\left[ \int_0^L [\eta \theta + vu]dx  \right]_0^T = \int_0^T \theta _x (t,L) h(t)dt.  
\]
By the Hilbert Uniqueness Method (see \cite{lions1, lions2}), the exact controllability of  \eqref{M70} 
holds in $X_0$ with control inputs $h\in L^2(0,T)$ if and only if the following observability inequality holds
\be
\label{M140}
\Vert (\theta ^1, u^1)\Vert ^2_{X_0} \le C \int_0^T  | \theta  _x(t,L)|^2 dt 
\ee
for the solutions of the {\em backward} system \eqref{M100}. 
Again the backward system enjoys the {\em global Kato smoothing property}: 
\[
\Vert (\theta _x, u_x)\Vert _{ L^2(0,T,X_0) } \le C  \Vert (\theta ^1, u^1)\Vert_{X_0}.
\] 
Proceeding as in \cite{R1}, 
we easily see that if \eqref{M140} is false, then we can find a pair of data $(\theta ^1,u^1)$ in $X_0$ such that
$\Vert (\theta ^1, u^1)\Vert _{X_0}=1$ and $\theta _x(\cdot , L)\equiv 0$. 
But this is  impossible 
by Theorem \ref{thm2}, for $L\not\in \mathcal N$. 

We have established the exact controllability of the linear system \eqref{M70}. The (local) exact controllability of the 
nonlinear system \eqref{INTRO1} in $X_0$ follows at once by applying the contraction mapping theorem as in \cite{R1} for KdV. (Note that $\eta _x(\cdot , L)\in L^2(0,T)$ by \eqref{PPPP1}, and hence $g_2=-\alpha \eta _x(\cdot, L)+h\in L^2(0,T)$ 
as well.)
We omit the details for the sake of shortness.
The proof of Theorem \ref{THMINTRO1} is complete.  

\section*{Appendix: Proof of Theorem \ref{thm100}.}

The proof of Theorem \ref{thm100} is not based on the multiplier method, but on the analysis of the spectral properties 
of the operator $A$. More precisely, we estimate the asymptotic behavior of the eigenvalues of $A$ and use it to
establish the observability inequalities \eqref{GG1}-\eqref{GG2}. The proof of Theorem \ref{thm100} is outlined as follows.  
In Step 1, we introduce  the operator $B y   = - y'''(L-x)-y'(L - x)$ with domain $D(B)  =  \{ y\in H^3(0,L)\cap H^1_0(0,L);\  y'(L)=0\}$, which is 
 closely related to the operator $A$ (but more easy to handle).  
We prove that it is selfadjoint and that it has a compact resolvent, so that it can be diagonalized in an orthonormal basis in $L^2(0,L)$. 
In Step 2, we estimate the asymptotic behavior of the eigenvalues of $B$. Finally, in Step 3 we show that $A$ can be diagonalized in an orthonormal basis of $[L^2(0,L)]^2$ and use the expansions of the solutions in terms of the 
eigenfunctions to prove   
\eqref{GG1} and \eqref{GG2}. \\

\noindent
{\sc Step 1 (Study of the operator $B$)}\\
 Let 
\[
(By)(x) :=-y'''(L-x) -y'(L-x)\quad \textrm{ for } y\in D(B)=:\{ y\in  H^3(0,L)\cap H^1_0(0,L); \ y'(L)=0\},  
\]
where $'=d/dx$. Then we have the following result. \\[3mm] 
{\sc Claim A.1} $B$ is a selfadjoint operator in $L^2(0,L)$.\\
Picking any $y,z\in D(B)$, we readily obtain by integration by parts that
\[
\int_0^L [-y'''(L-x)-y'(L-x)]z(x)dx = \int_0^L y(x)[-z'''(L-x)-z'(L-x)]dx, 
\] 
i.e. $(By,z)_{L^2}=(y, Bz)_{L^2}$ where $(.,.)_{L^2}$ stands for  the scalar product in $L^2(0,L)$. This means that 
$D(B)\subset D(B^*)$ and that $B^*=B$ on $D(B)$. Conversely, pick any 
 $z\in D(B^*)$. 
Then 
\[
\vert (By,z)_{L^2}\vert = \left\vert \int_0^L [-y'''(L-x)-y'(L-x)]z(x)\, dx \right\vert \le C\Vert y\Vert_{L^2}\quad
 \forall y\in D(B). 
\]
Setting $w(x)=z(L-x)$, we arrive to $\vert  (y'''+y', w) _{L^2} \vert \le C \Vert y \Vert _{L^2} $, or equivalently for
$y\in {\mathcal D } = {\mathcal D}(0,L)$
\[
\left\vert \langle w''' +  w', y \rangle _{ {\mathcal D} ',{\mathcal D}} \right\vert  \le C \Vert y\Vert _{L^2} \cdot 
\]  
It follows that $w'''+w'\in L^2(0,L)$, and hence $w\in H^3(0,L)$ and $z\in H^3(0,L)$. Integrating by parts in 
$(By,z)_{L^2}$ and using the fact that $y'''+y' \in L^2(0,L)$, we arrive to 
 \[
 \left\vert  -y''(L)z(0) + y''(0)z(L)  + y'(0)z'(L)   \right\vert \le C' \Vert y\Vert _{L^2},\quad \forall y\in D(B).  
 \] 
 This yields $z(0)=z(L)=z'(L)=0$. Thus $z\in D(B)$, and we infer that $D(B^*)=D(B)$ and that $B^*=B$.\\
 
 \noindent
 {\sc Claim A.2}  If $L\in (0,+\infty)\setminus \cN$, then $B^{-1}:L^2(0,L)\to H^3(0,L)$ is a well-defined continuous operator. \\[3mm]
 Pick any $L\not\in \cN$  and any $z\in L^2(0,L)$. We search for $y\in D(B)$ solving the equation
 $By=z$, i.e. 
 \begin{eqnarray*}
 -y'''(L-x)-y'(L-x) &=& z(x), \quad x\in (0,L), \\
 y(0)=y(L)=y'(L)&=&0.
 \end{eqnarray*} 
 We search for the function $y$ in the form $y(x)=y_1(x)+y_2(x)$ with 
 \[
 y_1(x)=\int_x^L [1-\cos (x-s)] z(L-s)ds. 
 \]
We see at once that $y_1\in H^3(0,L)$ with $y_1''' (x)+y_1'(x)=-z(L-x)$ a.e. in $(0,L)$ and $y_1(L)=y_1'(L)=y_1''(L)=0$. 
Thus, it remains to find $y_2\in H^3(0,L)$ such that $y_2'''+y_2'=0$ and $y_2(L)=y_2'(L)=0$, $y_2(0)=-y_1(0)$.
By linearity, there is no loss of generality in assuming that $y_2(0)=1$. 
The three roots of the equation $r^3+r=0$ are $r_1=i$, $r_2=-i$ and $r_3=0$. We search for $y_2$ in the form 
\[
y_2(x) =\sum_{j=1}^3 a_j [ e^{r_jx} -ie^{r_j (L-x)}],
\] 
where the coefficients $a_1,a_2,a_3$ have still to be found. It is clear that $y_2$ solves $y_2'''+y_2'= 0$, and the boundary 
conditions $y_2(L)=y_2'(L)=0$ and $y_2(0)=1$ give the following conditions 
\ba
\sum_{j=1}^3 a_j (e^{r_jL}-i) &=& 0, \label{GG31}\\
\sum_{j=1}^3 a_j (1-i e^{r_jL}) &=& 1, \label{GG32} \\
\sum_{j=1}^3 r_ja_j (e^{r_jL} +i) &=&0. \label{GG33} 
\ea
Using  \eqref{GG31}, \eqref{GG32} and the precise values of $r_1,r_2,r_3$, we obtain  that $a_3=\frac{1}{2} -a_1-a_2$ and
\[
a_2=\frac{1}{e^{-iL}-1} \left(\frac{i-1}{2} -a_1(e^{iL}- 1)  \right) . 
\] 
(Note that $L\not\in 2\pi \Z$, for $L\not\in {\mathcal N}$.)
Plugging these expressions of $a_2$ and $a_3$ in \eqref{GG33}, we arrive to 
\[
a_1\left( i(e^{iL}+i) -\frac{e^{iL} -1}{ e^{-iL} -1} (-i) (e^{-iL} +i) \right)  + (-i) (e^{-iL} +i) \frac{i-1}{2(e^{-iL} -1) }=0. 
\]
We readily see that the coefficient behind $a_1$ is not zero, for $L\not\in 2\pi {\mathbb Z} $, so that the last equation 
for $a_1$ can be  solved. Claim A.2 is proved. \\

\noindent
{\sc Claim A.3} There is an orthonormal basis $(v_n)_{n\in \N}$ in $L^2(0,L)$ composed of eigenvectors of $B$: for all $n\in \N$,
$v_n\in D(B)$ and $Bv_n=\lambda _n v_n$ for some $\lambda _n\in\R$.\\

It is a direct consequence of Claims A.1 and A.2, since $B^{-1}$ is a bounded compact selfadjoint operator in $L^2(0,L)$. Thus 
$B^{-1}$ is diagonalizable in an orthonormal basis in $L^2(0,L)$, and the same is true for $B$. \\ 

\noindent{\sc Step 2 (Asymptotics of the eigenvalues of $B$)}\\
 
\noindent 
{\sc  Claim A.4}  Using a convenient relabeling,  the sequence of eigenvalues of $B$ can be written 
$(\lambda _n)_{n\in \Z}$, with $\lambda _n\le \lambda _{n+1}$ for all $n\in \Z$ and 
\ba
\lambda _n = \left( \frac{\frac{\pi}{6} + 2\pi (k_1+n) }{L} \right) ^3 + O(n)&\textrm{ as }& n\to +\infty ,  \label{GG41}\\
\lambda _n = -  \left( \frac{\frac{7 \pi}{6} + 2\pi (k_2-n) }{L} \right) ^3 + O(n) &\textrm{ as }& n\to -\infty   \label{GG42}
\ea
for some numbers $k_1,k_2\in\Z$. 

The eigenvalues of $B$ as given in Claim A.3 satisfy $|\lambda _n|\to +\infty$ as $n\to +\infty$. We shall show that 
they can be separated into two subsequences, one with the asymptotics \eqref{GG41} and another one
with the asymptotics \eqref{GG42}.     
 
Assume that $(v,\lambda )$ is a pair of eigenvector/eigenvalue for $B$; that is, $v\in D(B)$, $v\ne 0$,  and 
$v'''(x)+v'(x)=-\lambda v(L-x)$. This yields 
\begin{equation}
\label{GG51} 
v^{(6)} + 2v^{(4)} + v''= - \lambda ^2 v.
\end{equation}
 The roots of the equation 
 \be
 \label{GG52}
 (r^3+r)^2= - \lambda ^2
 \ee
  read $r_1, r_2,r_3, -r_1,-r_2,-r_3$, where $r_1,r_2,r_3$ denote the roots of 
\begin{equation}
\label{GG53}
r^3+r=i\lambda .
\end{equation}
Note that if the $r_j$ are not pairwise distinct, then any multiple root $r$ should also solve $3r^2+1=0$, so that 
$r=\pm i/\sqrt{3}$ and $\lambda =\pm 2/(3\sqrt{3})$. Thereafter, we assume that $|\lambda |> 2/(2\sqrt{3})$ so that 
the roots $r_j$ ($1\le j\le 3$) of \eqref{GG53} are simple. Note that the roots $\pm r_j$ ($1\le j\le 3$) of \eqref{GG52}
are also simple,  for $r_i=-r_j$ yields $i\lambda =r_i^3+r_i=-(r_j^3+r_j) =-i\lambda$ and $\lambda =0$.  
It follows that the exponential maps $e^{ \pm r_j x }$ ($1\le j\le 3$)  are not linked. From \eqref{GG51}, we can write 
\[
v(x) = \sum_{j=1}^3 [ a_j e^{r_j x} + b_j e^{-r_j x} ]
\]
for some constants $a_j, b_j\in \C$ ($1\le j\le 3$). The equation $Bv=\lambda v$ yields
\[
\sum_{j=1}^3 \left( [a_j (r_j^3 + r_j)  + b_j\lambda e^{-r_jL} ]e^{r_jx} + [ a_j \lambda  e^{r_jL} -b_j (r_j^3+r_j)]e^{-r_jx}\right)=0. 
\]
This holds if and only if $b_j=-ie^{r_jL} a_j$ for $1\le j\le 3$. Thus $v$ takes the form 
\be
\label{GG61}
v(x)=\sum_{j=1}^3 a_j [ e^{r_jx} - ie^{r_j(L-x)}]. 
\ee
Then the conditions $v(L)=0$, $v(0)=0$ and $v'(L)=0$ are equivalent to 
\ba
\sum_{j=1}^3 a_j (e^{r_jL} - i) &=& 0, \label{GG71} \\
\sum_{j=1}^3 a_j (1-i e^{r_jL})&=& 0, \label{GG72}\\
\sum_{j=1}^3 r_ja_j (e^{r_jL} + i) &=& 0. \label{GG73}  
\ea
The equations \eqref{GG71}-\eqref{GG72} yield $\sum_{j=1}^3 a_j = \sum_{j=1}^3 a_j e^{r_jL} =0$, i.e.
\begin{eqnarray}
a_3&=& - a_1- a_2, \nonumber \\
a_1( e^{ r_1L} -e^{ r_3L } ) +a_2 (e^{ r_2 L } -e^{ r_3L } ) &=&0. \label{PPPP2}
\end{eqnarray}
Substituting the values of $a_3$ and $a_2$ in \eqref{GG73} results in 
\[
a_1 
\big[ 
r_1(e^{r_1L} +i)(e^{r_2L} -e^{r_3L}) + r_2 (e^{r_2L} +i)(e^{r_3L} -e^{r_1L}) +r_3 (e^{r_3L} +i) (e^{r_1L} -e^{r_2L})
\big]  =0.   
\]
Thus if $e^{r_2L}\ne e^{r_3L}$, then $a_2$ can be expressed in terms of $a_1$, and 
the system has a solution $(a_1,a_2,a_3)\ne (0,0,0)$ if and only if 
\be
\label{GG80}
r_1(e^{r_1L} +i)(e^{r_2L} -e^{r_3L}) + r_2 (e^{r_2L} +i)(e^{r_3L} -e^{r_1L}) +r_3 (e^{r_3L} +i) (e^{r_1L} -e^{r_2L}) =0.   
\ee
The case $e^{r_2L}=e^{r_3L}$ is actually impossible for $ | \lambda |$ large enough, see below \eqref{PPPP3} and \eqref{PPPP4}.  
We shall use several times the following classical result.
\begin{theorem} (Rouch\'e's theorem, see e.g. \cite[3.42]{titchmarsh}) Let $f$ and $g$ be analytic inside and on a closed contour $\Gamma $, and such that $|f(z)-g(z)|<|f(z)|$ on $\Gamma$. Then 
$f$ and $g$ have the same number of zeros inside $\Gamma$. 
\end{theorem}
\noindent
$\bullet$ Assume that $\lambda \to +\infty$. 
Then $|r|\to +\infty$ and  from the equation $r^3 ( 1 + \frac{1}{r^2} )=i\lambda$, we infer that 
\[
r^3\sim i\lambda. 
\]
Then we can choose $r_1,r_2,r_3$ so that 
\[
r_1\sim -i\lambda ^\frac{1}{3}, \ r_2\sim -i j \lambda ^\frac{1}{3}, \ r_3\sim -i j^2 \lambda ^\frac{1}{3}
\]
where $j=e^{i\frac{2\pi}{3}}$. More precisely, for any $C>1/3$ and for $\lambda$ large enough, there is only one solution $r$ of  $r^3 ( 1 + \frac{1}{r^2} )=i\lambda$ 
in the disk $\{ r\in\C ; \ |r-(-i\lambda ^\ff)   |  <C\lambda ^{-\ff} \}$. Indeed, letting $f(r)=r^3-i\lambda$, $g(r) = r^3+r-i\lambda$, and 
considering $r=-i\lambda ^\ff + \rho$ with
$|\rho| =C\lambda ^{-\ff}$, we have that 
$| f(r)-g(r) | = |r|=\lambda ^\ff + O(\lambda ^{-\ff})$, while 
\[ 
|f(r)|
=\left\vert  (-i\lambda ^\ff +\rho)^3-i\lambda| = |3(-i\lambda ^\ff)^2\rho + 3(-i\lambda ^\ff)\rho^2 + \rho ^3 \right\vert
=3C\lambda ^\ff + O(\lambda ^{-\ff}). 
\]
The conclusion for $r_1\sim r_1^*:=-i\lambda ^\ff$ follows then from Rouch\'e's theorem. We can do exactly the same for the two other roots
$r_2\sim r_2^*:=jr_1^*$ and $r_3\sim r_3^* :=j^2r_1^*$. 

From $r^3=i\lambda (1+r^{-2})^{-1} = i\lambda (1-r^{-2} + r^{-4} +\cdots) $, we obtain that 
\begin{eqnarray*}
r_1 &=& -i\lambda ^\ff (1 + \ff  \lambda ^{-\frac{2}{3}} + \cdots ) = -i\lambda ^\ff + O(\lambda ^{-\ff}) \sim -i\lambda ^\ff, \\
r_2 &=& -ij\lambda ^\ff (1 + \ff j^{-2}  \lambda ^{-\frac{2}{3}} + \cdots ) = -i j \lambda ^\ff + O(\lambda ^{-\ff}) \sim (\frac{\sqrt{3} }{2} + \frac{i}{2}) \lambda ^\ff , \\
\textrm{ and } \quad r_3 &=& -i j^2 \lambda ^\ff (1 + \ff j^{-4} \lambda ^{-\frac{2}{3}} + \cdots ) = -i j^2\lambda ^\ff + O(\lambda ^{-\ff}) \sim (- \frac{\sqrt{3} }{2} + \frac{i}{2}) \lambda ^\ff .
\end{eqnarray*} 
Clearly, 
\be
|e^{r_1L}| \to 1, \ \ |e^{r_2L} | \to + \infty \  \textrm{ and }  \ e^{r_3L}\to 0.  
\label{PPPP3}
\ee

Plugging these expressions in \eqref{GG80}, we arrive to 
\be
\label{GGG1}
e^{r_1L} (r_1-r_2) + i(r_1-r_3) = O(\lambda ^\ff 
e^{-\frac{\sqrt{3}}{2}  \lambda   ^\ff }).
\ee
On the other, we have that 
\[
r_1 - r_2 = - i \lambda ^\ff (1-j) + O(\lambda ^{-\ff} ), \quad  r_1-r_3 = -i \lambda ^\ff (1-j^2)  + O(\lambda ^{-\ff}),
\]
and hence
\be
e^{r_1L} = e^{-i\lambda ^\ff L+ O(\lambda ^{- \ff }) }= ij^2   + O(\lambda ^{-\frac{2}{3}} ) = e^{-i\frac{\pi }{6}}  + O(\lambda ^{-\frac{2}{3}}). 
\label{GGHH1}
\ee
We infer that for $\lambda$ large enough, 
\[
\lambda  ^\ff  L= \frac{\pi}{6} + 2n\pi + O(\lambda ^{-\ff})
\]
for some $n\in \Z$, so that $\lambda _n\sim  L^{-3} \left( \frac{\pi}{6} + 2n\pi \right) ^3$.
On the other hand, we claim that there is a simple eigenvalue $\lambda _n\sim  L^{-3} \left( \frac{\pi}{6} + 2n\pi \right) ^3$
for all $n\in \N $ large enough. Indeed, we infer from $r^3+ r -i\lambda =0$ that $r_2+r_3= - r_1$ and $r_2r_3=i\lambda /r_1$, so that 
\[
r_{2,3}=\frac{-r_1 \pm (r_1^2-\frac{4i\lambda}{r_1})^\frac{1}{2}}{2}= \frac{-r_1 \pm ( -3r_1^2 - 4)^\frac{1}{2}}{2}.
\] 
Since 
\[
r_2\sim r_2^*=jr_1^*= \frac{- r_1^* +\sqrt{3} i r_1^* }{2} =
 \frac{-r_1^* +  ( -3(r_1^*)^2 )^\frac{1}{2}}{2},
\]
 we have that 
\be
r_2 =\frac{-r_1  +  ( -3r_1^2 - 4)^\frac{1}{2}}{2}, \quad r_3 =\frac{-r_1 - ( -3r_1^2 - 4)^\frac{1}{2}}{2} \cdot
\label{PPPP5}
\ee
Let 
\[
g(r_1) =  \big[ r_1(e^{r_1L} +i)(e^{r_2L} -e^{r_3L}) + r_2 (e^{r_2L} +i)(e^{r_3L} -e^{r_1L}) 
+r_3 (e^{r_3L} +i) (e^{r_1L} -e^{r_2L}) \big] e^{-r_2 L }
\]
where $r_2$ and $r_3$ are as in \eqref{PPPP5}, and let $f(r_1) = e^{r_1L} r_1(1-j) +r_1i(1-j^2) $. Then for $\lambda$ large enough, 
$f(r_1)=0$ has a solution $r_1^*=-i\lambda _n^\ff$ where 
$\lambda _n= [\frac{1}{L} (\frac{\pi}{6} + 2\pi n ) ]^3$, $n\in \Z$. On the other hand, 
$| f(r_1)-g(r_1)| =O(\lambda ^{-\ff} )  <1< | f(r_1) |  $ if $ |r_1 -r_1^*|=1/L$ and $\lambda$ is large enough. As $f(r_1^*)=0$, we infer
from Rouch\'e's theorem that $g(r_1)=0$ has only one root in the disk $\{ r_1\in\C; \  |r_1-r_1^*|<\frac{1}{L}\}$. 
Combined with \eqref{GGHH1}, this yields for all $n\in\N$ large enough an eigenvalue $\lambda _n$ such that 
\[
\lambda _n^\ff L= \frac{\pi}{6} + 2n\pi + O(n^{-1}),
\]
and hence 
\[
\lambda _n = L^{-3} \left( \frac{\pi}{6} + 2n\pi \right) ^3 + O(n),
\] 
and we see from \eqref{GG61}-\eqref{GG73} that the associated eigenspace is onedimensional.

\noindent
$\bullet$  Assume that $\lambda \to -\infty$. 
We still choose $r_1,r_2,r_3$ so that 
\[
r_1\sim -i\lambda ^\frac{1}{3}, \ r_2\sim -i j \lambda ^\frac{1}{3}, \ r_3\sim -i j^2 \lambda ^\frac{1}{3}. 
\]
Then we obtain that 
\begin{eqnarray*}
r_1 &=&   -i\lambda ^\ff + O(\lambda ^{-\ff}) \sim i |\lambda| ^\ff, \\
r_2 &=& -i j \lambda ^\ff + O(\lambda ^{-\ff}) \sim  - (\frac{\sqrt{3} }{2} + \frac{i}{2}) |\lambda | ^\ff , \\
\textrm{ and } \quad r_3 &=& -i j^2\lambda ^\ff + O(\lambda ^{-\ff}) \sim  ( \frac{\sqrt{3} }{2} - \frac{i}{2}) |\lambda |^\ff .
\end{eqnarray*} 
Clearly, 
\be
|e^{r_1L}| \to 1, \ \  |e^{r_2L} | \to 0\  \textrm{ and }  \ |e^{r_3L}| \to +\infty.
\label{PPPP4}
\ee  
Plugging these expressions in \eqref{GG80}, we arrive to 
\[
e^{r_1L} (r_3-r_1) + i(r_2-r_1) = O( |  \lambda | ^\ff 
e^{-\frac{\sqrt{3}}{2} |  \lambda | ^\ff }).
\]

On the other, we have that 
\[
r_3 - r_1 = - i \lambda ^\ff (j^2-1) + O(\lambda ^{-\ff} ), \quad  r_2-r_1 = -i \lambda ^\ff (j-1)  + O(\lambda ^{-\ff}),
\]
and hence
\[
e^{r_1L} = e^{-i\lambda ^\ff L+ O(\lambda ^{-\ff}) }= i j   + O(\lambda ^{-\ff}) = e^{i\frac{7\pi }{6}}  + O(\lambda ^{-\frac{2}{3}}). 
\]

We infer that 
\[
-  \lambda ^\ff  L= \frac{7\pi}{6} + 2n\pi + O(\lambda ^{-\ff}),
\]
for some $n\in \Z$, so that 
\be
\lambda = - L^{-3} \left( \frac{7\pi}{6} + 2n\pi \right) ^3 + O(n). 
\ee 
On the other hand, we can prove as above that for all $n\in \Z$ with $n<0$ and $ | n | $ large enough, there is indeed an eigenvalue 
\[
\lambda_n  = - L^{-3} \left( \frac{7\pi}{6} + 2 | n | \pi \right) ^3 + O(n), 
\]
and that it is simple. \\
$\bullet$ Finally, 
we can relabel the $\lambda_n$'s so that the eigenvalues of $B$ form a sequence 
$(\lambda_n)_{n\in \Z}$  as in Claim A.4. \\[3mm]

\noindent
{\sc Step 3. Diagonalization of the operator $A$}\\

We denote by $(v_n)_{n\in \Z}$ an orthonormal basis in $L^2(0,L)$ such that $Bv_n = \lambda _n v_n$ for all $n\in\Z$.  
Let for $n\in \Z$ 
\[ \theta _n^+ (x):=-\frac{i}{\sqrt{2}} v_n(L-x), \ \theta _n^- (x) := \frac{i}{\sqrt{2}}  v_n (L-x), \ \textrm{ and } u_n^+(x)= u_n^-(x) := \frac{1}{\sqrt{2}} v_n(x). \]
Then $(\theta_n ^+, u_n^+)_{n\in \Z} \cup (\theta _n ^- , u_n^-)_{n\in \Z}$ is an orthonormal basis in $L_\C^2(0,L)\times L_\C^2(0,L)$ (endowed with the natural
scalar product 
$( (\theta ,u), (\varphi , w)) = \int_0^L  [ \theta(x)\overline{\varphi (x)} + u(x)\overline{ w(x)}] dx$) 
composed of eigenvectors of $A$:
\[
A(\theta _n^+, u_n^+)=i\lambda _n (\theta _n^+, u_n^+), \quad A(\theta _n^-, u_n^-)=(-i\lambda _n) (\theta _n^-, u_n^-). 
\]  
Let $(\theta ^0, u^0)\in [L^2_\C (0,L)]^2$ be  decomposed as 
\[
\theta ^0(x)=\sum_{n\in \Z} a_n v_n (L-x), \quad u^0(x)= \sum_{n\in \Z} b_n v_n(x). 
\]
(Note that $(v_n(L- \cdot ))_{n\in \Z}$ is also an orthonormal basis in $L^2(0,L)$.)
Then we can write 
\[
(\theta ^0,u^0)= \sum_{n\in \Z} [c_n^+ (\theta_n^+ ,  u_n^+) + c_n^-  (\theta _n^-, u_n^-) ]
\]
where 
\[
c_n^+=\frac{1}{\sqrt{2} } (i a_n+b_n), \quad c_n^- = \frac{1}{\sqrt{2} } (b_n-ia_n). 
\]
It follows that the solution $(\theta, u)$ of \eqref{A51} can be written as
\be
\label{GGG11}
(\theta , u ) = \sum_{n\in \Z} [c_n^+ e^{ i \lambda _nt} (\theta _n^+, u_n^+) + c_n^-  e^{- i \lambda _nt} (\theta _n^-, u_n^-)]. 
\ee
Let us show that there is an asymptotic spectral gap for the spectrum.  \\
It follows from \eqref{GG41} and \eqref{GG42} that 
\begin{eqnarray}
\lambda _{n+1}-\lambda  _n &\sim& \frac{24\pi ^3}{L^3} n^2 \quad \textrm{ as } n \to\infty , \label{GH3} \\
\lambda _n-\lambda  _{n-1} &\sim& \frac{24\pi ^3}{L^3} n^2 \quad \textrm{ as } n \to - \infty . \label{GH4} 
\end{eqnarray}
On the other hand, for all $k\in\Z$ and $n\in \N$, we have that 
\[
\lambda _n -(-\lambda_{k-n}) =(a-b)(a^2+ab+b^2) + O(n)
\]
where $a=(\frac{\pi}{6} + 2\pi (k_1+n))/L$, $b=(\frac{7\pi}{6} + 2\pi (  k_2-k+n))/L$.
Then 
\[ |a-b| =\left\vert \frac{2\pi(k_1-k_2+k)-\pi}{L} \right\vert \ge \frac{\pi}{L}
\]
and 
\[
a^2+ab+b^2\ge \frac{1}{2}(a^2+b^2). 
\]
It follows that there is some $n_0\in \N^*$ such that for any $k\in \Z$ and some $C\in \R_+$, 
\[
\vert \lambda _n -(-\lambda_{k-n}) \vert = \frac{\pi}{2L^3}
\left\vert \frac{\pi}{6} + 2\pi (k_1+n) \right\vert^2  
 + O(n) \ge C n^2 \quad \textrm{ for } n \ge n_0. 
\]
Combined with  \eqref{GH3}-\eqref{GH4}, we infer that for all $A>0$ we can find some $n_1\ge n_0$ such that
we have the spectral gap relation
\begin{eqnarray*}
&&|(i\lambda_n)-(i\lambda _m)| \ge A,\quad \forall m\ne n\  \textrm{ with } \ |m|\ge n_1,\ |n|\ge n_1, \\
&&|(i\lambda _n)-(-i\lambda _m)| \ge A, \quad \forall m,n\ \textrm{ with }\  |m|\ge n_1,\ |n|\ge n_1. 
\end{eqnarray*}
Pick any $T>0$. It follows then from Ingham's lemma that there exist an integer $N\in \N$ and a constant
$K>0$ such that for all $(c_n^+)_{n \in \Z}, (c_n^-)_{n \in \Z}\in l^2(\Z)$, we have 
\[
K^{-1} \sum_{|n|\ge N}  \big( |c_n^+|^2 + |c_n^-|^2\big)  \le \int_0^T | \sum_{|n|\ge N}
\big(  c_n^+ e^{i\lambda _nt} + c_n^- e^{-i\lambda _nt} \big)  \vert ^2 dt \le   
K\sum_{|n|\ge N}  \big(  |c_n^+|^2 + |c_n^-|^2\big) . 
\]
Next, we compare $|v_n''(0)|$ to  $|v_n''(L)|$ as $ | n | \to +\infty$. \\

\noindent
{\sc Claim A.5} There are two numbers $K_+,K_-\in \C ^*$ such that 
$v_n''(0)  \sim K_+ v_n''(L) $ as $n\to +\infty$ and $v_n''(0) \sim  K_- v_n''(L)$ as $n\to - \infty$. \\
It follows from \eqref{GG61} that 
\[
v_n''(0)= \sum_{j=1}^3 a_{j}r_{j}^2(1-ie^{r_{j} L}), \quad v_n''(L)=\sum_{j=1}^3 a_{j}r_{j}^2(e^{r_{j}L}-i) 
\]
where $r_j=r_j(n)$ and $a_j=a_j(n)$. 
Assume first that $n\to +\infty$. Then 
\[ a_3=-a_1 -a_2,\quad a_1(e^{r_1L} -e^{r_3L}) + a_2 (e^{r_2L} -e^{r_3L})=0. \]
It follows that 
\begin{eqnarray*}
v_n''(0)&=& a_1r_1^2(1-ie^{r_1L})+a_2r_2^2 (1-ie^{r_2L}) + a_3r_3^2 (1-ie^{r_3L}) \\
&=& a_1\bigg(  r_1^2 (1-ie^{r_1L}) +r_2^2 ie^{r_1L} -r_3^2  +
 O(\lambda _n ^\frac{2}{3} e^{-\frac{\sqrt{3}}{2} \lambda _n ^\frac{1}{3} } ) \bigg) \\
 &=& a_1 \big( (1-j^4) (-\lambda _n^\frac{2}{3} ) +ie^{r_1L} (j^2-1) (-\lambda _n ^\frac{2}{3})  + O(1)\big)\\
 &=& a_1(1-j)(-\lambda _n^\frac{2}{3})(1+j^2ie^{r_1L} + O(\lambda _n ^{-\frac{2}{3}})).
\end{eqnarray*}
But $e^{r_1L}\sim e^{-i\lambda _n ^\ff L} \sim e^{ -i\frac{\pi}{6}}=  i j^2$, and hence
\[
v_n''(0)\sim a_1 (1-j)^2 (-\lambda _n^\frac{2}{3}).
\]
Similarly, we obtain 
\begin{eqnarray*}
v_n''(L)&=& a_1r_1^2( e^{r_1L} -i )+a_2r_2^2 (e^{r_2L} -i ) + a_3r_3^2 (e^{r_3L} -i ) \\
&=& a_1\bigg(  r_1^2 (e^{r_1L} -i ) -r_2^2 e^{r_1L}  + i r_3^2  +
 O(\lambda  _n^\frac{2}{3} e^{-\frac{\sqrt{3}}{2} \lambda _n ^\frac{1}{3} } ) \bigg) \\
 &=& a_1 \big( e^{r_1L}  (1-j^2 ) (-\lambda _n^\frac{2}{3} ) +i (j-1) (-\lambda _n^\frac{2}{3})  + O(1)\big)\\
 &=& a_1(1-j)(-\lambda _n^\frac{2}{3})(-j^2 e^{r_1L} -i + O(\lambda _n^{-\frac{2}{3}} )),
\end{eqnarray*}
and hence
\[
v_n''(L)\sim a_1 (1-j) \, i j^2 (-\lambda _n^\frac{2}{3}).
\]
We conclude that $v_n''(0)\sim \frac{1-j}{ij^2} v_n''(L)$ as $n\to +\infty$. We can prove in the same way 
that $v_n''(0)\sim K_- v_n''(L)$ as $n\to -\infty$ for some constant $K_-\ne 0$. \\[3mm]

We are in a position to complete the proof of Theorem \ref{thm100}. Pick any $(\theta ^0, u^0)\in X_2$. From \eqref{GGG11}, we have that 
\begin{eqnarray*}
\theta _{xx}(t,L) &=& \sum _{n\in \Z} [c_n^+ e^{ i \lambda _nt} \theta _{n,xx}^+(L)
+ c_n^-  e^{- i \lambda _nt} \theta _{n,xx}^-(L) ] \\
&=&\frac{i}{\sqrt{2}} \sum_{n\in \Z} [-c_n^+ e^{ i \lambda _nt} v_n''(0)
+ c_n^-  e^{- i \lambda _nt} v_n''(0) ], \\
\theta _{xx}(t,0) &=& \frac{i}{\sqrt{2}} \sum_{b\in \Z} [-c_n^+ e^{ i \lambda _nt} v_n''(L)
+ c_n^-  e^{- i \lambda _nt} v_n''(L) ], \\
u_{xx}(t,L) &=& \frac{1}{\sqrt{2}} \sum_{n\in \Z} [c_n^+ e^{ i \lambda _nt} v_n''(L)
+ c_n^-  e^{- i \lambda _nt} v_n''(L) ], \\
\textrm{ and } \quad u_{xx}(t,0) &=& \frac{1}{\sqrt{2}}  \sum_{n\in \Z}  [c_n^+ e^{ i \lambda _nt} v_n''(0)
+ c_n^-  e^{- i \lambda _nt} v_n'' (0) ].
\end{eqnarray*}
Then by \eqref{C51}, Claim A.4 and Claim A.5, we have that  (with a constant $C$ that may change
from line to line)
\begin{eqnarray*}
\Vert  (\theta ^0, u^0)\Vert ^2_{X_2} 
&\le&  C \int_0^T [ | \theta _{xx} (t,L) |^2 + |u_{xx} (t,L )|^2 ] \, dt \\
&\le& C \left( \int_0^T \left\vert \sum_{n\in \Z} [-c_n^+ e^{ i \lambda _nt} v_n''(0)
+ c_n^-  e^{- i \lambda _nt} v_n''(0) ] \right\vert ^2 dt \right. \\
&& \left. \quad + \int_0^T \left\vert  \sum_{n \in \Z} [c_n^+ e^{ i \lambda _nt} v_n'' (L)
+ c_n^-  e^{- i \lambda _nt} v_n'' (L) ] \right\vert ^2 dt \right)\\ 
&\le& C\left( \Vert (\theta ^0, u^0)\Vert ^2_{X_0} + 
 \sum_{ |n| \ge N}(|c_n^+|^2 + |c_n^-|^2) ( |v_n''(0)|^2 + |v_n''(L)|^2)   \right) \\
 &\le& C\left( \Vert (\theta ^0, u^0)\Vert ^2_{X_0} + 
 \sum_{ |n| \ge N}(|c_n^+|^2 + |c_n^-|^2)  |v_n''(0)|^2  \right) \\
&\le& C\left( \Vert (\theta ^0, u^0)\Vert ^2_{X_0} + 
\int_0^T | \theta _{xx} (t,L) |^2  \, dt \right) . 
\end{eqnarray*}
Thus \eqref{GG1} is proved. The proof of \eqref{GG2} is similar, and therefore it is omitted.
The proof of Theorem \ref{thm100} is complete. \qed  
\section*{Acknowledgments}
RC was supported by CNPq and Capes (Brazil) via a Fellowship: ``Ci\^ encia
sem Fronteiras'' and Agence Nationale de la Recherche (ANR), Project CISIFS,
grant ANR-09-BLAN-0213-02. Part of this work was done during the preparation of the PhD
thesis of RC at Universit\'e de Lorraine and UFRJ. RC thanks both host
institutions for their warm hospitality.
AFP was partially supported by CNPq (Brazil) and the International Cooperation 
Agreement Brazil-France.
LR was partially supported by the ANR project Finite4SoS (ANR-15-CE23-0007) and the project 
ICoPS Math-AmSud.

\end{document}